\documentclass[10pt,reqno]{amsart}

\usepackage[dvipsnames]{xcolor}
\usepackage{amsmath,amssymb,amsthm,amsfonts,enumerate,tikz,bm,mathtools}
\usepackage{mathrsfs}
\usetikzlibrary{matrix,arrows,positioning,calc}
\usepackage[all]{xy}
\usepackage[bookmarksnumbered,colorlinks]{hyperref}
\usepackage{url}
\numberwithin{equation}{section}
\usepackage{graphicx}
\usepackage[T1]{fontenc}
\usepackage{textcomp}
\usepackage{palatino,helvet}

\usepackage{footmisc}

\newtheorem{definition}{Definition}[section]
\newtheorem{remark}[definition]{Remark}
\newtheorem{example}[definition]{Example}

\newtheorem{theorem}[definition]{Theorem}
\newtheorem{proposition}[definition]{Proposition}
\newtheorem{lemma}[definition]{Lemma}
\newtheorem{corollary}[definition]{Corollary}
\theoremstyle{remark}

\usepackage{paralist}

\usepackage{scrextend}
\deffootnote{2em}{0em}{\thefootnotemark\quad}

\newcommand{\Ext}{\mathrm{Ext}}

\newcommand{\Hom}{\mathrm{Hom}}

\newcommand{\A}{\mathcal{A}}
\newcommand{\B}{\mathcal{B}}

\newcommand{\C}{\mathfrak{C}}

\newcommand{\X}{\mathcal{X}}
\newcommand{\Y}{\mathcal{Y}}

\newcommand{\pd}{\mathrm{pd}}

\newcommand{\id}{\mathrm{id}}

\newcommand{\resdim}{\mathrm{resdim}}

\newcommand{\GF}{\mathcal{GF}}

\newcommand{\fd}{\mathrm{fd}}

\newcommand{\Ch}{\mathrm{Ch}}

\makeatletter
\def\@seccntformat#1{%
  \protect\textup{\protect\@secnumfont
    \ifnum\pdfstrcmp{section}{#1}=0 \scshape\bfseries\fi% section # in \scshape and \bfseries
    \ifnum\pdfstrcmp{subsection}{#1}=0 \bfseries\fi% subsection # in \bfseries
    \csname the#1\endcsname
    \protect\@secnumpunct
  }%
}
\makeatother

\begin{document}

\title[Complete Frobenius pairs]{Cotorsion pairs and model structures induced by resolution dimensions relative to Frobenius pairs}
\thanks{2020 MSC: 18G10; 18G20; 18G25; 16E10; 16E65.}
\thanks{Key Words: Frobenius pair, cotorsion pair, resolution dimension, abelian and exact model structures}

\author{V\'ictor Becerril}
\address[V. Becerril]{Centro de Ciencias Matem\'aticas. Universidad Nacional Aut\'onoma de M\'exico. 
 CP58089. Morelia, Michoac\'an, M\'EXICO}
\email{victorbecerril@matmor.unam.mx}

\author{Marco A. P\'erez}
\address[M. A. P\'erez]{Instituto de Matem\'atica y Estad\'istica ``Prof. Ing. Rafael Laguardia''. Facultad de Ingenier\'ia. Universidad de la Rep\'ublica. CP11300. Montevideo, URUGUAY}
\email{mperez@fing.edu.uy}

\maketitle

\begin{abstract}
We explore some methods to obtain cotorsion pairs and  model category structures from a Frobenius pair. Specifically, we consider a type of Frobenius pair $(\X,\omega)$ in an abelian category $\C$ for which there exists a nonnegative integer $m$ such that every object in the right orthogonal complement $\X^{\perp_1}$ of $\X$ under $\Ext^1_{\C}(-,\sim)$ has a special $\omega^\wedge_m$-precover and a special $[\omega^\wedge_m]^{\perp_1}$-preenvelope, where $\omega^\wedge_m$ denotes the class of objects in $\C$ with resolution dimension relative to $\omega$ at most $m$. The cotorsion pairs and model structures obtained from these Frobenius pairs will involve the classes $\X$, $\omega$, $\X^\wedge_m$, $\omega^\wedge_m$, and their orthogonal complements, in the description of its cofibrant, fibrant and trivial objects. As applications of our results, we obtain several cotorsion pairs and model structures related to relative and absolute Gorenstein homological dimensions.    
\end{abstract}

\setcounter{tocdepth}{1}
\tableofcontents

\pagestyle{myheadings}
\markboth{\rightline {\scriptsize V. Becerril and M. A. P\'{e}rez}}
         {\leftline{\scriptsize }}

%%%%%%%%%%%%%%%%%%%%%%%%%%%%%%%%%%%%%
%%%%%%%%%%%%%%%%%%%%%%%%%%%%%%%%%%%%%
%%%%%%%%%%%%%%%%%%%%%%%%%%%%%%%%%%%%%
%%%%%%%%%%%%%%%%%%%%%%%%%%%%%%%%%%%%%

\section*{\textbf{Introduction}}

In their seminal work \cite{AB89}, Auslander and Buchweitz systematically established the basis of what is currently known as  Auslander-Buchweitz approximation theory, having many important applications in a variety of contexts, such as cotilting theory ant $t$-structures. Motivated by their work, the authors of the present paper, with the collaboration of Mendoza and Santiago, proposed in the paper \cite{BMSP} the concept of left Frobenius pairs in abelian categories as a way to induce approximations, cotorsion pairs and exact model category structures from the two classes of objects forming the Frobenius pair. More precisely, given a left Frobenius pair $(\mathcal{X},\omega)$ in an abelian category $\mathfrak{C}$, under some additional assumptions (namely, that $(\mathcal{X},\omega)$ is strong), it is possible to obtain an exact projective model structure (in the sense of Gillespie's \cite[Def. 3.1]{GillespieExact}) on the (weakly idempotent complete exact) subcategory $\X^\wedge$ formed by the objects in $\C$ with finite resolution dimension relative to $\X$, where $\mathcal{X}$ and $\omega^\wedge$ are the classes of cofibrant and trivial objects, respectively (see \cite[Thm. 4.1]{BMSP} for details). On the other hand, left Frobenius pairs which are not necessarily strong are also useful to model general properties appearing in hereditary and complete cotorsion pairs, Gorenstein modules, and relativizations of them (such as Ding projective and Gorenstein AC-projective modules, projectively coresolved Gorenstein flat modules and Gorenstein flat modules relative to semi-definable classes, for instance). These properties are mostly related to resolution dimensions relative to $\X$ and $\omega$, and the construction of left or right approximations by these classes. 

Returning to the previous point on homotopical aspects of strong left Frobenius pairs, although several model structures in relative Gorenstein homological algebra can be obtained by means of such pairs, there are other models  related to certain Gorenstein modules which are not captured by the methods shown in \cite[Thm. 4.1]{BMSP}. As an example, the classes ${}_R\mathcal{GF}$ and ${}_R\mathcal{F} \cap {}_R\mathcal{C}$ of Gorenstein flat and flat-cotorsion left $R$-modules form a left Frobenius pair $({}_R\mathcal{GF},{}_R\mathcal{F} \cap {}_R\mathcal{C})$, which need not be strong (as we shall show below), and some model structures having for instance ${}_R\mathcal{GF}$ as the class of cofibrant objects cannot be deduced from \cite[Thm. 4.1]{BMSP}, like the one represented by the triple $({}_R\mathcal{GF},[{}_R\mathcal{PGF}]^\perp,{}_R\mathcal{C})$ mentioned in the comments right after \cite[Coroll. 4.12]{Saroch}. Recently in the works of Gao, Lu, and Zhang \cite{Chains}, and of El Maaouy \cite{Rachid}, several model structures were obtained from the classes of modules with finite Gorenstein projective dimension and of modules with finite Gorenstein flat dimension. The Hovey triples representing these model structures have as cofibrant objects the classes ${}_R\mathcal{GP}^\wedge_m$, ${}_R\mathcal{GF}^\wedge_m$ and ${}_R\mathcal{PGF}^\wedge_m$ of modules with Gorenstein projective, Gorenstein flat, and projectively coresolved Gorenstein flat dimension $\leq m$, respectively. 

One of the main purposes of the present paper is to describe some general patterns appearing in the construction of the model structures mentioned in the previos paragraph, introducing a new type of left Frobenius pairs, which enables us to provide a method for constructing further exact model structures without needing the strong assumption for the given pair. Namely, these will be left Frobenius pairs $(\X,\omega)$ in an abelian category $\C$ which we call \emph{$\omega^\wedge_m$-complete in $\X^{\perp_1}$}. For such pairs, for every object $M \in \X^{\perp_1}$ there exist short exact sequences of the form $K' \rightarrowtail L \twoheadrightarrow M$ and $M \rightarrowtail K \twoheadrightarrow L'$ with $L, L' \in \omega^\wedge_m$ and $K, K' \in [\omega^\wedge_m]^{\perp_1}$, where $\omega^\wedge_m$ is the class of objects in $\C$ with resolution dimension relative to $\omega$ at most $m$. From certain left Frobenius pairs $(\X,\omega)$ equipped with this completeness condition, and induced by a hereditary and complete cotorsion pair $(\X,\X^\perp)$, it will be possible to obtain abelian or exact model category structures having $\X^\wedge_m$, the class of objects in $\C$ with resolution dimension relative to $\X$ at most $m$, as the class of cofibrant objects. This allows us in particular to recover many of the model structures found in \cite{Chains,Rachid}, but also to produce new ones related to other classes of relative Gorenstein projective and relative Gorenstein flat modules and chain complexes. 

The other main goal of this paper is to present the general properties and some applications of left Frobenius pairs $(\X,\omega)$ in $\C$ which are $\omega^\wedge_m$-complete in $\X^{\perp_1}$. For example, with respect to these pairs, we shall be able to obtain right approximations for objects in $\X^\wedge$ by the class $\X^\wedge_m$ for some $m \in \mathbb{Z}_{\geq 0}$. More specifically, any left Frobenius pair $(\X,\omega)$ in $\C$, which is $\omega^\wedge_m$-complete in $\X^{\perp_1}$ for some $m \in \mathbb{Z}_{\geq 0}$, gives rise to a cotorsion pair $(\X^\wedge_m,[\omega^\wedge_m]^{\perp_1} \cap \X^{\perp_1})$ cut along $\X^\wedge$ (in the sense of \cite{HMP-cut}) and to a left Frobenius pair $(\X^\wedge_m, \omega^\wedge_m \cap [\omega^\wedge_m]^{\perp_1})$ in $\C$. In the case where the left Frobenius pair $(\X,\omega)$ is induced by a hereditary and complete cotorsion pair $(\X,\X^\perp)$ in $\C$, its $\omega^\wedge_m$-completeness in $\X^\perp$ characterizes when $\X^\wedge_m$ is the left half of a hereditary and complete cotorsion pair.

%%%%%%%%%%%%%%%%%%%%%%%%%%%%%%%%%%%%%
%%%%%%%%%%%%%%%%%%%%%%%%%%%%%%%%%%%%%

\subsection*{Organization of the paper}

Starting with Section \ref{sec:prelims}, we recall the necessary preliminaries on homological dimensions, (co)resolution dimensions, orthogonal complements, cotorsion pairs, approximations, and model structures.  

In Section \ref{sec:completeness}, we recall the concept of left Frobenius pairs, define left Frobenius pairs $(\X,\omega)$ in an abelian category $\C$ which are $\omega^\wedge_m$-complete in $\X^{\perp_1}$ for some $m \in \mathbb{Z}_{\geq 0}$, and put on display some basic examples of this notion involving Gorenstein projective and Gorenstein flat modules. More elaborated examples will be presented in the upcoming sections, but we point out that in Example \ref{ExGF} (2) we show that the class ${}_R\mathcal{GF}$ of Gorenstein flat left $R$-modules and the class ${}_R\mathcal{F} \cap {}_R\mathcal{C}$ of flat-cotorsion left $R$-modules form a left Frobenius pair which is $({}_R\mathcal{F} \cap {}_R\mathcal{C})^\wedge_m$-complete for every $m \in \mathbb{Z}_{\geq 0}$, but which is not strong in general.

Several properties of left Frobenius pairs with the completeness condition introduced in Section \ref{sec:completeness} are stated and proved in Section \ref{sec:induced_cotorsion}. These properties will be related to conditions so that the class $\X^\wedge_m$ is the left half of a hereditary and complete cotorsion pair or of a left Frobenius pair. Our first technical Lemma \ref{Lema01} describes $\X^\wedge_m$ as the intersection of $\X^\wedge$ with the left Ext-orthogonal complement ${}^{\perp_1}([\omega^\wedge_m]^{\perp_1} \cap \X^{\perp_1})$. This helps us to prove the main result in this section, namely Theorem \ref{Family}, where we show that every left Frobenius pair $(\X,\omega)$ in $\C$ which is $\omega^\wedge_m$-complete in $\X^{\perp_1}$ gives rise to a cotorsion pair $(\X^\wedge_m, [\omega^{\wedge}_m]^{\perp_1} \cap \X^{\perp_1})$ cut along $\X^{\wedge}$. In the case where $(\X,\omega)$ is induced by a hereditary and complete cotorsion pair $(\X,\X^\perp)$, meaning that $\omega = \X \cap \X^\perp$, then the pair $(\X^\wedge_m, [\omega^{\wedge}_m]^{\perp_1} \cap \X^{\perp_1})$ is actually a hereditary and complete cotorsion pair in $\C$. Complete left Frobenius pairs are also sources of new left Frobenius pairs involving resolution dimensions, as shown in Corollary \ref{coro:inducedXmFP}. As applications of these results, we obtain (1) some previously known hereditary and complete cotorsion pairs of the form $(\X^\wedge_m,[\X^\wedge_m]^{\perp})$, where $\X$ is the class of Gorenstein projective left $R$-modules over an Artin algebra, or the class of (projectively coresolved) Gorenstein flat left $R$-modules over any ring; and (2) some new hereditary and complete cotorsion pairs, also of the previous form, taking for instance $\X$ as the class of Ding projective or Gorenstein AC-projective left $R$-modules. 

Homotopical aspects for complete left Frobenius pairs are investigated in Section \ref{sec:model_structures}. Concretely, we provide four methods to obtain hereditary abelian or exact model category structures from a left Frobenius pair $(\X,\omega)$, induced by a hereditary and complete cotorsion pair $(\X,\X^\perp)$ in an abelian category $\C$, which is $\omega^\wedge_m$-complete. In the first three methods described in Theorems \ref{HoveyinC}, \ref{thm:models_two_bounds} and \ref{thm:models_two_bounds-2}, we shall require the existence of some $m \in \mathbb{Z}_{\geq 0}$ such that $\omega^\wedge_m$ is the left half of a complete cotorsion pair in $\C$. We point out that in Theorem \ref{HoveyinC}, assuming that $\C$ has enough projective objects, the thickness of the class $\X^\perp$ is equivalent to requiring that $\omega$ coincides with the class of projective objects (or equivalently, that the auxiliary pair $(\omega^\wedge_m,[\omega^\wedge_m]^{\perp_1})$ is hereditary). In this situation, we have the Hovey triple $(\X^\wedge_m, \X^\perp, [\omega^\wedge_m]^\perp)$. In Theorem \ref{thm:models_two_bounds}, on the other hand, we obtain the Hovey triple $(\X^\wedge_m, \omega^\wedge_n, [\omega^\wedge_m]^{\perp_1} \cap \X^\wedge_n)$, with $n \geq m$, in the exact subcategory $\mathcal{X}^\wedge_n$. With respect to the fourth method presented in Theorem \ref{thm:model_two_cotorsion_pairs}, we drop the assumption that $\omega^\wedge_m$ is the left half of a complete cotorsion pair in $\C$. Instead, we require two hereditary and complete cotorsion pairs $(\X,\X^\perp)$ and $(\Y,\Y^\perp)$ having the same kernel, that is, $\X \cap \X^\perp = \Y \cap \Y^\perp$, for which there exists $m \in \mathbb{Z}_{\geq 0}$ such that the induced left Frobenius pairs $(\X,\X \cap \X^\perp)$ and $(\Y, \Y \cap \Y^\perp)$ are $(\X \cap \X^\perp)^\wedge_m$-complete. Under these assumptions, we obtain the Hovey triple $
(\mathcal{Y}^\wedge_m, \mathcal{X}^\wedge_n, [\mathcal{X}^\wedge_m]^{\perp_1} \cap \mathcal{Y}^\wedge_n)$ for every $n \geq m$ in the exact subcategory $\mathcal{Y}^\wedge_n$. As mentioned previously, these methods capture general patterns that we have noticed in the construction of many of the abelian and exact Gorenstein model structures presented in \cite{Rachid,Chains,WeiWangLiu-DingModels}. Then, in particular we obtain as examples several of these model structures, but also we are able to obtain new ones thanks to the general approach of the previous four theorems (see Examples \ref{ex:Gorenstein-models-Rmod} (2), (3) and (4), \ref{ex:pairs_in_Ch}, \ref{ex:method2} (2), \ref{ex:Ch-models-containments} (2), and \ref{ex:method4} (2)). 

%We point out that the situation to obtain model structures is particularly different in the case of Gorenstein flat modules for which we present a different method to get model structures.

Finally, in Section \ref{sec:applications}, we provide more applications, like some ways to induce several instances of complete left Frobenius pairs in the category of chain complexes from a complete left Frobenius pair of modules. In the setting of module-valued representations over left rooted quivers, we consider again in Proposition \ref{coro:Rep_omega_complete} left Frobenius pairs $(\X,\omega)$ of $R$-modules such that $\omega^\wedge_m$ is the left half of a complete cotorsion pair for some $m \in \mathbb{Z}_{\geq 0}$. This allows us to obtain a left Frobenius pair in the category of representations. We close this section showing how left Frobenius pairs and their completeness are reflected in adjoint pairs of exact functors.

%%%%%%%%%%%%%%%%%%%%%%%%%%%%%%%%%%%%%
%%%%%%%%%%%%%%%%%%%%%%%%%%%%%%%%%%%%%
%%%%%%%%%%%%%%%%%%%%%%%%%%%%%%%%%%%%%
%%%%%%%%%%%%%%%%%%%%%%%%%%%%%%%%%%%%%

\section{Preliminaries} \label{sec:prelims}

In what follows, $\mathbb{Z}$ will denote the set of integers, and $\mathbb{Z}_{\geq 0}$ and $\mathbb{Z}_{> 0}$ the subsets of nonnegative and positive integers. 

\subsection*{Settings}

Throughout this paper, $\mathfrak{C}$ will denote an abelian category (not necessarily with enough projective and injective objects). The main examples of abelian categories considered here will by the categories ${}_R\mathsf{Mod}$ and $\mathsf{Mod}_R$ of left and right (unital) modules over an associative ring $R$ with identity, the categories $\mathsf{Ch}(\mathfrak{C})$ and $\mathsf{Rep}(Q,\mathfrak{C})$ of chain complexes of objects in $\mathfrak{C}$ and of $\mathfrak{C}$-valued representations of a quiver $Q$, and the category $\mathfrak{Qcoh}(X)$ of quasi-coherent sheaves over a scheme $X$. 
In the case where $\mathfrak{C} = {}_R\mathsf{Mod}$ or $\mathfrak{C} = \mathsf{Mod}_R$, the corresponding category of complexes will be denoted by ${}_R\mathsf{Ch}$ or $\mathsf{Ch}_R$, respectively. Given $X_\bullet = (X_m,\partial_m)_{m \in \mathbb{Z}} \in \Ch(\mathfrak{C})$, we denote by $Z_m(X_\bullet)$ the kernel of the differential map $\partial_m \colon X_m \to X_{m-1}$.  
 
Concerning functors defined on $R$-modules, $\Ext ^i _R(-,\sim)$ denotes the right $i$-th derived functor of $\Hom_R(-,\sim)$, where $i \in \mathbb{Z}_{>0}$. If $M \in \mathsf{Mod}_R$ and $N \in {}_R\mathsf{Mod}$, $M \otimes_R N$ denotes the tensor product of $M$ and $N$ over $R$. Extension functors can also be defined in the setting of abelian categories $\mathfrak{C}$ using the Baer-Yoneda description, where we shall use the notation $\Ext^i_{\mathfrak{C}}(M,N)$ for the group of $i$-fold extensions of $M, N \in \mathfrak{C}$. If we let $i = 0$, then $\Ext^0_{\mathfrak{C}}(M,N)$ will be regarded as the abelian group $\Hom_{\mathfrak{C}}(M,N)$ of morphisms from $M$ to $N$. 

Monomorphisms and epimorphisms in $\mathfrak{C}$ will be usually denoted by $\rightarrowtail$ and $\twoheadrightarrow$, respectively.

%%%%%%%%%%%%%%%%%%%%%%%%%%%%%%%%%%%%%
%%%%%%%%%%%%%%%%%%%%%%%%%%%%%%%%%%%%%

\subsection*{Orthogonal classes and homological dimensions}

Let $\mathcal{X} \subseteq \mathfrak{C}$ and $M \in \mathfrak{C}$. For each $i \in \mathbb{Z}_{> 0}$, let
\[
\X^{\perp_i} := \{ N \in \mathfrak{C}: \Ext^i_{\mathfrak{C}}(-,N) |_{\X} = 0 \} \mbox{ and } \X^{\perp} := \bigcap_{i \in \mathbb{Z}_{> 0}} \X^{\perp_i}
\]
denote the $i$-th and total \emph{right orthogonal complements} of $\mathcal{X}$. Dually, we have the $i$-th and total \emph{left orthogonal complements} ${}^{\perp_i}\X$ and ${}^{\perp}\X$. 
  
The \emph{relative projective dimension} of $M$ with respect to $\X$ is defined as the value
\[
\pd_{\X}(M) := \min \{n \in \mathbb{Z}_{\geq 0} \ : \ \Ext^i_{\mathfrak{C}}(M,-)|_{\X} = 0 \mbox{ for all } i > n\},
\]
where $\pd_{\X}(M) = \infty$ if no such an integer $n$ exists. Dually, we denote by $\id_{\X}(M)$ the \emph{relative injective dimension} of $M$ with respect to $\X$. Furthermore, given a class of objects $\Y \subseteq \mathfrak{C}$, we set  
\[
\pd_{\X}(\Y) := \sup\{ \pd_{\X}(Y) \ : \ Y \in \Y \} \ \ \mbox{ and } \ \ \id_{\X}(\Y) := \sup \{ \id_{\X}(Y) \ : \ Y \in \Y \}. 
\]
If $\X = \mathfrak{C}$, we shall simply write $\pd(M)$, $\id(M)$, $\pd(\Y)$ and $\id(\Y)$ for the (absolute) projective and injective dimensions of $M$ and $\Y$, respectively. 

Regarding the $i$-th derived functors of the tensor product $- \otimes_R \sim$, for $M \in {}_R\mathsf{Mod}$ or $M \in \mathsf{Mod}_R$, we denote the \emph{flat} or \emph{weak dimension} of $M$ by $\fd(M)$.

%%%%%%%%%%%%%%%%%%%%%%%%%%%%%%%%%%%%%
%%%%%%%%%%%%%%%%%%%%%%%%%%%%%%%%%%%%%

\subsection*{Relative resolution dimension}

Again, let $\mathcal{X} \subseteq \mathfrak{C}$ and $M \in \mathfrak{C}$. By an \emph{$\X$-resolution} of $M$ we mean an exact complex 
\[
\cdots \to X_n \to X_{n-1} \to \cdots \to X_1 \to X_0 \twoheadrightarrow M
\] 
where $X_i \in \X$ for every $i \in \mathbb{Z}_{\geq 0}$. The $i$-th \emph{$\X$-syzygy} in the previous resolution is the kernel of the map $X_{i-1} \to X_{i-2}$, where $X_{-1} = M$ and $X_{-2} = 0$. The \emph{$\X$-resolution dimension} of $M$, denoted  by $\resdim_{\X}(M)$, is the smallest $n \in \mathbb{Z}_{\geq 0}$ such that there is an $\X$-resolution where $X_i = 0$ for every $i > n$. If such an $n$ does not exist, we set $\resdim_{\X}(M) := \infty$. Also, we set the notations
\[
\X^{\wedge}_n := \{ M \in \mathfrak{C} \ : \ \resdim_{\mathcal{X}}(M) \leq n \} \ \ \text{ and } \ \ \X^\wedge := \bigcup_{n \in \mathbb{Z}_{\geq 0}} \X^\wedge_n.
\]
Given a class of objects $\Y \subseteq \mathfrak{C}$, we set  
\[
\resdim_{\X}(\Y) := \sup\{ \resdim_{\X}(Y) \ : \ Y \in \Y \}.
\]
The resolution dimension defines a function
\begin{align*}
\resdim_{\mathcal{X}}(-) & \colon \C \to \mathbb{Z}_{\geq 0} \cup \{ \infty \},
\end{align*}
for which one can consider the following notion of stability.

\begin{definition}\label{def:stable}
We say that $\resdim_{\mathcal{X}}(-)$ is \textbf{stable} if for every $M \in \C$ and $n \in \mathbb{Z}_{\geq 0}$, the following are equivalent:
\begin{enumerate}[(a)]
\item[(a)] $\resdim_{\mathcal{X}}(M) \leq n$.

\item[(b)] Any $n$-th $\mathcal{X}$-syzygy (in any $\X$-resolution) of $M$ belongs to $\mathcal{X}$. 
\end{enumerate}
\end{definition}

We can note the following property concerning stability.

\begin{proposition}\label{prop:stability_in_ses}
Let $\X \subseteq \mathfrak{C}$ be a class of objects such that $\resdim_{\X}(-)$ is stable. Then for every short exact sequence $A \rightarrowtail X \twoheadrightarrow C$ with $X \in \X$ and $A, C \in \X^\wedge$, and every $n \in \mathbb{Z}_{> 0}$, one has that $\resdim_{\X}(C) \leq n$ if, and only if, $\resdim_{\X}(A) \leq n-1$.
\end{proposition}

%%%%%%%%%%%%%%%%%%%%%%%%%%%%%%%%%%%%%
%%%%%%%%%%%%%%%%%%%%%%%%%%%%%%%%%%%%%

\subsection*{Cotorsion pairs} 

Two classes of objects $\X, \Y \subseteq \mathfrak{C}$ form a \emph{cotorsion pair} $(\X,\Y)$ in $\mathfrak{C}$ if $\X^{\perp_1} = \Y$ and $\X = {}^{\perp_1}\Y$. A cotorsion pair $(\X,\Y)$ is \emph{left complete} if for any $C \in \mathfrak{C}$ there exists a short exact sequence $Y' \rightarrowtail X \twoheadrightarrow C$ with $X \in \mathcal{X}$ and $Y' \in \mathcal{Y}$ (in other words, every object of $\mathfrak{C}$ has a special $\X$-precover). Dually, $(\mathcal{X,Y})$ is \emph{right complete} if for any $C \in \mathfrak{C}$ there exists a short exact sequence $Y' \rightarrowtail X \twoheadrightarrow C$ with $X \in \mathcal{X}$ and $Y' \in \mathcal{Y}$ (in other words, every object of $\mathfrak{C}$ has a special $\Y$-preenvelope). A cotorsion pair that is both left and right complete is called \emph{complete}. Finally, a cotorsion pair $(\X,\Y)$ is \emph{hereditary} if $\id_{\X}(\Y) = 0$. 

Perfect cotorsion pairs are particular instances of complete cotorsion pairs, and will also appear in this paper. A cotorsion pair $(\X,\Y)$ in $\C$ is \emph{perfect} if every object in $\C$ has an $\X$-cover and a $\Y$-envelope.

%%%%%%%%%%%%%%%%%%%%%%%%%%%%%%%%%%%%%
%%%%%%%%%%%%%%%%%%%%%%%%%%%%%%%%%%%%%

\subsection*{Hereditary abelian model structures} 

A model structure on an abelian category $\C$ is given by three distinguished classes of morphisms in $\C$, called cofibrations, fibrations and weak equivalences, satisfying a series of axioms (see \cite{HoveyBook}). Following \cite{Hov02}, a model structure on $\C$ is said to be \emph{abelian} if the following two conditions are satisfied:
\begin{itemize}
\item a morphism $f$ is a cofibration if, and only if, $f$ is a monomorphism and ${\rm CoKer}(f)$ is a cofibrant object; and

\item a morphism $g$ is a fibration if, and only if, $g$ is an epimorphism and ${\rm Ker}(g)$ is a fibrant object.
\end{itemize}
In this paper, we shall consider abelian model structures obtained by means of Hovey triples. Indeed, there is an appealing one-to-one correspondence between abelian model structures and cotorsion pairs, proved in \cite[Thm. 2.2]{Hov02}. We shall only be interested in the particular case where the cotorsion pairs are hereditary. In \cite[Main Thm. 1.2]{GillespieHereditary}, Gillespie shows that if $(\mathcal{X},\mathcal{Y}')$ and $(\mathcal{X}',\mathcal{Y})$ are hereditary and complete cotorsion pairs in $\C$ such that $\mathcal{X}' \subseteq \mathcal{X}$, $\mathcal{Y}' \subseteq \mathcal{Y}$ and $\mathcal{X} \cap \mathcal{Y}' = \mathcal{X}' \cap \mathcal{Y}$, then there exists a unique thick class $\mathcal{W}$ such that $(\mathcal{X},\mathcal{W},\mathcal{Y})$ is a \emph{Hovey triple}, that is, $\mathcal{X}' = \mathcal{X} \cap \mathcal{W}$ and $\mathcal{Y}' = \mathcal{Y} \cap \mathcal{W}$. Concretely, the class $\mathcal{W}$ is given by 
\begin{align*}
\mathcal{W} & = \{ M \in \C \, \text{ : } \, \text{there is a s.e.s. } M \rightarrowtail Y \twoheadrightarrow X' \, \text{ with $Y \in \Y'$ and $X' \in \X'$} \} \\
& = \{ M \in \C \, \text{ : } \, \text{there is a s.e.s. } Y' \rightarrowtail X \twoheadrightarrow M \, \text{ with $X \in \X'$ and $Y' \in \Y'$} \}. 
\end{align*}
Moreover, the hereditary abelian model structure associated to this Hovey triple is described as follows:
\begin{itemize}
\item (trivial) cofibrations are given by the class of monomorphisms with cokernel in $\mathcal{X}$ (resp., in $\mathcal{X}'$),

\item (trivial) fibrations are given by the class of epimorphisms with kernel in $\mathcal{Y}$ (resp., in $\mathcal{Y}'$), and

\item weak equivalences are given by the compositions of trivial cofibrations followed by trivial fibrations. 
\end{itemize}
By \cite[Thm. 2.6 (i)]{GillespieHereditary}, the homotopy category of this model structure is equivalent to the quotient category $\mathcal{X} \cap \mathcal{Y} / \sim$ of \emph{bifibrant} (that is, fibrant and cofibrant) objects such that for every pair of morphisms $f, g \colon X \to Y$ with $X, Y \in \mathcal{X} \cap \mathcal{Y}$, one has that $f \sim g$ if, and only if, $f - g$ factors through an object in $\mathcal{X} \cap \mathcal{W} \cap \mathcal{Y}$. 

%%%%%%%%%%%%%%%%%%%%%%%%%%%%%%%%%%%%%
%%%%%%%%%%%%%%%%%%%%%%%%%%%%%%%%%%%%%
%%%%%%%%%%%%%%%%%%%%%%%%%%%%%%%%%%%%%
%%%%%%%%%%%%%%%%%%%%%%%%%%%%%%%%%%%%%

\section{Relative completeness of Frobenius Pairs} \label{sec:completeness}

In this section, we first recall the concept of left Frobenius pairs, along with some properties and examples. After that, we focus on presenting the main concept of this paper, namely, left Frobenius pairs $(\X,\omega)$ which are $\omega^\wedge_m$-complete in $\X^{\perp_1}$, for some $m \in \mathbb{Z}_{\geq 0}$  (see Definition \ref{Def-w-complete} below). This is motivated by the following. We know from \cite[Thms. 2.11 \& 3.6]{BMSP} that if $(\X,\omega)$ is a left Frobenius pair in an abelian category $\mathfrak{C}$, then $\X^\wedge$ is a thick (and so weakly idempotent complete exact) subcategory of $\mathfrak{C}$ and that $(\X,\omega^\wedge)$ is a complete cotorsion pair in $\X^\wedge$. Moreover, if $(\X,\omega)$ is strong, then $(\omega,\X^\wedge)$ is also a complete cotorsion pair in $\X^\wedge$, compatible with $(\X,\omega^\wedge)$ in the sense that $\omega = \X \cap \omega^\wedge$ and $\omega^\wedge = \X^{\perp_1} \cap \X^\wedge$. This allows to obtain a projective exact model structure on $\X^\wedge$ where $\X$ and $\omega^\wedge$ are the classes of cofibrant and trivial objects, respectively. Using complete (and not necessarily strong) left Frobenius pairs, we show in Section \ref{sec:induced_cotorsion} how to induce other families of cotorsion pairs and abelian and exact model structures involving $\resdim_{\X}(-)$ and $\resdim_{\omega}(-)$. 

Left Frobenius pairs are a key concept in this paper, since some interesting homological and homotopical outcomes can be obtained from the existence of such pairs, as mentioned previously. For instance, some sufficient conditions that a class of objects $\X \subseteq \mathfrak{C}$ needs to fulfill so that $\resdim_{\X}(-)$ is stable are included in the concept of left Frobenius pairs. From \cite[Def. 2.5]{BMSP}, we know that a \emph{left Frobenius pair} is formed by two classes of objects $\X, \omega \subseteq \mathfrak{C}$ such that:
\begin{enumerate}[(1)]
\item[{\sf (FP1)}] $\mathcal{X}$ is left thick, that is, it is closed under extensions, under taking kernels of epimorphisms between its objects (epikernels, for short) and direct summands. 

\item[{\sf (FP2)}] $\omega$ is closed under direct summands (in $\mathfrak{C}$).

\item[{\sf (FP3)}] $\id_{\mathcal{X}}(\omega) = 0$.

\item[{\sf (FP4)}] $\omega$ is a \emph{relative cogenerator} in $\mathcal{X}$, that is, $\omega \subseteq \X$ and for every object $X \in \mathcal{X}$ there is an exact sequence $X \rightarrowtail W \twoheadrightarrow X'$ where $W \in \omega$ and $X' \in \mathcal{X}$. 
\end{enumerate}
If in addition $\pd_{\X}(\omega) = 0$ and $\omega$ is a relative generator in $\X$ (that is, the dual of conditions {\sf (FP3)} and {\sf (FP4)} are satisfied), then the left Frobenius pair $(\X,\omega)$ is called \emph{strong}. The dual concept is called \emph{right (strong) Frobenius pair}.  In this paper, only left Frobenius pairs will be considered, and so by ``\emph{Frobenius pair}'' we shall mean a left Frobenius pairs. This will imply in particular that, for simplicity, all of the results presented in this paper will be stated in terms of left Frobenius pairs, although dual versions for right Frobenius pairs are also valid.

\begin{example}\label{ex:FP-GP}
Let ${}_R\mathcal{P}$ and ${}_R\mathcal{GP}$ denote the classes of projective and Gorenstein projective left $R$-modules, respectively. It is known that $$({}_R\mathcal{GP},{}_R\mathcal{P})$$ is a strong Frobenius pair (see \cite[Prop. 6.1]{BMSP}). We shall refer to it as the \textbf{Gorenstein projective Frobenius pair}.
\end{example}

More examples, some previously known in the literature and some other new, will be presented in the sequel. The following result summarizes some basic important properties of Frobenius pairs that will be very useful in the sequel.

\begin{proposition}\label{prop:Frobenius_stable}  
Let $(\mathcal{X},\omega)$ be a Frobenius pair in $\mathfrak{C}$. The following assertions hold true:
\begin{enumerate}[(1)]
\item $\resdim_{\mathcal{X}}(-)$ is stable.

\item $\X^\wedge_m$ is left thick for every $m \in \mathbb{Z}_{\geq 0}$.
\end{enumerate}  
\end{proposition}

\begin{proof}
Part (1) can be found in \cite[Prop. 2.14]{BMS}, and was originally proved in \cite[Prop. 3.3]{AB89} (although stated with a different but equivalent formulation). Part (2), on the other hand, follows by the fact that $\X^\wedge$ is thick (meaning that $\X^\wedge$ is left thick and closed under monocokernels) proved in \cite[Thm. 2.11]{BMSP}, along with the description of $\resdim_{\X}(-)$ given by $\resdim_{\X}(M) = \pd_{\omega}(M)$ for every $M \in \X^\wedge$ (see \cite[Thm. 2.10]{BMSP}).
\end{proof}

Complete Frobenius pairs obtained from hereditary and complete cotorsion pairs will be the main source of model structures in Section \ref{sec:model_structures}. It is then important to present the following ad hoc example.

\begin{example}\label{ex:GP}
Every hereditary and (right) complete cotorsion pair $(\X,\Y)$ in $\mathfrak{C}$ induces a Frobenius pair, namely $$(\X,\X \cap \X^\perp)$$ (see \cite[Prop. 2.5]{LiangYangConstructions}). We shall say that a Frobenius pair $(\X,\omega)$ in $\mathfrak{C}$ is \textbf{induced} if there exists a hereditary and complete cotorsion pair $(\X,\X^\perp)$ such that $\omega = \X \cap \X^\perp$. A counterexample of a Frobenius pair that is not induced by a hereditary and complete cotorsion pair is displayed in \cite[Ex. 2.13]{LiangMaYang-contraejemplo}. Moreover, the Frobenius pair from this reference is also strong (see \cite[Ex. 3.12]{HMP-cut}).
\begin{enumerate}[(1)]
\item It is not known whether the Gorenstein projective Frobenius pair is induced by a hereditary and complete cotorsion pair. This is related to the open question of determining if the Gorenstein projective cotorsion pair $({}_R\mathcal{GP},[{}_R\mathcal{GP}]^{\perp_1})$ is complete. Indeed, in \cite[Coroll. 3.4]{Cortes}, Cort\'es-Izurdiaga and \v{S}aroch proved that $({}_R\mathcal{GP},[{}_R\mathcal{GP}]^{\perp_1})$ is always a hereditary cotorsion pair with $[{}_R\mathcal{GP}]^{\perp_1}$ resolving (that is, closed under extensions, epikernels and contains ${}_R\mathcal{P}$). Moreover, since the pair is hereditary, $[{}_R\mathcal{GP}]^{\perp_1}$ is also coresolving and $[{}_R\mathcal{GP}]^{\perp_1} = [{}_R\mathcal{GP}]^{\perp}$, and so it is thick. 

Let us recall some well known instances of rings $R$ over which ${}_R\mathcal{GP}$ is special precovering (and so the hereditary cotorsion pair $({}_R\mathcal{GP},[{}_R\mathcal{GP}]^{\perp})$ is complete):
\begin{enumerate}[(i)]
\item If $R$ is an $m$-Iwanaga-Gorenstein ring, then $[{}_R\mathcal{GP}]^\perp = {}_R\mathcal{P}^\wedge_m$ and $({}_R\mathcal{GP},[{}_R\mathcal{GP}]^\perp)$ is complete (see for instance \cite[Thm. 8.3]{Hov02}). Here, ${}_R\mathcal{P}^\wedge_m$ denotes the class of left $R$-modules with projective dimension $\leq m$. The containment ${}_R\mathcal{P}^\wedge_m \subseteq [{}_R\mathcal{GP}]^{\perp}$, however, may be strict over an arbitrary ring $R$.

\item Estrada, Iacob and Yeomans proved in \cite[Thm. 1]{EIY17} that every left $R$-module has a special Gorenstein projective precover if ${}_R\mathcal{GP} \subseteq {}_R\mathcal{GF} \subseteq {}_R\mathcal{GP}^\wedge$, where ${}_R\mathcal{GF}$ denotes the class of Gorenstein flat $R$-modules. 

\item Estrada and Iacob also showed in \cite[Thm. 3]{Estrada24} that ${}_R\mathcal{GP}$ is special precovering provided that there exists $n \in \mathbb{Z}_{\geq 0}$ such that every finitely presented left $R$-module has Gorenstein projective dimension $\leq n$. 
\end{enumerate}

\item Let ${}_R\mathcal{F}$, ${}_R\mathcal{C} = [{}_R\mathcal{F}]^{\perp_1}$ and ${}_R\mathcal{GC} = [{}_R\mathcal{GF}]^{\perp_1}$ denote the classes of flat, cotorsion and Gorenstein cotorsion left $R$-modules. It is known that $({}_R\mathcal{F}, {}_R\mathcal{C})$ is a hereditary and perfect cotorsion pair in ${}_R\mathsf{Mod}$ with $R$ an arbitrary ring (see for instance \cite[Prop. 7.4.3 \& Thm. 7.4.4]{EnJen00}). On the other hand, also for an arbitrary ring $R$, in \cite[Coroll. 4.12]{Saroch} it is shown that $({}_R\mathcal{GF},{}_R\mathcal{GC})$ is a hereditary and perfect cotorsion pair in ${}_R\mathsf{Mod}$. Also, it was shown in \cite[Prop. 3.1 (1)]{EstradaMarco} that ${}_R\mathcal{GF} \cap {}_R\mathcal{GC} = {}_R\mathcal{F} \cap {}_R\mathcal{C}$. Hence, $$({}_R\mathcal{F},{}_R\mathcal{F} \cap {}_R\mathcal{C})$$ and $$({}_R\mathcal{GF},{}_R\mathcal{F} \cap {}_R\mathcal{C})$$ are induced Frobenius pairs in ${}_R\mathsf{Mod}$. We shall refer to them as the \textbf{flat Frobenius pair} and the \textbf{Gorenstein flat Frobenius pair}, respectively. They are not strong in general. Consider for instance a commutative von Neumann regular ring $R$ which is not noetherian (and so not quasi-Frobenius). Then, ${}_R\mathcal{F} = {}_R\mathsf{Mod}$, and so one has ${}_R\mathcal{F} \cap {}_R\mathcal{C} = {}_R\mathsf{Mod} \cap {}_R\mathcal{I} = {}_R\mathcal{I}$, where ${}_R\mathcal{I}$ is the class of injective left $R$-modules. Note also that ${}_R\mathcal{F} = {}_R\mathsf{Mod}$ implies ${}_R\mathcal{GF} = {}_R\mathsf{Mod}$. Since ${\rm pd}({}_R\mathcal{I}) > 0$, we have that $({}_R\mathcal{F},{}_R\mathcal{F} \cap {}_R\mathcal{C}) = ({}_R\mathcal{GF},{}_R\mathcal{F} \cap {}_R\mathcal{C}) = ({}_R\mathsf{Mod},{}_R\mathcal{I})$ is not strong.

\item Some of the results just mentioned in (2) are valid in the category $\mathfrak{Qcoh}(X)$ of quasi-coherent sheaves over a semi-separated and noetherian scheme $X$. Namely, if we let $\mathfrak{F}(X)$, $\mathfrak{C}(X)$, $\mathfrak{GF}(X)$ and $\mathfrak{GC}(X)$ denote the classes of flat, cotorsion, Gorenstein flat, and Gorenstein cotorsion quasi-coherent sheaves over $X$, then $(\mathfrak{F}(X),\mathfrak{C}(X))$ and $(\mathfrak{GF}(X),\mathfrak{GC}(X))$ are hereditary and complete cotorsion pairs in $\mathfrak{Qcoh}(X)$ (see \cite[Rmk. 2.4 \& Thm. 2.2]{ChristensenEstradaThompson-stable}, respectively). Moreover, the compatibility condition $\mathfrak{GF}(X) \cap \mathfrak{GC}(X) = \mathfrak{F}(X) \cap \mathfrak{C}(X)$ is proved in \cite[Lem. 2.3]{ChristensenEstradaThompson-stable}. We thus have the Frobenius pairs in $\mathfrak{Qcoh}(X)$ given by 
\[
(\mathfrak{F}(X),\mathfrak{F}(X) \cap \mathfrak{C}(X)) \, \, \text{ and} \, \, (\mathfrak{GF}(X),\mathfrak{F}(X) \cap \mathfrak{C}(X)).
\] 

\item Flat and Gorenstein flat objects have also been studied in categories of quiver representations. Consider the category $\mathsf{Rep}(Q,{}_R\mathsf{Mod})$, and let $\mathsf{F}(Q)$ and $\mathsf{GF}(Q)$ denote the classes of flat and Gorenstein flat representations of $Q$ in ${}_R\mathsf{Mod}$ (the reader can recall these definitions in \cite[\S \, 1.8 \& Def. 3.2]{DiEstradaLiangOdabasi}). Let $\mathsf{C}(Q) \coloneqq [\mathsf{F}(Q)]^{\perp_1}$ and $\mathsf{GC}(Q) \coloneqq [\mathsf{GF}(Q)]^{\perp_1}$ denote the classes of cotorsion and Gorenstein cotorsion representations. In the case where $Q$ is a left rooted quiver, it is proved in \cite[Lem. 1.10 \& Coroll. 3.10]{DiEstradaLiangOdabasi} that $(\mathsf{F}(Q), \mathsf{C}(Q))$ is a hereditary complete cotorsion pair and that $(\mathsf{GF}(Q), \mathsf{GC}(Q))$ is a hereditary perfect cotorsion pair with $\mathsf{GF}(Q)$ closed under direct limits. Furthermore, $\mathsf{F}(Q) \cap \mathsf{C}(Q) = \mathsf{GF}(Q) \cap \mathsf{GC}(Q)$. It then follows that
\[
(\mathsf{F}(Q), \mathsf{F}(Q) \cap \mathsf{C}(Q)) \, \, \text{ and } \, \, (\mathsf{GF}(Q), \mathsf{F}(Q) \cap \mathsf{C}(Q))
\]
are Frobenius pairs in $\mathsf{Rep}(Q,{}_R\mathsf{Mod})$.
\end{enumerate}
\end{example}

The following property of Frobenius pairs $(\X,\omega)$ provides a description of $\omega^\wedge_m$. It can be regarded as a ``bounded version'' of \cite[Prop. 2.13]{BMSP}.

\begin{lemma} \label{LemaGorro}
Let $(\X, \omega)$ be a Frobenius pair in $\mathfrak{C}$. Then, the following assertions hold for every $m \in \mathbb{Z}_{\geq 0}$:
\begin{enumerate}[(1)]
\item $\omega^{\wedge}_m$ is closed under direct summands.

\item $\omega^{\wedge}_m =  \X^{\perp} \cap \X^{\wedge}_m = \X^{\perp_1} \cap \X^{\wedge}_m$. \\
\end{enumerate}
\end{lemma}

\begin{proof} ~\
\begin{enumerate}[(1)]
\item Let $C^1 \oplus C^2 = C \in \omega^\wedge_m$. First, note that $\omega^\wedge = \mathcal{X}^\perp \cap \mathcal{X}^\wedge$ by \cite[Prop. 2.13]{BMSP}, and so $\omega^\wedge$ is closed under direct summands since $\mathcal{X}^\perp$ and $\mathcal{X}^\wedge$ are. Indeed, this closure property is clear for the former, while for the latter class one uses \cite[Thm. 2.11]{BMSP}. Thus, $C^1, C^2 \in \omega^\wedge$, and so there are exact sequences
\[
K^1_m \rightarrowtail W^1_{m-1} \to \cdots \to W^1_1 \to W^1_0 \twoheadrightarrow C^1
\]
and
\[
K^2_m \rightarrowtail W^2_{m-1} \to \cdots \to W^2_1 \to W^2_0 \twoheadrightarrow C^2
\]
where $W^i_k \in \omega$ for $i = 1, 2$ and every $0 \leq k \leq m-1$, and $K^i_m \in \omega^\wedge = \mathcal{X}^\wedge \cap \mathcal{X}^\perp$ for $i = 1, 2$. The direct sum of the previous two complexes yields an exact complex
\[
K^1_m \oplus K^2_m \rightarrowtail W^1_{m-1} \oplus W^2_{m-1} \to \cdots \to W^1_1 \oplus W^2_1 \to W^1_0 \oplus W^2_0 \twoheadrightarrow C^1 \oplus C^2
\]
where $W^1_k \oplus W^2_k \in \omega$ since $\omega$ is closed under finite coproducts by \cite[Prop. 2.7 (2)]{BMSP}. By Proposition \ref{prop:Frobenius_stable} (1), we have that $K^1_m \oplus K^2_m \in \mathcal{X}$. Hence, $K^1_m \oplus K^2_m \in \mathcal{X} \cap \omega^\wedge = \omega$. Finally, since $\omega$ is closed under direct summands, we obtain that $K^1_m \in \omega$ and $K^2_m \in \omega$, and hence $C^1, C^2 \in \omega^\wedge_m$.

\item Let $M \in \omega^\wedge_m$. It is clear that $M \in \X^\wedge_m$. On the other hand, since $\id_{\X}(\omega^\wedge_m) = \id_{\mathcal{X}}(\omega) = 0$ (for the first equality, see for instance \cite[Lem. 2.1]{BMSP}), we have that $M \in \mathcal{X}^\perp$. Thus, we have the containment $\omega^\wedge_m \subseteq \X^{\perp} \cap \X^\wedge_m \subseteq \X^{\perp_1} \cap \X^\wedge_m$. Now let $N \in \mathcal{X}^{\perp_1} \cap \mathcal{X}^\wedge_m$. By \cite[Thm. 2.8]{BMSP}, there is a short exact sequence $N \rightarrowtail H \twoheadrightarrow X$ with $X \in \mathcal{X}$ and $H \in \omega^\wedge_m$. Since $\Ext^1_{\mathfrak{C}}(X,N) = 0$, the previous sequence splits, and so $N \in \omega^\wedge_m$ by part (1). 
\end{enumerate}
\end{proof}

Now we are in position to present the main concept of this paper.

\begin{definition} \label{Def-w-complete}
Let $(\mathcal{X},\omega)$ be a Frobenius pair in $\mathfrak{C}$. Given $m \in \mathbb{Z}_{\geq 0}$, we say that $(\mathcal{X},\omega)$ is (\textbf{totally}) $\bm{\omega^{\wedge}_m}$\textbf{-complete in} $\bm{\mathcal{X}^{\perp_1}}$ if for every $M \in \mathcal{X}^{\perp_1}$:
\begin{enumerate}
\item[{\sf (LC)}] there exists a short exact sequence $K' \rightarrowtail L \twoheadrightarrow M$ with $L \in \omega^\wedge_m$ and $K' \in (\omega^\wedge_m)^{\perp_1}$ (resp., $K' \in (\omega^\wedge_m)^\perp$); and

\item[{\sf (RC)}] there exists an exact sequence $M \rightarrowtail K \twoheadrightarrow L'$ with $L' \in \omega^\wedge_m$ and $K \in (\omega^\wedge_m)^{\perp_1}$ (resp., $K \in (\omega^\wedge_m)^\perp$).
\end{enumerate}
In what follows, if $(\X , \omega)$ only satisfies the condition {\sf (LC)} (resp., {\sf (RC)}) we say that $(\mathcal{X},\omega)$ is \textbf{left} (resp., \textbf{right}) (\textbf{totally}) $\bm{\omega^{\wedge}_m}$\textbf{-complete in} $\bm{\mathcal{X}^{\perp_1}}$.

In some occasions, we shall simply say that a Frobenius pair $(\mathcal{X},\omega)$ is (totally) $\omega^\wedge_m$-complete (without mentioning $\mathcal{X}^{\perp_1}$). 
\end{definition}

\begin{example}\label{ExGF} ~\
\begin{enumerate}[(1)]
\item The Gorenstein projective Frobenius pair $({}_R\mathcal{GP},{}_R\mathcal{P})$ is totally $({}_R\mathcal{P})^\wedge_m$-complete for every $m \in \mathbb{Z}_{\geq 0}$. This is a consequence of the fact that $({}_R\mathcal{P}^\wedge_m,[{}_R\mathcal{P}^\wedge_m]^{\perp})$ is a hereditary and complete cotorsion pair in ${}_R\mathsf{Mod}$ (see for instance \cite[Thm. 7.4.6]{EnJen00}).

\item The flat and Gorenstein flat Frobenius pairs $({}_R\mathcal{F},{}_R\mathcal{F} \cap {}_R\mathcal{C})$ and  $({}_R\mathcal{GF},{}_R\mathcal{F} \cap {}_R\mathcal{C})$ are $({}_R\mathcal{F} \cap {}_R\mathcal{C})^\wedge_m$-complete for every $m \in \mathbb{Z}_{\geq 0}$. Let us first recall that if ${}_R\mathcal{F}^{\wedge}_m$ denotes the class of left $R$-modules with flat dimension $\leq m$, then $({}_R\mathcal{F}^{\wedge}_m, [{}_R\mathcal{F}^\wedge_m]^{\perp})$ is a hereditary and perfect cotorsion pair in ${}_R\mathsf{Mod}$ with $R$ an arbitrary ring (see \cite[Thm. 3.4 (2)]{MaoDing} or \cite[Thm. 4.1.3]{GT06}). In particular, for any $M \in {}_R\mathcal{GC} = [{}_R\mathcal{GF}]^{\perp}$ we have an exact sequence $K \rightarrowtail L \twoheadrightarrow M$ where $L \in {}_R\mathcal{F}^{\wedge}_m$ and $K \in [{}_R\mathcal{F}^\wedge_m]^{\perp} \subseteq [({}_R\mathcal{F} \cap {}_R\mathcal{C})^\wedge_m]^{\perp_1}$. On the other hand, since ${}_R\mathcal{F}^{\wedge}_m$ and ${}_R\mathcal{GF}$ both contain ${}_R\mathcal{F}$, we have that $[{}_R\mathcal{F}^{\wedge}_m]^{\perp} \cup {}_R\mathcal{GC} \subseteq {}_R\mathcal{C}$. It then follows that $L \in {}_R\mathcal{F}^\wedge_m \cap {}_R\mathcal{C} = ({}_R\mathcal{F} \cap {}_R\mathcal{C})^\wedge_m$, by Lemma \ref{LemaGorro} (2) applied to the Frobenius pair $({}_R\mathcal{F}, {}_R\mathcal{F} \cap {}_R\mathcal{C})$. Hence, we have {\sf (LC)} in Definition \ref{Def-w-complete}. Regarding {\sf (RC)}, note also the existence of a short exact sequence $M \rightarrowtail K \twoheadrightarrow W$ where $K \in [{}_R\mathcal{F}^\wedge_m]^{\perp} \subseteq [({}_R\mathcal{F} \cap {}_R\mathcal{C})^\wedge_m]^{\perp_1}$ and $W \in {}_R\mathcal{F}^\wedge_m$. Moreover, since $M, K \in {}_R\mathcal{C}$ and ${}_R\mathcal{C}$ is coresolving, we have that $W \in {}_R\mathcal{F}^\wedge_m \cap {}_R\mathcal{C} = ({}_R\mathcal{F} \cap {}_R\mathcal{C})^\wedge_m$. Hence, $({}_R\mathcal{GF},{}_R\mathcal{F} \cap {}_R\mathcal{C})$ is $({}_R\mathcal{F} \cap {}_R\mathcal{C})^\wedge_m$-complete in ${}_R\mathcal{GC}$ for every $m \in \mathbb{Z}_{\geq 0}$, and in a similar way one can note that $({}_R\mathcal{F},{}_R\mathcal{F} \cap {}_R\mathcal{C})$ satisfies both {\sf (LC)} and {\sf (RC)}. 

\item In a similar way, the flat and Gorenstein flat Frobenius pairs $(\mathfrak{F}(X),\mathfrak{F}(X) \cap \mathfrak{C}(X))$ and $(\mathfrak{GF}(X),\mathfrak{F}(X) \cap \mathfrak{C}(X))$ are $(\mathfrak{F}(X) \cap \mathfrak{C}(X))^\wedge_m$-complete for every $m \in \mathbb{Z}_{\geq 0}$, over a semi-separated noetherian scheme $X$. Indeed, we can follow the arguments in (2), using the fact that $([\mathfrak{F}(X)]^\wedge_m,([\mathfrak{F}(X)]^\wedge_m)^{\perp})$ is a hereditary complete cotorsion pair for $m \geq 1$, as claimed in \cite[Prop. 4.1]{ElMaaouyOuberka}.

\item In the category $\mathsf{Rep}(Q,{}_R\mathsf{Mod})$, with $Q$ a left rooted quiver, the flat and Gorenstein flat Frobenius pairs $(\mathsf{F}(Q), \mathsf{F}(Q) \cap \mathsf{C}(Q))$ and $(\mathsf{GF}(Q), \mathsf{F}(Q) \cap \mathsf{C}(Q))$ are $(\mathsf{F}(Q) \cap \mathsf{C}(Q))^\wedge_m$-complete for every $m \in \mathbb{Z}_{\geq 0}$, since $([\mathsf{F}(Q)]^\wedge_n,([\mathsf{F}(Q)]^\wedge_n)^{\perp})$ is a hereditary complete cotorsion pair by \cite[Thm. 5.4 (2)]{ElMaaouy26}. 
\end{enumerate}
\end{example}

From \cite[Prop. 2.13]{BMSP}, we can note that within the thick subcategory $\X^\wedge$ the classes $\omega^\wedge$ and $\X^{\perp_1}$ coincide, but in general one has the containments $\omega^\wedge_m \subseteq \omega^\wedge \subseteq \X^{\perp_1}$ which may be strict, as shown in the following remark.

\begin{remark}\label{rmk:strict_perp}
By Lemma \ref{LemaGorro} (2), condition {\sf (LC)} is trivially satisfied by Frobenius pairs $(\mathcal{X},\omega)$ for which $\mathcal{X}^{\perp_1} \subseteq \mathcal{X}^\wedge_m$ for some $m \in \mathbb{Z}_{\geq 0}$. There are many examples of Frobenius pairs $(\mathcal{X},\omega)$ for which the containment $\mathcal{X}^{\perp_1} \subseteq \mathcal{X}^\wedge_m$ does not occur for any $m \in \mathbb{Z}_{\geq 0}$. Indeed, it suffices to consider hereditary and complete cotorsion pairs $(\mathcal{X},\mathcal{X}^\perp)$ in an abelian category $\mathfrak{C}$ with $\resdim_{\mathcal{X}}(\mathfrak{C}) = \infty$, and its induced Frobenius pair $(\X,\X \cap \X^\perp)$. 
\end{remark}

%%%%%%%%%%%%%%%%%%%%%%%%%%%%%%%%%%%%%
%%%%%%%%%%%%%%%%%%%%%%%%%%%%%%%%%%%%%
%%%%%%%%%%%%%%%%%%%%%%%%%%%%%%%%%%%%%
%%%%%%%%%%%%%%%%%%%%%%%%%%%%%%%%%%%%%

\section{Induced cotorsion pairs from complete Frobenius pairs and resolution dimensions} \label{sec:induced_cotorsion}

The goal of this section is to develop the necessary homological machinery for Frobenius pairs so that they can induce cotorsion pairs. We relate the notions of (left and right) $\omega^{\wedge}_m$-complete Frobenius pairs, (cut) cotorsion pairs and resolution dimensions. More precisely, if $(\X,\omega)$ is Frobenius pair $\omega^{\wedge}_m$-complete in $\X^{\perp_1}$, then $(\X^\wedge_m, [\omega^{\wedge}_m]^{\perp_1} \cap \X^{\perp_1})$ is a cotorsion pair cut along $\X^{\wedge}$ in the sense of \cite{HMP-cut}. In particular, if $(\X,\omega)$ is induced by a hereditary and complete cotorsion pair $(\X,\X^\perp)$, then its $\omega^\wedge_m$-completeness is equivalent to the assertion that $\mathcal{X}^\wedge_m$ is the left half of a hereditary and complete cotorsion pair. This is the first main result of this paper, and will be formally stated and proved in Theorem \ref{Family} below. To that end, we need the following two lemmas. The first one follows the ideas of the implications (1) $\Rightarrow$ (2) in \cite[Thms. 2.5. \& 3.3]{Chains}.

\begin{lemma} \label{w-m-exact}
Let $(\mathcal{X},\omega)$ be a Frobenius pair in $\mathfrak{C}$ and $m \in \mathbb{Z}_{\geq 0}$. Then for each $M \in \X^{\wedge}_m$ there is an exact sequence $K \rightarrowtail X \twoheadrightarrow M$ with $X \in \X$ and $K \in \omega^{\wedge}_{m-1}$, which is $\Hom_{\C}(-, [\omega^{\wedge}_m]^{\perp _1})$-exact, that is, for every $E \in [\omega^{\wedge}_m]^{\perp _1}$ the induced $\mathbb{Z}$-homomorphism $\Hom_{\mathfrak{C}}(X,E) \to \Hom_{\mathfrak{C}}(K,E)$ is epic.
\end{lemma}

\begin{proof}
Since $(\mathcal{X},\omega)$ is a Frobenius pair, by \cite[Thm. 2.8]{BMSP}  there is a short exact sequence $K \rightarrowtail X \twoheadrightarrow M$ with $X \in \mathcal{X}$ and $K \in \omega^\wedge_{m-1}$. On the other hand, there is a short exact sequence $X \rightarrowtail W \twoheadrightarrow X'$ with $W \in \omega$ and $X' \in \mathcal{X}$. Taking the pushout of $W \leftarrowtail X \twoheadrightarrow M$ yields the following commutative diagram with exact rows and columns: 
\[
\begin{tikzpicture}[description/.style={fill=white,inner sep=2pt}] 
\matrix (m) [matrix of math nodes, row sep=2.3em, column sep=2.3em, text height=1.25ex, text depth=0.25ex] 
{ 
K & X & M \\
K & W & Q \\
{} & X' & X' \\
}; 
\path[->] 
(m-1-2)-- node[pos=0.5] {\footnotesize$\mbox{\bf po}$} (m-2-3) 
;
\path[>->]
(m-1-1) edge (m-1-2)
(m-2-1) edge (m-2-2)
(m-1-2) edge (m-2-2)
(m-1-3) edge (m-2-3)
;
\path[->>]
(m-1-2) edge (m-1-3)
(m-2-2) edge (m-2-3)
(m-2-2) edge (m-3-2)
(m-2-3) edge (m-3-3)
;
\path[-,font=\scriptsize]
(m-1-1) edge [double, thick, double distance=2pt] (m-2-1)
(m-3-2) edge [double, thick, double distance=2pt] (m-3-3)
;
\end{tikzpicture}
\]
Now for $E \in [\omega^\wedge_m]^{\perp_1}$, we have the following commutative diagram with exact rows:
\[
\begin{tikzpicture}[description/.style={fill=white,inner sep=2pt}] 
\matrix (m) [matrix of math nodes, row sep=2.3em, column sep=2.3em, text height=1.25ex, text depth=0.25ex] 
{ 
\Hom_{\mathfrak{C}}(Q,E) & \Hom_{\mathfrak{C}}(W,E) & \Hom_{\mathfrak{C}}(K,E) & \Ext^1_{\mathfrak{C}}(Q,E) & {} \\
\Hom_{\mathfrak{C}}(M,E) & \Hom_{\mathfrak{C}}(X,E) & \Hom_{\mathfrak{C}}(K,E) & {} & {} \\
}; 
\path[>->]
(m-1-1) edge (m-1-2)
(m-2-1) edge (m-2-2)
;
\path[->]
(m-1-2) edge (m-1-3) edge (m-2-2)
(m-1-1) edge (m-2-1)
(m-2-2) edge (m-2-3)
(m-1-3) edge (m-1-4)
;
\path[-,font=\scriptsize]
(m-1-3) edge [double, thick, double distance=2pt] (m-2-3)
;
\end{tikzpicture}
\]
where $\Ext^1_{\mathfrak{C}}(Q,E) = 0$ since $Q \in \omega^\wedge_m$. The arrow $\Hom_{\mathfrak{C}}(X,E) \to \Hom_{\mathfrak{C}}(K,E)$ is then an epimorphism. 
\end{proof}

Motivated by \cite[Thm. 3.3]{Chains}, we have the following result.

\begin{lemma} \label{Lema01}
Let $(\mathcal{X},\omega)$ be a Frobenius pair in $\mathfrak{C}$ and $m \in \mathbb{Z}_{\geq 0}$. Consider the following assertions:
\begin{enumerate}[(a)]
\item $M \in \mathcal{X}^\wedge_m$.

\item $\Ext^1_{\mathfrak{C}}(M,E) = 0$ for every $E \in [\omega^\wedge_m]^{\perp_1} \cap \mathcal{X}^{\perp_1}$.
\end{enumerate} 
Then, the implication (a) $\Rightarrow$ (b) holds. Moreover, if $(\mathcal{X},\omega)$ is left $\omega^\wedge_m$-complete in $\mathcal{X}^{\perp_1}$, then (b) $\Rightarrow$ (a) holds in the case where $M \in \mathcal{X}^\wedge$. \\
\end{lemma}

\begin{proof} 
The implication (a) $\Rightarrow$ (b) is straightforward from Lemma \ref{w-m-exact}, so we focus on showing the converse implication under the additional assumptions that $(\mathcal{X},\omega)$ is left $\omega^\wedge_m$-complete in $\mathcal{X}^{\perp_1}$ and that $M$ is an object in $\mathcal{X}^\wedge$ satisfying $\Ext^1_{\mathfrak{C}}(M,E) = 0$ for every $E \in [\omega^\wedge_m]^{\perp_1} \cap \mathcal{X}^{\perp_1}$. By \cite[Thm. 2.8]{BMSP}, there exists a short exact sequence $M \rightarrowtail H \twoheadrightarrow X$ with $H \in \omega^\wedge$ and $X \in \mathcal{X}$. Since $\id_{\mathcal{X}}(\omega^\wedge) = \id_{\mathcal{X}}(\omega) = 0$, we have that $H \in \mathcal{X}^{\perp}$ and $X \in {}^\perp[\omega^\wedge]$. Moreover, for any $E \in [\omega^\wedge_m]^{\perp_1} \cap \mathcal{X}^{\perp_1}$ we have an exact sequence
\[
0 = \Ext^1_{\mathfrak{C}}(X,E) \to \Ext^1_{\mathfrak{C}}(H,E) \to \Ext^1_{\mathfrak{C}}(M,E) = 0.
\]
Then, $\Ext^1_{\mathfrak{C}}(H,E) = 0$. On the other hand, we can assert that $H \in \omega^\wedge_m$. Indeed, since $(\mathcal{X},\omega)$ is left $\omega^\wedge_m$-complete in $\mathcal{X}^{\perp_1}$ and $H \in \mathcal{X}^{\perp_1}$, there exists a short exact sequence $K \rightarrowtail L \twoheadrightarrow H$ with $L \in \omega^\wedge_m \subseteq \mathcal{X}^{\perp}$ (since $\id_{\mathcal{X}}(\omega^\wedge) = 0$) and $K \in [\omega^\wedge_m]^{\perp_1}$. Now for $X' \in \mathcal{X}$, we have an exact sequence
\[
0 = \Ext^1_{\mathfrak{C}}(X',H) \to \Ext^2_{\mathfrak{C}}(X',K) \to \Ext^2_{\mathfrak{C}}(X',L) = 0.
\]
Then, $\Ext^2_{\mathfrak{C}}(X',K) = 0$, that is, $K \in \mathcal{X}^{\perp_2}$. On the other hand, for $X' \in \mathcal{X}$ consider a short exact sequence $X' \rightarrowtail W \twoheadrightarrow X''$ with $W \in \omega$ and $X'' \in \mathcal{X}$, from which we obtain an exact sequence
\[
0 = \Ext^1_{\mathfrak{C}}(W,K) \to \Ext^1_{\mathfrak{C}}(X',K) \to \Ext^2_{\mathfrak{C}}(X'',K) = 0.
\]
Hence, $\Ext^1_{\mathfrak{C}}(X',K) = 0$ for every $X' \in \mathcal{X}$, and so $K \in \mathcal{X}^{\perp_1}$. We thus have that $K \in [\omega^\wedge_m]^{\perp_1} \cap \mathcal{X}^{\perp_1}$. Since $\Ext^1_{\mathfrak{C}}(H,K) = 0$, the sequence $K \rightarrowtail L \twoheadrightarrow H$ splits, and so $H$ is a direct summand of $L \in \omega^\wedge_m$, which in turn implies that $H \in \omega^\wedge_m$ by Lemma \ref{LemaGorro} (1). Then, there is an exact sequence $H' \rightarrowtail W_0 \twoheadrightarrow H$ where $W_0 \in \omega$ and $H' \in \omega^\wedge_{m-1}$. Taking the pullback of $M \rightarrowtail H \twoheadleftarrow W_0$ yields the following commutative diagram with exact rows and columns:
\[
\begin{tikzpicture}[description/.style={fill=white,inner sep=2pt}] 
\matrix (m) [matrix of math nodes, row sep=2.3em, column sep=2.3em, text height=1.25ex, text depth=0.25ex] 
{ 
H' & H' & {} \\
X_0 & W_0 & X \\
M & H & X \\
}; 
\path[->] 
(m-2-1)-- node[pos=0.5] {\footnotesize$\mbox{\bf pb}$} (m-3-2) 
;
\path[>->]
(m-1-1) edge (m-2-1)
(m-1-2) edge (m-2-2)
(m-2-1) edge (m-2-2)
(m-3-1) edge (m-3-2)
;
\path[->>]
(m-2-2) edge (m-2-3)
(m-3-2) edge (m-3-3)
(m-2-2) edge (m-3-2)
(m-2-1) edge (m-3-1)
;
\path[-,font=\scriptsize]
(m-1-1) edge [double, thick, double distance=2pt] (m-1-2)
(m-2-3) edge [double, thick, double distance=2pt] (m-3-3)
;
\end{tikzpicture}
\]
Now since $\mathcal{X}$ is left thick and $W_0, X \in \mathcal{X}$, we have that $X_0 \in \mathcal{X}$. Also, $H' \in \omega^\wedge_{m-1} \subseteq \mathcal{X}^\wedge_{m-1}$, and hence $M \in \mathcal{X}^\wedge_m$. 
\end{proof}

It is possible to determine another case in which the equivalence between (a) and (b) in the previous lemma holds.

\begin{remark} \label{RK-H-C}
Let $(\mathcal{X},\mathcal{X}^\perp)$ be a hereditary and complete cotorsion pair in $\mathfrak{C}$, and consider its induced Frobenius pair $(\mathcal{X},\omega)$. If this pair is left $\omega^{\wedge}_m$-complete in $\X^{\perp}$, then the implication (b) $\Rightarrow$ (a) in Lemma \ref{Lema01} holds without assuming that $M \in \X^{\wedge}$.

Indeed, let $M \in \mathfrak{C}$ be such that $\Ext^1_{\mathfrak{C}}(M,E) = 0$ for every $E \in [\omega^\wedge_m]^{\perp_1} \cap \mathcal{X}^{\perp}$. Since $(\mathcal{X},\mathcal{X}^\perp)$ is right complete, there is an exact sequence $ M \rightarrowtail H \twoheadrightarrow X$  with $H \in \X ^{\perp}$ and $X \in \X$. From this point on, one just needs to mimic the arguments given in the proof of (b) $\Rightarrow$ (a).
\end{remark}

Next we point out some outcomes on cotorsion pairs in exact categories induced from a cotorsion pair in an abelian category. The following lemma is an adaptation of \cite[Thm. 5.1]{Chains} to the purposes of our paper.

\begin{lemma} \label{RK-H-C2}
Let $\X, \Y \subseteq \mathfrak{C}$ be classes such that $(\Y, \Y ^{\perp_1})$ is a cotorsion pair with $\Y \subseteq \X ^{\wedge}$ and such that every object in $\X^{\wedge}$ has a special $\Y$-precover. Then, for any left thick class $\mathcal{B}$ such that $\Y \subseteq \mathcal{B} \subseteq \X^{\wedge}$, the following assertion holds:
\begin{enumerate}[(1)]
\item $\Y = {}^{\perp_1}[\Y^{\perp _1} \cap \B] \cap \B$. In other words, $(\Y, \Y^{\perp _1} \cap \B)$ is a cotorsion pair in the (weakly idempotent complete) exact subcategory $\mathcal{B}$.
\end{enumerate}
If in addition every object in $\X^{\wedge}$ has a special $\Y ^{\perp _1}$-preenvelope, then:
\begin{enumerate}[(1)]
\setcounter{enumi}{1}
\item The cotorsion pair $(\Y, \Y^{\perp _1} \cap \B)$ in $\mathcal{B}$ is complete.

\item If $(\Y, \Y^{\perp_1})$ is hereditary, then so is $(\Y, \Y^{\perp_1} \cap \B)$ (as a cotorsion pair in $\mathcal{B}$). \\
\end{enumerate}
\end{lemma}

Recall that an exact category is \emph{weakly idempotent complete} (\emph{w.i.c.}) if every split monomorphism has a cokernel and every split epimorphism has a kernel. This is equivalent to requiring that inflations (a.k.a. admissible monomorphisms) are closed under retracts (or that deflations / admissible epimorphisms are closed under retracts). See \cite[Prop. 2.4]{GillespieExact} to recall these concepts and results.

\begin{proof}
Note from the assumptions on $\mathcal{B}$ that $\mathcal{B}$ is a w.i.c. exact category. The proofs of assertions (1), (2) and (3) are straightforward. For the latter, it may be useful to consider the characterization of hereditary complete cotorsion pairs in exact categories given in \cite[Lem. 6.17]{Stovi14}. 
\end{proof}

We are in a position to prove our first main result. Firstly, we show that $\omega^{\wedge}_m$-completeness of certain  Frobenius pairs $(\mathcal{X},\omega)$ characterize the existence of hereditary and complete cotorsion pairs which have $\X^{\wedge} _m$ as its left half. Specifically, the latter occurs for Frobenius pairs induced by a hereditary and complete cotorsion pair $(\mathcal{X},\mathcal{X}^{\perp})$. If we drop this assumption, one can at least obtain cut cotorsion pairs. 

Let us recall from \cite[Def. 2.1]{HMP-cut} that if we are given three classes of objects $\mathcal{S}$, $\mathcal{X}$ and $\mathcal{Y}$ in an abelian category $\mathfrak{C}$, then $(\mathcal{X,Y})$ is a \emph{left cotorsion pair cut along} $\mathcal{S}$ if the following conditions are satisfied:
\begin{enumerate}
\item[{\sf (LCCP1)}] $\mathcal{X}$ is closed under direct summands.

\item[{\sf (LCCP2)}] $\mathcal{X} \cap \mathcal{S} = {}^{\perp_1}\mathcal{Y} \cap \mathcal{S}$.

\item[{\sf (LCCP3)}] For every $S \in \mathcal{S}$ there is an exact sequence $Y \rightarrowtail X \twoheadrightarrow S$ with $X \in \mathcal{X}$ and $Y \in \mathcal{Y}$.
\end{enumerate}
Similarly, if $(\mathcal{X,Y})$ and $\mathcal{S}$ satisfy the dual conditions, we say that $(\mathcal{X,Y})$ is a \emph{right cotorsion pair cut along} $\mathcal{S}$. Pairs which are both left and right cotorsion pairs cut along the same class of objects $\mathcal{S}$ will be simply called \emph{cotorsion pairs cut along} $\mathcal{S}$.

We are now in position to show the main result of this section.

\begin{theorem} \label{Family}
Let $(\mathcal{X},\omega)$ be a Frobenius pair in $\mathfrak{C}$ and $m \in \mathbb{Z}_{\geq 0}$. The following assertions hold true:
\begin{enumerate}[(1)]
\item If $(\mathcal{X},\omega)$ is induced by a hereditary and complete cotorsion pair $(\mathcal{X},\mathcal{X}^{\perp})$ in $\mathfrak{C}$, then: 
\begin{enumerate}[(i)]
\item If $(\X, \omega)$ is $\omega^{\wedge}_m$-complete in $\X^{\perp}$, then $(\X^{\wedge}_m, [\omega^{\wedge}_m]^{\perp_1} \cap \X^\perp)$ is a hereditary and complete cotorsion pair in $\C$. Conversely, if $\X^{\wedge}_m$ is the left half of a hereditary and complete cotorsion pair in $\C$, then the pair $(\X,\omega)$ is $\omega^{\wedge}_m$-complete in $\X^{\perp}$.

\item Let $\mathcal{B} \subseteq \mathfrak{C}$ be a left thick class such that $\X^{\wedge}_m \subseteq \mathcal{B} \subseteq \X^{\wedge}$. If $(\X, \omega)$ is $\omega^{\wedge}_m$-complete in $\X^\perp$, then $(\X^{\wedge}_m, [\omega^{\wedge}_m]^{\perp_1} \cap \X^{\perp} \cap \mathcal{B})$ is a hereditary and complete cotorsion pair in the w.i.c. exact category $\mathcal{B}$.
\end{enumerate}

\item If $(\X , \omega)$ is $\omega^{\wedge}_m$-complete in $\X^{\perp_1}$, then $(\X^\wedge_m, [\omega^{\wedge}_m]^{\perp_1} \cap \X^{\perp_1})$ is a cotorsion pair cut along $\X^{\wedge}$. \\
\end{enumerate} 
\end{theorem}

\begin{proof} ~\
\begin{enumerate}[(1)]
\item Let $(\X,\X^\perp)$ be a hereditary and complete cotorsion pair in $\mathfrak{C}$. It suffices to show part (i), as part (ii) follows from (i) and Lemma \ref{RK-H-C2}.
\begin{itemize}
\item \underline{Orthogonality}: $\X^\wedge_m = {}^{\perp_1}([\omega^\wedge_m]^{\perp_1} \cap \mathcal{X}^\perp)$ and $[\X^\wedge_m]^{\perp _1} = [\omega^\wedge_m]^{\perp_1} \cap \X^\perp$ follow from Remark \ref{RK-H-C} and the definitions and properties of orthogonal complements. 

\item \underline{Hereditariness}: We have to show that $\Ext^i_{\C}(M,E) = 0$ for every $M \in \X^\wedge_m$, $E \in [\omega^\wedge_m]^{\perp_1} \cap \mathcal{X}^\perp$ and $i \in \mathbb{Z}_{>0}$. By Proposition \ref{prop:Frobenius_stable} (2), $\X^\wedge_m$ left thick, and so in particular closed under $\X$-syzygies. We have by dimension shifting that $\Ext^{i+1}_{\C}(M,E) \cong \Ext^1_{\C}(M_i,E) = 0$ for every $M \in \X^\wedge_m$, $E \in [\omega^\wedge_m]^{\perp_1} \cap \mathcal{X}^\perp$ and $i \in \mathbb{Z}_{\geq 0}$, where $M_i \in \X^\wedge_m$ is any $i$-th $\X$-syzygy of $M$. 

\item \underline{Left completeness}: For any $M \in \C$, since $(\X,\X^\perp)$ is a left complete cotorsion pair, there is an exact sequence $Y \rightarrowtail X \twoheadrightarrow M$ with $X \in \X$ and $Y \in \X^\perp$. On the other hand, since $(\X,\omega)$ is right $\omega^\wedge_m$-complete in $\X^\perp$, there is an exact sequence $Y \rightarrowtail K \twoheadrightarrow L$ with $K \in [\omega^\wedge_m]^{\perp_1}$ and $L \in \omega^{\wedge}_m$. Taking the pushout of $K \leftarrowtail Y \rightarrowtail X$ yields the following commutative diagram with exact rows and columns: 
\[
\begin{tikzpicture}[description/.style={fill=white,inner sep=2pt}] 
\matrix (m) [matrix of math nodes, row sep=2.3em, column sep=2.3em, text height=1.25ex, text depth=0.25ex] 
{ 
Y & X & M \\ K & Q & M \\ L & L & {} \\
}; 
\path[->] 
(m-1-1)-- node[pos=0.5] {\footnotesize$\mbox{\bf po}$} (m-2-2) 
;
\path[>->]
(m-1-1) edge (m-1-2) edge (m-2-1)
(m-1-2) edge (m-2-2)
(m-2-1) edge (m-2-2)
;
\path[->>]
(m-1-2) edge (m-1-3)
(m-2-2) edge (m-2-3)
(m-2-1) edge (m-3-1)
(m-2-2) edge (m-3-2)
;
\path[-,font=\scriptsize]
(m-1-3) edge [double, thick, double distance=2pt] (m-2-3)
(m-3-1) edge [double, thick, double distance=2pt] (m-3-2)
;
\end{tikzpicture}
\]
Note that since $L \in \omega^\wedge_m$ and $X \in \X$, and $\X^\wedge_m$ is closed under extensions, we have that $Q \in \X^\wedge_m$. Moreover, since $Y \in \X^\perp$ and $L \in \omega^\wedge_m \subseteq \X^\perp$, and $\X^\perp$ is closed under extensions, we have that $K \in [\omega^\wedge_m]^{\perp_1} \cap \X^\perp$. 
   
\item \underline{Right completeness}: Analogous to the previous point, using that $(\X,\X^\perp)$ is right complete. 
\end{itemize}     
Now suppose that $(\X^\wedge_m,[\X^\wedge_m]^\perp)$ is a hereditary and complete cotorsion pair in $\C$, and let $M \in \X^\perp$. Then, there is an exact sequence $M \rightarrowtail K \twoheadrightarrow L'$ with $K \in [\X^\wedge_m]^\perp \subseteq [\omega^\wedge_m]^{\perp_1}$ and $L' \in \X^\wedge_m$. Since $[\X^\wedge_m]^\perp \subseteq \X^\perp$ and $\X^\perp$ is closed under monocokernels, we have that $L' \in \X^\perp \cap \X^\wedge_m = \omega^\wedge_m$ by Lemma \ref{LemaGorro} (2). The existence of an exact sequence $K' \rightarrowtail L \twoheadrightarrow M$ with $L \in \omega^\wedge_m$ and $K' \in [\omega^\wedge_m]^{\perp_1}$ follows similarly. 

\item We have that $\X^\wedge_m$ is closed under direct summands by Proposition \ref{prop:Frobenius_stable} (2), while the same property clearly holds for $[\omega^{\wedge}_m]^{\perp_1} \cap \X^{\perp_1}$. So condition {\sf (LCCP1)} and its dual are satisfied. Moreover, {\sf (LCCP2)} and its dual follow by Lemma \ref{Lema01}. Finally, {\sf (LCCP3)} and its dual follow as left and right completeness in part (1)-(i), using \cite[Thm. 2.8]{BMSP}, instead of the left and right completeness of the cotorsion pair $(\X,\X^\perp)$. 
\end{enumerate}
\end{proof}

The previous theorem, in particular, provides a way to generate hereditary and complete cotorsion pairs of the form $(\X^\wedge_m,[\X^\wedge_m]^{\perp_1})$. Let us apply this in the following example. \\

\begin{example}\label{ex:PGF} ~\
\begin{enumerate}[(1)]
\item Recall from Example \ref{ex:GP} that $({}_R\mathcal{GP},[{}_R\mathcal{GP}]^\perp)$ is always a hereditary cotorsion pair in ${}_R\mathsf{Mod}$. In \cite[Thm. 3.4]{Chains}, the authors show that $({}_R\mathcal{GP}^\wedge_m,[{}_R\mathcal{GP}^\wedge_m]^{\perp_1})$ is a complete cotorsion pair if $R$ is an Artin algebra, extending thus a result by Beligiannis and Reiten that shows that $({}_R\mathcal{GP},[{}_R\mathcal{GP}]^\perp)$ is a complete cotorsion pair over such algebras (see \cite[Thm. X.2.4.]{BeligiannisReiten}). The pair $({}_R\mathcal{GP}^\wedge_m,[{}_R\mathcal{GP}^\wedge_m]^{\perp_1})$ is also hereditary since ${}_R\mathcal{GP}^\wedge_m$ is resolving. It turns out that the assumption that $R$ is an Artin algebra is not relevant in the sense that the same conclusion can be obtained with respect to any ring $R$ over which $({}_R\mathcal{GP},[{}_R\mathcal{GP}]^\perp)$ is complete. Indeed, if the latter holds, we have a Frobenius pair $({}_R\mathcal{GP},{}_R\mathcal{P})$ induced by $({}_R\mathcal{GP},[{}_R\mathcal{GP}]^\perp)$, which is also ${}_R\mathcal{P}^\wedge_m$-complete in $[{}_R\mathcal{GP}]^\perp$ for every $m \in \mathbb{Z}_{\geq 0}$ (by Example \ref{ExGF} (1)). It then follows from Theorem \ref{Family} (1) that the following assertions are equivalent:
\begin{enumerate}[(a)] 
\item $({}_R\mathcal{GP}^\wedge_m,[{}_R\mathcal{GP}^\wedge_m]^{\perp})$ is a hereditary and complete cotorsion pair in ${}_R\mathsf{Mod}$ for every $m \in \mathbb{Z}_{\geq 0}$, with $[{}_R\mathcal{GP}^\wedge_m]^{\perp} = [{}_R\mathcal{P}^\wedge_m]^{\perp} \cap {}_R\mathcal{GP}^{\perp}$; 

\item ${}_R\mathcal{GP}$ is special precovering. 
\end{enumerate}

Regarding the completeness of $({}_R\mathcal{GP},[{}_R\mathcal{GP}]^\perp)$, a surprising result by Cortés-Izurdiaga and \v{S}aroch in \cite[Thm. 5.2]{Cortes} asserts that if there exists an infinite regular cardinal $\lambda$ such that all projective left $R$-modules are $\lambda$-pure-injective, then $({}_R\mathcal{GP},[{}_R\mathcal{GP}]^\perp)$ is complete. Of course, if such cardinals exist, $({}_R\mathcal{GP}^\wedge_m,[{}_R\mathcal{GP}^\wedge_m]^\perp)$ is also a complete cotorsion pair for every $m \in \mathbb{Z}_{\geq 0}$.

\item Recall from Example \ref{ex:GP} (2) that $({}_R\mathcal{GF},{}_R\mathcal{GC})$ is a hereditary and perfect cotorsion pair in ${}_R\mathsf{Mod}$, for every ring $R$. If one considers the class ${}_R\mathcal{GF}^\wedge_m$ of left $R$-modules with Gorenstein flat dimension $\leq m$, El Maaouy proved in \cite[Coroll. 3.5]{Rachid} that $({}_R\mathcal{GF}^\wedge_m,[{}_R\mathcal{GF}^\wedge_m]^{\perp})$ is a hereditary and perfect cotorsion pair in ${}_R\mathsf{Mod}$, for every ring $R$, where $[{}_R\mathcal{GF}^\wedge_m]^{\perp} = [{}_R\mathcal{PGF}]^{\perp} \cap [{}_R\mathcal{F}^\wedge_m]^{\perp}$. Here, ${}_R\mathcal{PGF}$ denotes the class of projectively coresolved Gorenstein flat left $R$-modules (see \cite[\S 4]{Saroch}).

A possible way to prove that a hereditary and complete cotorsion pair $(\X,\X^\perp)$ is perfect is to show that $\X$ is closed under direct limits (see for instance \cite[Prop. 3.1]{GT06}). This is precisely the approach followed by El Maaouy with the complete cotorsion pair $({}_R\mathcal{GF}^\wedge_m,[{}_R\mathcal{GF}^\wedge_m]^{\perp})$ (see \cite[Coroll. 3.5]{Rachid}). Another way to show that this pair is complete is to use Theorem \ref{Family} (1), since we already know that the Frobenius pair $({}_R\mathcal{GF},{}_R\mathcal{F} \cap {}_R\mathcal{C})$ is $({}_R\mathcal{F} \cap {}_R\mathcal{C})^\wedge_m$-complete. This proof does not rely on the theory of projectively coresolved Gorenstein flat modules. Moreover, we obtain an alternative description of $[{}_R\mathcal{GF}^\wedge_m]^{\perp}$, namely, $[{}_R\mathcal{GF}^\wedge_m]^{\perp} = [({}_R\mathcal{F} \cap {}_R\mathcal{C})^\wedge_m]^{\perp_1} \cap {}_R\mathcal{GC}^\perp$. 

\item Regarding the class ${}_R\mathcal{PGF}$ just mentioned above, it is known by \cite[Thm. 4.9]{Saroch} that $({}_R\mathcal{PGF},{}_R\mathcal{PGF}^\perp)$ is a hereditary and complete cotorsion pair in ${}_R\mathsf{Mod}$, with ${}_R\mathcal{PGF}^\perp$ thick. On the other hand, since ${}_R\mathcal{PGF} \subseteq {}_R\mathcal{GP}$ by \cite[Thm. 4.4]{Saroch}, one can easily note that ${}_R\mathcal{PGF} \cap {}_R\mathcal{PGF}^\perp = {}_R\mathcal{P}$. By Theorem \ref{Family} (1), we have that $({}_R\mathcal{PGF}^{\wedge}_m, [{}_R\mathcal{P}^{\wedge}_m]^{\perp_1} \cap {}_R\mathcal{PGF}^\perp)$ is a hereditary and complete cotorsion pair in ${}_R\mathsf{Mod}$ for every $m \in \mathbb{Z}_{\geq 0}$, thus recovering \cite[Coroll. 2.6 (1)]{Chains}. 

\item Given the $(\mathfrak{F}(X) \cap \mathfrak{C}(X))^\wedge_m$-complete Frobenius pair $(\mathfrak{GF}(X),\mathfrak{F}(X) \cap \mathfrak{C}(X))$ in $\mathfrak{Qcoh}(X)$, induced by the hereditary complete cotorsion pair $(\mathfrak{GF}(X),\mathfrak{GC}(X))$, we have that $([\mathfrak{GF}(X)]^\wedge_m,([\mathfrak{GF}(X)]^\wedge_m)^\perp)$ is a hereditary complete cotorsion pair, for $X$ noetherian and semi-separated. 

\item Similarly, regarding the $(\mathsf{F}(Q) \cap \mathsf{C}(Q))^\wedge_m$-complete Gorenstein flat Frobenius pair $(\mathsf{GF}(Q),\mathsf{F}(Q) \cap \mathsf{C}(Q))$ in $\mathsf{Rep}(Q,{}_R\mathsf{Mod})$, induced by the hereditary complete cotorsion pair $(\mathsf{GF}(Q),\mathsf{GC}(Q))$, we have that $([\mathsf{GF}(Q)]^\wedge_m,([\mathsf{GF}(Q)]^\wedge_m)^\perp)$ is a hereditary complete cotorsion pair, provided that $Q$ is a left rooted quiver. This result is proved under different conditions in \cite[Thm. 5.11 (2)]{ElMaaouy26}, namely, if $Q$ is a right coherent quiver and $R$ is a right coherent ring. 
\end{enumerate}
\end{example}

Another application of the previous theorem is the induction of new Frobenius pairs involving $\X^\wedge_m$.

\begin{corollary}\label{coro:inducedXmFP}
Let $(\X,\omega)$ be a Frobenius pair in $\C$ which is $\omega^{\wedge}_m$-complete in $\X^{\perp_1}$, then $(\X^\wedge_m, \omega^\wedge_m \cap [\omega^\wedge_m]^{\perp_1})$ is a Frobenius pair in $\C$.  
\end{corollary}

\begin{proof}
Let us verify the four conditions in the definition of Frobenius pairs:
\begin{itemize}
\item {\sf (FP1)}: It is a consequence of Proposition \ref{prop:Frobenius_stable} (2). 

\item {\sf (FP2)}: That $\omega^\wedge_m$ is closed under direct summands follows by Lemma \ref{LemaGorro} (1), while $[\omega^\wedge_m]^{\perp_1}$ clearly satisfies the same closure property. Then, $\omega^\wedge_m \cap [\omega^\wedge_m]^{\perp_1}$ is closed under direct summands. 

\item {\sf (FP3)}: We proceed as in part (1) of Theorem \ref{Family}. Let $L \in \X^\wedge_m$ and $K \in \omega^\wedge_m \cap [\omega^\wedge_m]^{\perp_1}$. Note that $K \in  [\omega^\wedge_m]^{\perp_1} \cap \X^\perp$ since $\id_{\X}(\omega^\wedge_m) = 0$. On the other hand, since $\X^\wedge_m$ is closed under $\X$-syzygies, for every $i > 0$ we have by dimension shifting and Lemma \ref{Lema01} that $\Ext^{i+1}_{\C}(L,K) \cong \Ext^1_{\C}(L',K) = 0$, where $L'$ is any $m$-th $\X$-syzygy of $L$. Hence, $\id_{\X^\wedge_m}(\omega^\wedge_m \cap [\omega^\wedge_m]^{\perp_1}) = 0$.

\item {\sf (FP4)}: The containment $\omega^\wedge_m \cap [\omega^\wedge_m]^{\perp_1} \subseteq \X^\wedge_m$ is clear. On the other hand, by Theorem \ref{Family} (2), for every $M \in \X^\wedge_m$ there is a short exact sequence $M \rightarrowtail K \twoheadrightarrow L$ with $L \in \X^\wedge_m$ and $K \in [\omega^\wedge_m]^{\perp_1} \cap \X^{\perp_1}$. By Lemma \ref{LemaGorro} (2) and the fact that $\X^\wedge_m$ is closed under extensions, we have that $K \in [\omega^\wedge_m]^{\perp_1} \cap \X^{\perp_1} \cap \X^\wedge_m = \omega^\wedge_m \cap [\omega^\wedge_m]^{\perp_1}$. Hence, $\omega^\wedge_m \cap [\omega^\wedge_m]^{\perp_1}$ is a relative cogenerator in $\X^\wedge_m$.
\end{itemize}
\end{proof}

\begin{example}
From Examples \ref{ExGF} and \ref{ex:PGF}, and the previous corollary, we have several Frobenius pairs for each $m \in \mathbb{Z}_{\geq 0}$. 
\end{example}

%%%%%%%%%%%%%%%%%%%%%%%%%%%%%%%%%%%%%
%%%%%%%%%%%%%%%%%%%%%%%%%%%%%%%%%%%%%
%%%%%%%%%%%%%%%%%%%%%%%%%%%%%%%%%%%%%
%%%%%%%%%%%%%%%%%%%%%%%%%%%%%%%%%%%%%

\section{Some homotopical aspects from complete Frobenius pairs} \label{sec:model_structures}

This section is devoted to build hereditary abelian and exact model category structures by means of Hovey triples obtained from one or two Frobenius pairs induced by cotorsion pairs. The condition of $\omega$-completeness will be key in our constructions. More precisely, the first three methods to obtain model structures described below involve a hereditary and complete cotorsion pair $(\mathcal{X},\mathcal{X}^\perp)$ in an abelian category such that for $\omega = \mathcal{X} \cap \mathcal{X}^\perp$ there exists $m \in \mathbb{Z}_{\geq}$ for which $\omega^\wedge_m$ is the left half of a complete cotorsion pair. Several model structures in the literature involving the classes of Gorenstein projective and projectively coresolved Gorenstein flat modules can be obtained using one of these three methods. Regarding model structures involving Gorenstein flat modules, it is not true in general that $\omega^\wedge_m$ is the left half of a complete cotorsion pair. For these situations, we also present a forth method to obtain such model structures. \\

\begin{theorem}[method 1] \label{HoveyinC}
Let $(\mathcal{X},\mathcal{X}^{\perp})$ be a hereditary and complete cotorsion pair in an abelian category $\mathfrak{C}$ with enough projective objects. For $\omega = \mathcal{X} \cap \mathcal{X}^\perp$, if there exists $m \in \mathbb{Z}_{\geq 0}$ such that $(\omega^\wedge_m, [\omega^\wedge_m]^{\perp_1})$ is a complete cotorsion pair in $\C$ (and so, the induced Frobenius pair $(\X,\omega)$ is $\omega^\wedge_m$-complete in $\X^\perp$), then the following assertions are equivalent:
\begin{enumerate}[(a)]
\item $\X^\perp$ is closed under epikernels (and so, thick).

\item $\omega^\wedge_m$ is closed under epikernels (and so, the pair $(\omega^\wedge_m, [\omega^\wedge_m]^{\perp_1})$ is hereditary).

\item $\omega$ is the class $\mathcal{P}$ of projective objects of $\C$.
\end{enumerate}
If any of these conditions holds, then there is a Hovey triple $$(\X^\wedge_m, \X^\perp, [\mathcal{P}^\wedge_m]^\perp)$$ in $\C$ and the homotopy category of its associated hereditary abelian model structure is equivalent to the quotient category $$\X^\wedge_m \cap [\mathcal{P}^\wedge_m]^\perp / \sim,$$ where for two morphisms $f, g \colon X \to Y$ with $X, Y \in \X^\wedge_m \cap [\mathcal{P}^\wedge_m]^\perp$ one has that $f \sim g$ if, and only if, $f - g$ factors through a trivially bifibrant object (that is, and object in $\X^\wedge_m \cap \mathcal{X}^\perp \cap [\mathcal{P}^\wedge_m]^\perp = \mathcal{P}^\wedge_m \cap [\mathcal{P}^\wedge_m]^\perp$). \newpage 
\end{theorem}

\begin{proof} ~\
\begin{itemize}
\item (a) $\Rightarrow$ (b): By Lemma \ref{LemaGorro} (2) we have that $\omega^\wedge_m = \mathcal{X}^\wedge_m \cap \mathcal{X}^\perp$, where $\mathcal{X}^\wedge_m$ is closed under epikernels by Proposition \ref{prop:Frobenius_stable} (2). Then, since $\mathcal{X}^\perp$ is closed under epikernels by assumption, so is $\omega^\wedge_m$. 

\item (b) $\Rightarrow$ (c): For any $P \in \mathcal{P}$, since $(\omega^\wedge_m, [\omega^\wedge_m]^{\perp_1})$ is a complete cotorsion pair, there is an exact sequence $K \rightarrowtail W \twoheadrightarrow P$ with $W \in \omega$, which is split since $P$ is projective. Then, $P \in \omega$ since $\omega$ is closed under direct summands. 

Now let $W \in \omega$. Since $\C$ has enough projective objects, there is an exact sequence $W' \rightarrowtail P \twoheadrightarrow W$ with $P \in \mathcal{P} \subseteq \omega$. Using the assumption that $\omega^\wedge_m$ is closed under epikernels, we have $W' \in \omega^\wedge_m$. Note then that $\Ext^1_{\C}(W,W') = 0$ since $\id_{\X}(\omega^\wedge) = 0$, $W \in \X$ and $W' \in \omega^\wedge_m$. Thus, the previous exact sequence splits, and hence $W$ is projective.  

\item (c) $\Rightarrow$ (a): The following argument mimics the proof of \cite[Lem. 3.2 (2)]{Cortes}. Consider a short exact sequence $Y_1 \rightarrowtail Y_2 \twoheadrightarrow Y_3$ with $Y_2, Y_3 \in \mathcal{X}^\perp$, and let $X \in \mathcal{X}$. Since $(\mathcal{X},\omega = \mathcal{P})$ is a Frobenius pair, there is an exact sequence $X \rightarrowtail P \twoheadrightarrow X'$ with $P \in \mathcal{P}$ and $X' \in \mathcal{X}$. On the one hand, for $X' \in \mathcal{X}$ we have the induced exact sequence 
\[
\Ext^1_{\C}(X',Y_3) \to \Ext^2_{\C}(X',Y_1) \to \Ext^2_{\C}(X',Y_2),
\]
where $\Ext^1_{\C}(X',Y_3) = 0$ and $\Ext^2_{\C}(X',Y_2) = 0$ since $Y_2, Y_3 \in \mathcal{X}^\perp$, and so $\Ext^2_{\C}(X',Y_1) = 0$. On the other hand, we have the induced exact sequence
\[
\Ext^1_{\C}(P,Y_1) \to \Ext^1_{\C}(X,Y_1) \to \Ext^2_{\C}(X',Y_1),
\]
where $\Ext^1_{\C}(P,Y_1) = 0$ and $\Ext^2_{\C}(X',Y_1) = 0$, and so $\Ext^1_{\C}(X,Y_1) = 0$. It follows that $Y_1 \in \mathcal{X}^{\perp_1} = \mathcal{X}^\perp$. 
\end{itemize}
Regarding the last part, from any of these conditions and Theorem \ref{Family} (1)-(i), we have two hereditary and complete cotorsion pairs $(\mathcal{X}^\wedge_m, [\mathcal{X}^\wedge_m]^\perp)$ and $(\mathcal{P}^\wedge_m,[\mathcal{P}^\wedge_m]^\perp)$, where $\mathcal{P}^\wedge_m \subseteq \mathcal{X}^\wedge_m$ and $[\X^{\wedge}_m]^{\perp} = [\mathcal{P}^\wedge_m]^\perp \cap \X^\perp$. These two pairs are compatible, since from Lemma \ref{LemaGorro} (2) we have that
\[
\mathcal{X}^\wedge_m \cap [\mathcal{X}^\wedge_m]^\perp = \mathcal{X}^\wedge_m \cap [\mathcal{P}^\wedge_m]^\perp \cap \X^\perp = \mathcal{P}^\wedge_m \cap [\mathcal{P}^\wedge_m]^\perp.
\]
It then follows that there is a unique thick class $\mathcal{W}$ such that $(\mathcal{X}^\wedge_m,\mathcal{W},[\mathcal{P}^\wedge_m]^\perp)$ is a Hovey triple. Concretely, 
\begin{align*}
\mathcal{W} & = \{ M \in \C \, \text{ : } \, \text{there is a s.e.s. } M \rightarrowtail K \twoheadrightarrow L' \, \text{ with $K \in [\mathcal{X}^\wedge_m]^\perp$ and $L' \in \mathcal{P}^\wedge_m$} \} \\
& = \{ M \in \C \, \text{ : } \, \text{there is a s.e.s. } K' \rightarrowtail L \twoheadrightarrow M \, \text{ with $L \in \mathcal{P}^\wedge_m$ and $K' \in [\mathcal{X}^\wedge_m]^\perp$} \}. 
\end{align*}
We show that $\mathcal{W} = \mathcal{X}^\perp$. For any $W \in \mathcal{W}$, there exists a short exact sequence $K' \rightarrowtail L \twoheadrightarrow W$ with $L \in \mathcal{P}^\wedge_m$ and $K' \in [\mathcal{X}^\wedge_m]^\perp$. Since $\mathcal{P}^\wedge_m, [\mathcal{X}^\wedge_m]^\perp \subseteq \mathcal{X}^\perp$, and $\mathcal{X}^\perp$ is closed under monocokernels, we have that $W \in \mathcal{X}^\perp$. Hence, $\mathcal{W} \subseteq \mathcal{X}^\perp$. On the other hand, for any $Y \in \mathcal{X}^\perp$ we have a short exact sequence $Y \rightarrowtail K \twoheadrightarrow L'$ with $K \in [\mathcal{P}^\wedge_m]^\perp$ and $L' \in \mathcal{P}^\wedge_m$. Note that $L' \in \mathcal{X}^\perp$, and since $\mathcal{X}^\perp$ is closed under extensions, we get $K \in [\mathcal{P}^\wedge_m]^\perp \cap \mathcal{X}^\perp = [\X^{\wedge}_m]^{\perp}$, and hence $Y \in \mathcal{W}$. \\ 
\end{proof}

\begin{example}\label{ex:Gorenstein-models-Rmod} ~\
\begin{enumerate}[(1)]
\item We can recover \cite[Coroll. 2.6 (2) \& Thm. 3.4]{Chains} from the previous theorem. Indeed, over an arbitrary ring $R$, since $({}_R\mathcal{PGF},{}_R\mathcal{PGF}^\perp)$ is a hereditary and complete cotorsion pair with ${}_R\mathcal{PGF} \cap {}_R\mathcal{PGF}^\perp = {}_R\mathcal{P}$ (see Example \ref{ex:PGF}), there exists a unique hereditary abelian model structure on ${}_R\mathsf{Mod}$ given by the Hovey triple 
\[
({}_R\mathcal{PGF}^\wedge_m,{}_R\mathcal{PGF}^\perp,[{}_R\mathcal{P}^\wedge_m]^\perp)
\]
for every $m \in \mathbb{Z}_{\geq 0}$. On the other hand, if $R$ is a ring over which every left $R$-module has a Gorenstein projective special precover, then for every $m \in \mathbb{Z}_{\geq 0}$ one has the Hovey triple 
\[
({}_R\mathcal{GP}^\wedge_m,{}_R\mathcal{GP}^\perp,[{}_R\mathcal{P}^\wedge_m]^\perp).
\] 

\item Recall that a \textbf{duality pair} $(\mathcal{L,A})$ over $R$ is given by two classes $\mathcal{L} \subseteq {}_R\mathsf{Mod}$ and $\mathcal{A} \subseteq \mathsf{Mod}_R$ satisfying:
\begin{itemize}
\item $M \in \mathcal{L}$ if, and only if, $M^+ := \Hom_{\mathbb{Z}}(M,\mathbb{Q/Z}) \in \mathcal{A}$.

\item $\mathcal{A}$ is closed under direct summands and finite direct sums. 
\end{itemize}
The duality pair $(\mathcal{L,A})$ is called \textbf{perfect} if $\mathcal{L}$ is closed under coproducts and extensions, and $R \in \mathcal{L}$ as a left $R$-module. If for a perfect duality pair $(\mathcal{L,A})$ one has that $(\mathcal{A,L})$ is also a duality pair, then $(\mathcal{L,A})$ is called \textbf{complete}. 
\begin{enumerate}
\item[(i)] If $R$ is a right coherent ring, we can deduce from \cite{Fieldhouse} that $({}_R\mathcal{F},\mathcal{AP}_R)$ is a complete duality pair, where $\mathcal{AP}_R$ denotes the class of absolutely pure (a.k.a. FP-injective) right $R$-modules. 

\item[(ii)] Over any ring $R$, if ${}_R\mathcal{L}$ and $\mathcal{AC}_R$ denote the classes of level left $R$-modules and absolutely clean right $R$-modules, respectively, then $({}_R\mathcal{L},\mathcal{AC}_R)$ is a complete duality pair. See \cite[\S 2]{BravoGillespieHovey-arxiv} to recall these concepts. 
\end{enumerate}
More examples of compete duality pairs can by found in \cite[\S 2.1]{Gillespie-DP}. On the other hand, given a (complete) duality pair $(\mathcal{L,A})$, recall from \cite[Def. 4.2]{Gillespie-DP} that a left $R$-module $M$ is \textbf{Gorenstein $\bm{(\mathcal{L,A})}$-projective} if there exists an exact complex $P_\bullet$ of projective left $R$-modules such that $M = Z_0(P_\bullet)$ and the induced complex $\Hom_R(P_\bullet, L)$ of abelian groups is exact for every $L \in \mathcal{L}$. Let us denote by $\mathcal{GP}_{(\mathcal{L,A})}$ the class of Gorenstein $(\mathcal{L,A})$-projective left $R$-modules. In particular:
\begin{enumerate}
\item[(iii)] $\mathcal{GP}_{(\mathcal{L,A})}$ coincides with the class of Ding projective left $R$-modules, denoted ${}_R\mathcal{DP}$, if $(\mathcal{L,A}) = ({}_R\mathcal{F},\mathcal{AP}_R)$.

\item[(iv)] $\mathcal{GP}_{(\mathcal{L,A})}$ coincides with the class of Gorenstein AC-projective left $R$-modules, denoted ${}_R\mathcal{GP}_{\rm AC}$, if $(\mathcal{L,A}) = ({}_R\mathcal{L},\mathcal{AC}_R)$. 
\end{enumerate}
In \cite[dual results in \S 4]{Gillespie-DP}, it is shown that if $(\mathcal{L,A})$ is a complete duality pair, then $(\mathcal{GP}_{(\mathcal{L,A})},[\mathcal{GP}_{(\mathcal{L,A})}]^\perp)$ is a hereditary and complete cotorsion pair in ${}_R\mathsf{Mod}$. It is clear from the definition of $\mathcal{GP}_{(\mathcal{L,A})}$ that $\mathcal{GP}_{(\mathcal{L,A})} \cap [\mathcal{GP}_{(\mathcal{L,A})}]^\perp = {}_R\mathcal{P}$. It then follows by Theorem \ref{HoveyinC} that $[\mathcal{GP}_{(\mathcal{L,A})}]^\perp$ is thick (also stated in \cite[dual of Lem. 4.5]{Gillespie-DP}) and that for every $m \in \mathbb{Z}_{\geq 0}$ there is a hereditary abelian model structure represented by the Hovey triple
\[
([\mathcal{GP}_{(\mathcal{L,A})}]^\wedge_m,[\mathcal{GP}_{(\mathcal{L,A})}]^\perp,[{}_R\mathcal{P}^\wedge_m]^\perp).
\]
In particular, over any ring $R$ we obtain the hereditary abelian model structure on ${}_R\mathsf{Mod}$ represented by the Hovey triple
\[
([{}_R\mathcal{GP}_{\rm AC}]^\wedge_m,[{}_R\mathcal{GP}_{\rm AC}]^\perp,[{}_R\mathcal{P}^\wedge_m]^\perp).
\] 
The authors are not aware if these model structures were previously recorded in the literature besides the case $m = 0$ discovered in \cite[Thm. 8.5]{BravoGillespieHovey-arxiv}.

On the other hand, if $R$ is right coherent, then ${}_R\mathcal{F} = {}_R\mathcal{L}$ and $\mathcal{AP}_R = \mathcal{AC}_R$ by \cite[Corolls. 2.9 \& 2.11]{BravoGillespieHovey-arxiv}, and so the previous Hovey triple becomes
\[
([{}_R\mathcal{DP}]^\wedge_m,[{}_R\mathcal{DP}]^\perp,[{}_R\mathcal{P}^\wedge_m]^\perp),
\] 
generalizing thus the model structures mentioned in \cite[p. 1466]{WeiWangLiu-DingModels} (which were obtained over Ding-Chen rings). One can also note from \cite[Thm. A.6]{BravoGillespieHovey-arxiv} that ${}_R\mathcal{PGF} \supseteq {}_R\mathcal{GP}_{\rm AC}$ for any ring $R$, and that reverse containment holds when $R$ is right coherent. 
\end{enumerate}
\end{example}

\begin{remark}\label{rmk:Frob_pair_from_aproxs} ~\
\begin{enumerate}[(1)]
\item Let $(\X,\X^\perp)$ be a hereditary cotorsion pair in $\C$ such that every object in $\X^\wedge$ has a special $\X^\perp$-preenvelope. Then, for $\omega = \X \cap \X^\perp$ the following assertions hold:
\begin{enumerate}[(i)]
\item $(\X,\omega)$ is a Frobenius pair in $\C$.

\item If $\C$ has enough projective objects and $(\omega^\wedge_m, [\omega^\wedge_m]^{\perp_1})$ is a complete cotorsion pair in $\C$ for some $m \in \mathbb{Z}_{\geq 0}$, then assertions (a), (b) and (c) in Theorem \ref{HoveyinC} are equivalent. 
\end{enumerate}
Indeed, for (1) it suffices to check that $\omega$ is a relative cogenerator in $\X$, but this follows from the existence of special $\X^\perp$-preenvelopes for objects in $\X^\wedge$ and the fact that $\X$ is closed under extensions. Then, as $(\X,\omega)$ is a Frobenius pair, we can apply Lemma \ref{LemaGorro} (2) and mimic the proof of Theorem \ref{HoveyinC} to get the equivalence between (a), (b) and (c). 

\item In \cite[Thm. VI.2.1]{BeligiannisReiten}, the authors show that (a) and (c) in Theorem \ref{HoveyinC} are equivalent under different assumptions, namely, that $(\mathcal{X},\mathcal{X}^\perp)$ is a hereditary and complete cotorsion pair in an abelian category $\C$ with enough projective objects such that $\omega = \mathcal{X} \cap \mathcal{X}^\perp$ \textbf{functorially finite} (that is, precovering and preenveloping). The latter is not the case in general for $\omega = {}_R\mathcal{P}$. 

A complete cotorsion pair $(\mathcal{X},\mathcal{X}^{\perp_1})$ in an abelian category $\C$ with enough projective objects is called \textbf{projective} if $\mathcal{X} \cap \mathcal{X}^{\perp_1}$ is the class of projective objects and $\mathcal{X}^{\perp_1}$ is thick. Besides \cite{BeligiannisReiten}, these pairs have also been considered in other parts of the literature, like for instance in \cite{GillespieAdvances}, where the author proves a complete description for them in Proposition 3.7. 

\item Frobenius pairs $(\mathcal{X},\omega)$ for which $(\omega^\wedge_m,[\omega^\wedge_m]^{\perp_1})$ is a complete cotorsion pair will be also considered later in Proposition \ref{coro:Rep_omega_complete}.
\end{enumerate}
\end{remark}

The Gorenstein cotorsion pairs mentioned in Example \ref{ex:Gorenstein-models-Rmod} are of course projective. More examples can be found in the category of chain complexes.

\begin{example}\label{ex:pairs_in_Ch}
Let ${}_R\mathsf{Ch}$ denote the category of chain complexes of left $R$-modules. Given a cotorsion pair $(\mathcal{X},\mathcal{X}^{\perp_1})$, recall from \cite{GillespieFlat} the following notations: 
\begin{enumerate}[(i)]
\item ${\rm dw}\,\widetilde{\X}$ denotes the class of complexes $X_\bullet \in {}_R\mathsf{Ch}$ such that $X_k \in \X$ for every $k \in \mathbb{Z}$.

\item $\widetilde{\mathcal{X}}$ denotes the class of exact complexes with cycles in $\mathcal{X}$. 

\item ${\rm ex}\,\widetilde{\X}$ denotes the class of exact complexes in ${\rm dw}\,\widetilde{\X}$.

\item ${\rm dg}\,\widetilde{\X}$ denotes the class of complexes in $X_\bullet \in {\rm dw}\,\widetilde{\X}$ such that $\mathcal{H}om(X_{\bullet}, W_\bullet)$ is exact for every $W_\bullet \in \widetilde{\mathcal{X}^{\perp_1}}$. The definition of ${\rm dg}\,\widetilde{\X^{\perp_1}}$ is dual. Here, $\mathcal{H}om(X_{\bullet}, W_\bullet)$ denotes the complex of abelian groups with
\[
\mathcal{H}om(X_{\bullet}, W_\bullet)_n := \prod_{k \in \mathbb{Z}} \Hom_R(X_k,W_{n+k})
\]
and with differentials given by
\[
(f_k)_{k \in \mathbb{Z}} \mapsto [\partial_{n+k} \circ f_k - (-1)^n f_{k-1} \circ \partial_k]_{k \in \mathbb{Z}}
\] 
(see \cite[\S 2.1]{GRcomplexes} for details). 
\end{enumerate}
In \cite[Prop. 7.3]{GillespieAdvances}, the author shows that if $(\mathcal{X},\mathcal{X}^\perp)$ is a projective cotorsion pair cogenerated by some set (and so complete), then
\[
({\rm dw}\,\widetilde{\mathcal{X}}, [{\rm dw}\,\widetilde{\mathcal{X}}]^\perp), \, \, ({\rm ex}\,\widetilde{\mathcal{X}}, [{\rm ex}\,\widetilde{\mathcal{X}}]^\perp), \, \, (\widetilde{\mathcal{X}}, {\rm dg}\,\widetilde{\mathcal{X}^\perp}) \, \, \text{and} \, \, ({\rm dg}\widetilde{\mathcal{X}}, \widetilde{\mathcal{X}^\perp})
\] 
are also projective cotorsion pairs. Then, the intersections
\[
{\rm dw}\,\widetilde{\mathcal{X}} \cap [{\rm dw}\,\widetilde{\mathcal{X}}]^\perp, \, \, {\rm ex}\,\widetilde{\mathcal{X}} \cap [{\rm ex}\,\widetilde{\mathcal{X}}]^\perp, \, \, \widetilde{\mathcal{X}} \cap {\rm dg}\,\widetilde{\mathcal{X}^\perp} \, \, \text{and} \, \, {\rm dg}\widetilde{\mathcal{X}} \cap \widetilde{\mathcal{X}^\perp}
\]
coincide with the class of projective complexes of left $R$-modules. The latter class is given by $\widetilde{{}_R\mathcal{P}}$, the class of exact chain complexes with cycles in ${}_R\mathcal{P}$. By \cite[Thm. 7.4.6]{EnJen00} and \cite[Thm. 7.3.2]{EnJen00-2}, we have that $(\widetilde{{}_R\mathcal{P}^\wedge_m}, [\widetilde{{}_R\mathcal{P}^\wedge_m}]^\perp)$ is a hereditary and complete cotorsion pair in ${}_R\mathsf{Ch}$. On the other hand, note that $\widetilde{{}_R\mathcal{P}^\wedge_m} = [\widetilde{{}_R\mathcal{P}}]^\wedge_m$. We are thus under the assumptions of Theorem \ref{HoveyinC}, and hence we obtain the following Hovey triples in ${}_R\mathsf{Ch}$:
\begin{enumerate}[(i)]
\item $([{\rm dw}\,\widetilde{\mathcal{X}}]^\wedge_m, [{\rm dw}\,\widetilde{\mathcal{X}}]^\perp, [(\widetilde{{}_R\mathcal{P}})^\wedge_m]^\perp)$,

\item $([{\rm ex}\,\widetilde{\mathcal{X}}]^\wedge_m, [{\rm ex}\,\widetilde{\mathcal{X}}]^\perp, [(\widetilde{{}_R\mathcal{P}})^\wedge_m]^\perp)$,

\item $(\widetilde{\mathcal{X}}^\wedge_m, \widetilde{\mathcal{X}}^\perp, [(\widetilde{{}_R\mathcal{P}})^\wedge_m]^\perp)$, and

\item $([{\rm dg}\,\widetilde{\mathcal{X}}]^\wedge_m, [{\rm dg}\,\widetilde{\mathcal{X}}]^\perp, [(\widetilde{{}_R\mathcal{P}})^\wedge_m]^\perp)$.
\end{enumerate} 
In particular, one has the chain complex versions of the model structures mentioned in Example \ref{ex:Gorenstein-models-Rmod}.
\end{example}

In the previous model structures, the classes of cofibrant objects are given by objects with $\mathcal{X}$-resolution dimension uniformly bounded by $m$. In the following theorem, we consider two bounds $m$ and $n \geq m$, where the latter defines the trivial and fibrant objects. The following lemma is a consequence of Lemma \ref{LemaGorro} (2) and Proposition \ref{prop:Frobenius_stable} (2), along with the fact that $\X^\perp$ is closed under monocokernels.

\begin{lemma} \label{GruesaIzq}
Let $(\X, \omega)$ be a Frobenius pair and $m \in \mathbb{Z}_{\geq 0}$. For any short exact sequence $X \rightarrowtail Y \twoheadrightarrow Z$, if $X, Y \in \omega^\wedge_m$ and $Z \in \mathcal{X}^\wedge_m$, then $Z \in \omega^\wedge_m$. Furthermore, $\omega^\wedge_m$ is a right thick class in the w.i.c. exact subcategory $\X^\wedge_m$.
\end{lemma}

\begin{theorem}[method 2]\label{thm:models_two_bounds}
Let $(\mathcal{X},\mathcal{X}^\perp)$ be a hereditary cotorsion pair in $\C$ (not necessarily with enough projective objects) such that every object in $\mathcal{X}^\wedge$ has a special $\mathcal{X}$-precover and a special $\X^\perp$-preenvelope.  For $\omega = \mathcal{X} \cap \mathcal{X}^\perp$, if $(\omega^\wedge_m, [\omega^\wedge_m]^{\perp_1})$ is a complete cotorsion pair in $\C$ for some $m \in \mathbb{Z}_{\geq 0}$, then for every $n \geq m$ there is a hereditary Hovey triple in the w.i.c. exact subcategory $\mathcal{X}^\wedge_n$ given by $$(\X^\wedge_m, \omega^\wedge_n, [\omega^\wedge_m]^{\perp_1} \cap \X^\wedge_n),$$ provided that $\omega^\wedge_n$ is closed under epikernels. The homotopy category of its associated hereditary abelian model structure is equivalent to the quotient category $$\X^\wedge_m \cap [\omega^\wedge_m]^{\perp_1} / \sim,$$ where for two morphisms $f, g \colon X \to Y$ with $X, Y \in \X^\wedge_m \cap [\omega^\wedge_m]^{\perp_1}$ one has that $f \sim g$ if, and only if, $f - g$ factors through a trivially bifibrant object (that is, and object in $\omega^\wedge_m \cap [\omega^\wedge_m]^{\perp_1}$). 
\end{theorem}

\begin{proof}
The idea is to show that $(\omega^\wedge_m, [\omega^\wedge_m]^{\perp_1} \cap \mathcal{X}^\wedge_n)$ and $(\X^\wedge_m, [\omega^{\wedge}_m]^{\perp_1} \cap \X^\perp \cap \mathcal{X}^\wedge_n)$ are compatible and hereditary complete cotorsion pairs in $\mathcal{X}^\wedge_n$. 
\begin{itemize}
\item $(\omega^\wedge_m, [\omega^\wedge_m]^{\perp_1} \cap \mathcal{X}^\wedge_n)$ is a complete cotorsion pair in $\mathcal{X}^\wedge_n$: Since $(\omega^\wedge_m, [\omega^\wedge_m]^{\perp_1})$ is a complete cotorsion pair in $\C$, setting $\mathcal{X} := \mathcal{X}$, $\mathcal{Y} := \omega^\wedge_m$ and $\mathcal{B} := \mathcal{X}^\wedge_n$ in Lemma \ref{RK-H-C2} yields the assertion. Recall that $\mathcal{X}^\wedge_n$ is left thick by Proposition \ref{prop:Frobenius_stable} (2).

\item $(\X^\wedge_m, [\omega^{\wedge}_m]^{\perp_1} \cap \X^\perp \cap \mathcal{X}^\wedge_n)$ is a complete cotorsion pair in $\mathcal{X}^\wedge_n$: By Remark \ref{rmk:Frob_pair_from_aproxs} (1) we have that $(\X,\omega)$ is a Frobenius pair, so by Theorem \ref{Family} (2), $(\X^\wedge_m, [\omega^{\wedge}_m]^{\perp_1} \cap \X^\perp)$ is a cotorsion pair cut along $\mathcal{X}^\wedge$. From this, we can prove our assertion:
\begin{itemize}
\item $\mathcal{X}^\wedge_m = {}^{\perp_1}([\omega^\wedge_m]^{\perp_1} \cap \mathcal{X}^\perp \cap \mathcal{X}^\wedge_n) \cap \mathcal{X}^\wedge_n$: The containment ($\subseteq$) is clear since $\mathcal{X}^\wedge_m = {}^{\perp_1}([\omega^\wedge_m]^{\perp_1} \cap \mathcal{X}^\perp) \cap \mathcal{X}^\wedge$ and $n \geq m$. Now let $M$ be an object in ${}^{\perp_1}([\omega^\wedge_m]^{\perp_1} \cap \mathcal{X}^\perp \cap \mathcal{X}^\wedge_n) \cap \mathcal{X}^\wedge_n$. Since $M \in \mathcal{X}^\wedge$ and $(\X^\wedge_m, [\omega^{\wedge}_m]^{\perp_1} \cap \X^\perp)$ is a cotorsion pair cut along $\mathcal{X}^\wedge$, there is an exact sequence $K \rightarrowtail L \twoheadrightarrow M$ with $L \in \mathcal{X}^\wedge_m$ and $K \in [\omega^\wedge_m]^{\perp_1} \cap \mathcal{X}^\perp$. Now since $\mathcal{X}^\wedge_n$ is left thick and $M \in \mathcal{X}^\wedge_n$ and $L \in \mathcal{X}^\wedge_m \subseteq \mathcal{X}^\wedge_n$, we have that $K \in [\omega^\wedge_m]^{\perp_1} \cap \mathcal{X}^\perp \cap \mathcal{X}^\wedge_n$. On the other hand, $M \in {}^{\perp_1}([\omega^\wedge_m]^{\perp_1} \cap \mathcal{X}^\perp \cap \mathcal{X}^\wedge_n)$, and so the previous sequence splits, which in turn implies that $M \in \mathcal{X}^\wedge_m$.

\item $[\omega^\wedge_m]^{\perp_1} \cap \mathcal{X}^\perp \cap \mathcal{X}^\wedge_n = [\mathcal{X}^\wedge_m]^{\perp_1} \cap \mathcal{X}^\wedge_n$: Since $(\X^\wedge_m, [\omega^{\wedge}_m]^{\perp_1} \cap \X^\perp)$ is a cotorsion pair cut along $\mathcal{X}^\wedge$, we get $[\X^\wedge_m]^{\perp_1} \cap \X^\wedge = [\omega^{\wedge}_m]^{\perp_1} \cap \X^\perp \cap \X^\wedge$. Then, intersection each side of the previous equality with $\X^\wedge_n$ yields the result. 

\item $(\X^\wedge_m, [\omega^{\wedge}_m]^{\perp_1} \cap \X^\perp \cap \mathcal{X}^\wedge_n)$ is complete (in $\mathcal{X}^\wedge_n$): Follows from the facts that $(\X^\wedge_m, [\omega^{\wedge}_m]^{\perp_1} \cap \X^\perp)$ is a cotorsion pair cut along $\mathcal{X}^\wedge$ and that $\mathcal{X}^\wedge_n$ is left thick. 
\end{itemize}

\item Compatibility: The equalities $\omega^\wedge_m = \X^\wedge_m \cap \omega^\wedge_n$ and $[\omega^{\wedge}_m]^{\perp_1} \cap \X^\perp \cap \mathcal{X}^\wedge_n = ([\omega^\wedge_m]^{\perp_1} \cap \mathcal{X}^\wedge_n) \cap \omega^\wedge_n$ follow from Lemma \ref{LemaGorro} (2). Finally, $\omega^\wedge_n$ is thick in $\mathcal{X}^\wedge_n$ by Lemma \ref{GruesaIzq} and the assumption that $\omega^\wedge_n$ is closed under epikernels.

\item Hereditariness: Note that $\omega^\wedge_m = \X^\wedge_m \cap \omega^\wedge_n$ is closed under epikernels by the mentioned properties of $\omega^\wedge_n$ and the fact that $\X^\wedge_m$ is closed under epikernels. So the hereditariness of $(\omega^\wedge_m, [\omega^\wedge_m]^{\perp_1} \cap \mathcal{X}^\wedge_n)$ and $(\X^\wedge_m, [\omega^{\wedge}_m]^{\perp_1} \cap \X^\perp \cap \mathcal{X}^\wedge_n)$ follows from \cite[Lem. 6.17]{Stovi14}.

\item Homotopy category: Follows by \cite[Prop. 5.2, Corolls. 4.8 \& 5.4]{GillespieExact}.
\end{itemize}
\end{proof}

\begin{example}\label{ex:method2} ~\
\begin{enumerate}
\item Consider the hereditary cotorsion pair $({}_R\mathcal{GP},{}_R\mathcal{GP}^\perp)$ in ${}_R\mathsf{Mod}$. By \cite[Thm. 2.8 \& Prop. 6.1]{BMSP}, every left $R$-module in ${}_R\mathcal{GP}^\wedge$ has a special Gorenstein projective precover. We are then under the assumptions of Theorem \ref{thm:models_two_bounds}, and hence for every $n \geq m \geq 0$ there is a hereditary Hovey triple 
\[
({}_R\mathcal{GP}^\wedge_m, {}_R\mathcal{P}^\wedge_n, [{}_R\mathcal{P}^\wedge_m]^{\perp_1} \cap {}_R\mathcal{GP}^\wedge_n)
\]
in the w.i.c. exact subcategory ${}_R\mathcal{GP}^\wedge_n$, obtaining thus \cite[Thm. 6.3 (2)]{Chains}. 

\item Similar to the previous example, if we consider the hereditary and complete cotorsion pair $({}_R\mathcal{PGF},{}_R\mathcal{PGF}^\perp)$ in ${}_R\mathsf{Mod}$, the hypothesis in Theorem \ref{thm:models_two_bounds} are clearly satisfied by this pair, and hence for every $n \geq m \geq 0$ there is a hereditary Hovey triple 
\[
({}_R\mathcal{PGF}^\wedge_m, {}_R\mathcal{P}^\wedge_n, [{}_R\mathcal{P}^\wedge_m]^{\perp_1} \cap {}_R\mathcal{PGF}^\wedge_n)
\]
in the w.i.c. exact subcategory ${}_R\mathcal{PGF}^\wedge_n$, obtaining thus \cite[Coroll. 5.3 (3)]{Chains}. In a similar way, one can obtain such Hovey triples using the hereditary and complete cotorsion pairs from Example \ref{ex:Gorenstein-models-Rmod} (2) or Example \ref{ex:pairs_in_Ch}.
\end{enumerate}
\end{example}

In some cases, it is possible to extend the ambient w.i.c. exact category $\mathcal{X}^\wedge_n$ to a bigger category $\mathcal{Y}^\wedge_n$, provided that the hereditary cotorsion pair $(\X,\X^\perp)$ is somehow compatible with a Frobenius pair of the form $(\Y,\X \cap \X^\perp)$ with $\X \subseteq \Y$.

\begin{theorem}[method 3]\label{thm:models_two_bounds-2}
Let $(\mathcal{X},\mathcal{X}^\perp)$ be a hereditary and complete cotorsion pair in $\C$ (not necessarily with enough projective objects) and $(\mathcal{Y},\omega)$ be a Frobenius pair in $\C$ such that:
\begin{itemize}
\item $\mathcal{X}^\perp$ is thick,
\item $\mathcal{X} \subseteq \mathcal{Y}$, and
\item $\mathcal{X} \cap \mathcal{X}^\perp = \omega$.
\end{itemize}  
If $(\omega^\wedge_m, [\omega^\wedge_m]^{\perp_1})$ is a complete cotorsion pair in $\C$ for some $m \in \mathbb{Z}_{\geq 0}$, then for every $n \geq m$ there is a hereditary Hovey triple $$(\mathcal{X}^\wedge_m,\mathcal{X}^\perp \cap \mathcal{Y}^\wedge_n, [\omega^\wedge_m]^{\perp_1} \cap \mathcal{Y}^\wedge_n)$$ in the w.i.c. exact subcategory $\mathcal{Y}^\wedge_n$. The homotopy category of its associated hereditary abelian model structure is equivalent to the quotient category $$\X^\wedge_m \cap [\omega^\wedge_m]^{\perp_1} / \sim,$$ where for two morphisms $f, g \colon X \to Y$ with $X, Y \in \X^\wedge_m \cap [\omega^\wedge_m]^{\perp_1}$ one has that $f \sim g$ if, and only if, $f - g$ factors through a trivially bifibrant object (that is, and object in $\omega^\wedge_m \cap [\omega^\wedge_m]^{\perp_1}$). 
\end{theorem}

\begin{proof}
Let us show that $(\mathcal{X}^\wedge_m, [\omega^\wedge_m]^{\perp_1} \cap \mathcal{X}^\perp \cap \mathcal{Y}^\wedge_n)$ and $(\omega^\wedge_m,[\omega^\wedge_m]^{\perp_1} \cap \mathcal{Y}^\wedge_n)$ are compatible hereditary and complete cotorsion pairs in $\mathcal{Y}^\wedge_n$. By Theorem \ref{Family} (1), we first note that $(\mathcal{X}^\wedge_m, [\omega^\wedge_m]^{\perp_1} \cap \mathcal{X}^\perp)$ is a hereditary and complete cotorsion pair in $\C$, and so by Lemma \ref{RK-H-C2} we get that $(\mathcal{X}^\wedge_m, [\omega^\wedge_m]^{\perp_1} \cap \mathcal{X}^\perp \cap \mathcal{Y}^\wedge_n)$ is a hereditary and complete cotorsion pair in $\mathcal{Y}^\wedge_n$ after setting $\mathcal{Y} := \mathcal{X}^\wedge_m$, $\mathcal{X} := \C$ and $\mathcal{B} := \mathcal{Y}^\wedge_n$ in the statement of \ref{RK-H-C2}. Similarly, if we set there $\mathcal{Y} := \omega^\wedge_m$, $\mathcal{X} := \C$ and $\mathcal{B} := \mathcal{Y}^\wedge_n$, we have that $(\omega^\wedge_m,[\omega^\wedge_m]^{\perp_1} \cap \mathcal{Y}^\wedge_n)$ is a complete cotorsion pair in $\mathcal{Y}^\wedge_n$. Moreover, $\omega^\wedge_m = \X^\wedge_m \cap \X^\perp$ is left thick in $\C$ since so are $\X^\wedge_m$ and $\X^\perp$, by Proposition \ref{prop:Frobenius_stable} and the assumption on $\X^\perp$. Now, $\mathcal{Y}^\wedge_n$ is also left thick in $\C$, we have that $\omega^\wedge_m$ is closed under kernels of admissible epimorphisms in $\mathcal{Y}^\wedge_n$, and so by by Lemma \ref{RK-H-C2} we have that $(\omega^\wedge_m,[\omega^\wedge_m]^{\perp_1} \cap \mathcal{Y}^\wedge_n)$ is also hereditary in $\mathcal{Y}^\wedge_n$.

The compatibility of $(\mathcal{X}^\wedge_m, [\omega^\wedge_m]^{\perp_1} \cap \mathcal{X}^\perp \cap \mathcal{Y}^\wedge_n)$ and $(\omega^\wedge_m,[\omega^\wedge_m]^{\perp_1} \cap \mathcal{Y}^\wedge_n)$ is easy to note from Lemma \ref{LemaGorro} (2), so it remains to show that $\mathcal{X}^\perp \cap \mathcal{Y}^\wedge_n$ is thick in $\mathcal{Y}^\wedge_n$. Clearly, $\mathcal{X}^\perp \cap \mathcal{Y}^\wedge_n$ is closed under extensions and direct summands in $\C$, and so the same closure properties carry over to $\mathcal{Y}^\wedge_n$. The remaining closure properties for thickness follow from the assumption that $\mathcal{X}^\perp$ is thick in $\C$. The assertions regarding the homotopy category follow by \cite[Prop. 5.2, Corolls. 4.8 \& 5.4]{GillespieExact}. \\
\end{proof}

\begin{example}\label{ex:Ch-models-containments} ~\
\begin{enumerate}[(1)]
\item Over an arbitrary ring $R$, we know that the containments ${}_R\mathcal{GP}_{\rm AC} \subseteq {}_R\mathcal{PGF}$ and ${}_R\mathcal{DP} \subseteq {}_R\mathcal{GP}$ hold. Moreover, it is shown in \cite[Coroll. 1]{Alina2020} that ${}_R\mathcal{PGF} \subseteq {}_R\mathcal{DP}$, also over any ring $R$. On the one hand, $({}_R\mathcal{GP}_{\rm AC},[{}_R\mathcal{GP}_{\rm AC}]^\perp)$ and $({}_R\mathcal{PGF},[{}_R\mathcal{PGF}]^\perp)$ are hereditary and complete cotorsion pairs in ${}_R\mathsf{Mod}$ with the same kernel $\omega = {}_R\mathcal{P}$, with $[{}_R\mathcal{GP}_{\rm AC}]^\perp$ and $[{}_R\mathcal{PGF}_{\rm AC}]^\perp$ thick. On the other hand, $({}_R\mathcal{DP},{}_R\mathcal{P})$ and $({}_R\mathcal{GP},{}_R\mathcal{P})$ are Frobenius pairs in ${}_R\mathsf{Mod}$. We then obtain from Theorem \ref{thm:models_two_bounds-2} the Hovey triples
\[
([{}_R\mathcal{GP}_{\rm AC}]^\wedge_m, [{}_R\mathcal{GP}_{\rm AC}]^\perp \cap [{}_R\mathcal{PGF}]^\wedge_n, [{}_R\mathcal{P}^\wedge_m]^{\perp_1} \cap [{}_R\mathcal{PGF}]^\wedge_n),
\] 
\[
([{}_R\mathcal{GP}_{\rm AC}]^\wedge_m, [{}_R\mathcal{GP}_{\rm AC}]^\perp \cap [{}_R\mathcal{DP}]^\wedge_n, [{}_R\mathcal{P}^\wedge_m]^{\perp_1} \cap [{}_R\mathcal{DP}]^\wedge_n),
\] 
\[
([{}_R\mathcal{GP}_{\rm AC}]^\wedge_m, [{}_R\mathcal{GP}_{\rm AC}]^\perp \cap [{}_R\mathcal{GP}]^\wedge_n, [{}_R\mathcal{P}^\wedge_m]^{\perp_1} \cap [{}_R\mathcal{GP}]^\wedge_n),
\] 
in the w.i.c. exact subcategories $[{}_R\mathcal{PGF}]^\wedge_n$, $[{}_R\mathcal{DP}]^\wedge_n$ and $[{}_R\mathcal{GP}]^\wedge_n$, respectively, and 
\[
([{}_R\mathcal{PGF}]^\wedge_m, [{}_R\mathcal{PGF}]^\perp \cap [{}_R\mathcal{DP}]^\wedge_n, [{}_R\mathcal{P}^\wedge_m]^{\perp_1} \cap [{}_R\mathcal{DP}]^\wedge_n),
\] 
\[
([{}_R\mathcal{PGF}]^\wedge_m, [{}_R\mathcal{PGF}]^\perp \cap [{}_R\mathcal{GP}]^\wedge_n, [{}_R\mathcal{P}^\wedge_m]^{\perp_1} \cap [{}_R\mathcal{GP}]^\wedge_n),
\] 
in the w.i.c. exact subcategories $[{}_R\mathcal{DP}]^\wedge_n$ and $[{}_R\mathcal{GP}]^\wedge_n$, respectively. In particular, we get \cite[Coroll. 6.1 (3)]{Chains}.

\item Let $(\X,\X^\perp)$ be a projective cotorsion pair in ${}_R\mathsf{Mod}$, and consider the induced projective cotorsion pairs in ${}_R\mathsf{Ch}$ given in Example \ref{ex:pairs_in_Ch}. We have a hereditary exact model structure for each of the following containments:
\[
\widetilde{\X} \stackrel{\scriptsize\textcircled{1}}\subseteq {\rm ex}\,\widetilde{\X}, \, \, {\rm ex}\,\widetilde{\X} \stackrel{\scriptsize\textcircled{2}}\subseteq {\rm dw}\,\widetilde{\X}, \, \, \widetilde{\X} \stackrel{\scriptsize\textcircled{3}}\subseteq {\rm dw}\,\widetilde{\X}, \, \, \widetilde{\X} \stackrel{\scriptsize\textcircled{4}}\subseteq {\rm dg}\,\widetilde{\X} \, \, \text{and} \, \, {\rm dg}\,\widetilde{\X} \stackrel{\scriptsize\textcircled{5}}\subseteq {\rm dw}\,\widetilde{\X}.
\]
Indeed, consider for instance ${\scriptsize\textcircled{1}}$. The hereditary and complete cotorsion pair $(\widetilde{\X},\widetilde{\X}^\perp)$ and the Frobenius pair $({\rm ex}\,\widetilde{\X},\widetilde{{}_R\mathcal{P}})$ in ${}_R\mathsf{Ch}$ satisfy the assumptions of Theorem \ref{thm:models_two_bounds-2}. Then, for every pair of integers $(m,n)$ with $n \geq m \geq 0$ there exists a hereditary Hovey triple
\[
(\widetilde{\X}^\wedge_m,\widetilde{\X}^\perp \cap [{\rm ex}\,\widetilde{\X}]^\wedge_n,[(\widetilde{{}_R\mathcal{P}})^\wedge_m]^\perp \cap [{\rm ex}\,\widetilde{\X}]^\wedge_n)
\]
in the w.i.c. exact subcategory $[{\rm ex}\,\widetilde{\X}]^\wedge_n$.
\end{enumerate} 
\end{example}

We know that there are hereditary and complete cotorsion pairs $(\X,\X^\perp)$ with $\omega = \X \cap \X^\perp$ for which $(\omega^\wedge_m,[\omega^\wedge_m]^{\perp_1})$ is not necessarily a complete cotorsion pair for any $m \in \mathbb{Z}_{\geq 0}$, although its induced Frobenius pair $(\X,\omega)$ is indeed $\omega^\wedge_m$-complete for some $m \in \mathbb{Z}_{\geq 0}$ (see Example \ref{ExGF} (2)). Under certain conditions, it is still possible to obtain hereditary abelian model structures in this situation. The following result specifies these conditions.

\begin{theorem}[method 4]\label{thm:model_two_cotorsion_pairs}
Let $(\mathcal{X},\mathcal{X}^\perp)$ and $(\mathcal{Y},\mathcal{Y}^\perp)$ be hereditary and complete cotorsion pairs in $\C$ such that  $\mathcal{X} \subseteq \mathcal{Y}$ and $\mathcal{X} \cap \mathcal{X}^\perp = \mathcal{Y} \cap \mathcal{Y}^\perp$. Suppose that for $\omega = \mathcal{X} \cap \mathcal{X}^\perp$ there is an integer $m \in \mathbb{Z}_{>0}$ such that the induced Frobenius pairs $(\mathcal{X},\omega)$ and $(\mathcal{Y},\omega)$ are $\omega^\wedge_m$-complete in $\mathcal{X}^\perp$ and $\mathcal{Y}^\perp$, respectively. Then, for every $n \geq m$ there is a hereditary Hovey triple 
\[
(\mathcal{Y}^\wedge_m, \mathcal{X}^\wedge_n, [\mathcal{X}^\wedge_m]^{\perp_1} \cap \mathcal{Y}^\wedge_n)
\] 
in the w.i.c. exact subcategory $\mathcal{Y}^\wedge_n$. The homotopy category of its associated hereditary abelian model structure is equivalent to the quotient category 
\[
\mathcal{Y}^\wedge_m \cap [\mathcal{X}^\wedge_m]^{\perp_1} / \sim,
\] 
where for two morphisms $f, g \colon X \to Y$ with $X, Y \in \mathcal{Y}^\wedge_m \cap [\mathcal{X}^\wedge_m]^{\perp_1}$ one has that $f \sim g$ if, and only if, $f - g$ factors through a trivially bifibrant object (that is, an object in $\mathcal{X}^\wedge_m \cap [\mathcal{X}^\wedge_m]^{\perp_1} = \omega^\wedge_m \cap [\omega^\wedge_m]^{\perp_1}$). 
\end{theorem}

\begin{proof}
By Theorem \ref{Family}, $(\X^{\wedge}_m, [\omega^{\wedge}_m]^{\perp_1} \cap \X^\perp)$ and $(\mathcal{Y}^{\wedge}_m, [\omega^{\wedge}_m]^{\perp_1} \cap \mathcal{Y}^\perp)$ are hereditary and complete cotorsion pair in $\C$, and so by Lemmas \ref{RK-H-C2} and \ref{LemaGorro} (2) we have that the restrictions 
\[
(\X^{\wedge}_m, [\omega^{\wedge}_m]^{\perp_1} \cap \X^\perp \cap \mathcal{Y}^\wedge_n) \, \, \text{and} \, \, (\mathcal{Y}^{\wedge}_m, [\omega^{\wedge}_m]^{\perp_1} \cap \mathcal{Y}^\perp \cap \mathcal{Y}^\wedge_n) = (\mathcal{Y}^{\wedge}_m, [\omega^{\wedge}_m]^{\perp_1} \cap \omega^\wedge_n)
\] 
are hereditary and complete cotorsion pairs in $\mathcal{Y}^\wedge_n$. 

Let us show that 
\[
\mathcal{X}^\wedge_m = \mathcal{X}^\wedge_n \cap \mathcal{Y}^\wedge_m.
\]
The containment ($\subseteq$) is clear. Now let $M \in \mathcal{X}^\wedge_n \cap \mathcal{Y}^\wedge_m$. Then, by \cite[Thm. 2.10]{BMSP} we have that 
\[
\resdim_{\X}(M) = \pd_{\omega}(M) = \resdim_{\mathcal{Y}}(M) \leq m,
\]
and so $M \in \mathcal{X}^\wedge_m$. On the other hand, again by Lemma \ref{LemaGorro} (2) we obtain
\[
([\omega^{\wedge}_m]^{\perp_1} \cap \X^\perp \cap \mathcal{Y}^\wedge_n) \cap \mathcal{X}^\wedge_n = [\omega^{\wedge}_m]^{\perp_1} \cap \X^\perp \cap \mathcal{X}^\wedge_n = [\omega^{\wedge}_m]^{\perp_1} \cap \omega^\wedge_n.
\]
Finally, it remains to show that $\mathcal{X}^\wedge_n$ is thick in $\mathcal{Y}^\wedge_n$. We already know that $\mathcal{X}^\wedge_n$ is left thick in $\C$, and so left thick in $\mathcal{Y}^\wedge_n$. Right thickness in $\mathcal{Y}^\wedge_n$ follows by \cite[Thms. 2.10 \& 2.11]{BMSP}. 

From the previous, we can note that 
\[
\mathcal{Y}^\wedge_m \cap \mathcal{X}^\wedge_n \cap [\mathcal{X}^\wedge_m]^{\perp_1} \cap \mathcal{Y}^\wedge_n = \mathcal{X}^\wedge_m \cap [\mathcal{X}^\wedge_m]^{\perp_1} = \mathcal{X}^\wedge_m \cap [\omega^{\wedge}_m]^{\perp_1} \cap \X^\perp = \omega^{\wedge}_m \cap [\omega^{\wedge}_m]^{\perp_1},
\]
and hence in the Hovey triple $(\mathcal{Y}^\wedge_m, \mathcal{X}^\wedge_n, [\mathcal{X}^\wedge_m]^{\perp_1} \cap \mathcal{Y}^\wedge_n)$ in $\mathcal{Y}^\wedge_n$, the bifibrant trivial objects in the associated model structure are given by $\omega^{\wedge}_m \cap [\omega^{\wedge}_m]^{\perp_1}$. \\
\end{proof}

\begin{example}\label{ex:method4} ~\
\begin{enumerate}[(1)]
\item Let $({}_R\mathcal{F},{}_R\mathcal{C})$ and $({}_R\mathcal{GF},{}_R\mathcal{GC})$ be the flat and Gorenstein flat (hereditary and complete) cotorsion pairs in ${}_R\mathsf{Mod}$. We know from Example \ref{ExGF} (2) that the induced Frobenius pairs $({}_R\mathcal{F},{}_R\mathcal{F} \cap {}_R\mathcal{C})$ and $({}_R\mathcal{GF},{}_R\mathcal{F} \cap {}_R\mathcal{C})$ are $({}_R\mathcal{F} \cap {}_R\mathcal{C})^\wedge_m$-complete in ${}_R\mathcal{C}$ and ${}_R\mathcal{GC}$, respectively, for every $m \in \mathbb{Z}_{\geq 0}$. By Theorem \ref{thm:model_two_cotorsion_pairs}, we have the hereditary Hovey triple
\[
({}_R\mathcal{GF}^\wedge_m, {}_R\mathcal{F}^\wedge_n, [{}_R\mathcal{F}^\wedge_m]^{\perp_1} \cap {}_R\mathcal{GF}^\wedge_n)
\]
in the w.i.c. exact subcategory ${}_R\mathcal{GF}^\wedge_n$ for every pair of nonnegative integers $(m,n)$ with $n \geq m$, obtained thus the model structure mentioned in \cite[Coroll. 7.7.]{Chains}.

From Examples \ref{ExGF} (3) and (4), we have similar Gorenstein flat exact model structures in the contexts of quasi-coherent sheaves over a semi-separated noetherian scheme $X$, and of module-valued representations of a left rooted quiver $Q$. In particular, we have a partial generalization of \cite[Thm. C]{DiEstradaLiangOdabasi}.

\item Applying Theorem \ref{thm:model_two_cotorsion_pairs} to $({}_R\mathcal{GP}_{\rm AC},[{}_R\mathcal{GP}_{\rm AC}]^\perp)$ and $({}_R\mathcal{PGF},[{}_R\mathcal{PGF}]^\perp)$, we get the hereditary Hovey triple 
\[
([{}_R\mathcal{PGF}]^\wedge_m, [{}_R\mathcal{GP}_{\rm AC}]^\wedge_n, ([{}_R\mathcal{GP}_{\rm AC}]^\wedge_m)^{\perp_1} \cap [{}_R\mathcal{PGF}]^\wedge_n)
\] 
in the w.i.c. exact subcategory $[{}_R\mathcal{PGF}]^\wedge_n$, for every pair of integers $n \geq m \geq 0$.

\item As in Example \ref{ex:Ch-models-containments} (2), we can obtain a hereditary exact model structure for the containments ${\scriptsize\textcircled{1}}$, ${\scriptsize\textcircled{2}}$, ${\scriptsize\textcircled{3}}$, ${\scriptsize\textcircled{4}}$ and ${\scriptsize\textcircled{5}}$. For instance, if we consider ${\scriptsize\textcircled{1}}$, we can apply Theorem \ref{thm:model_two_cotorsion_pairs} to the cotorsion pairs $(\widetilde{\X},\widetilde{\X}^\perp)$ and $({\rm ex}\,\widetilde{\X},[{\rm ex}\,\widetilde{\X}]^\perp)$, since they are projective. Then, for every pair of nonnegative integers $(m,n)$ with $n \geq m \geq 0$ we have the hereditary Hovey triple
\[
([{\rm ex}\,\widetilde{\X}]^\wedge_m, \widetilde{\X}^\wedge_n, [\widetilde{\X}^\wedge_m]^{\perp_1} \cap [{\rm ex}\,\widetilde{\X}]^\wedge_n)
\] 
in $[{\rm ex}\,\widetilde{\X}]^\wedge_n$. 
\end{enumerate}
\end{example}

%%%%%%%%%%%%%%%%%%%%%%%%%%%%%%%%%%%%%%%%%%%%%%%%
%%%%%%%%%%%%%%%%%%%%%%%%%%%%%%%%%%%%%%%%%%%%%%%%

\subsection*{Model structures from generalized Gorenstein modules}

There are other model structures involving Gorenstein flat modules and Gorenstein flat dimensions which are not covered by the methods presented in the previous section, like for instance \cite[Thm. B]{Rachid}. There, the author shows the existence of a hereditary Hovey triple $({}_R\mathcal{GF}^\wedge_m, [{}_R\mathcal{PGF}]^\perp, [{}_R\mathcal{F}^\wedge_m]^\perp)$ in ${}_R\mathsf{Mod}$ for every $m \in \mathbb{Z}_{\geq 0}$. In this section with generalize this and other results from \cite{Rachid} to the class of Gorenstein flat left $R$-modules relative to a class $\mathcal{A} \subseteq \mathsf{Mod}_R$. 

We recall from \cite[Def. 2.1]{EstradaMarco} that a left $R$-module $M \in {}_R\mathsf{Mod}$ is \textbf{(projectively coresolved) Gorenstein $\bm{\mathcal{A}}$-flat} if there exists an exact complex $F_\bullet$ of (projective) flat left $R$-modules such that $A \otimes_R F_\bullet$ is an exact complex of abelian groups for every $A \in \mathcal{A}$. Let us denote by (${}_R\mathcal{PGF}_{\mathcal{A}}$) ${}_R\mathcal{GF}_{\mathcal{A}}$ the class of (projectively coresolved) Gorenstein $\A$-flat left $R$-modules. In particular, if $\A$ is the right half of a complete duality pair $(\mathcal{L,A})$, then ${}_R\mathcal{GF}_{\mathcal{A}}$ is the class of Gorenstein $(\mathcal{L,A})$-flat modules from \cite[Def. 5.1]{Gillespie-DP}. The following is a consequence of \cite{EstradaMarco}.

\begin{example}\label{ex:FP-Gflat-relative}
If ${}_R\mathcal{PGF}_{\mathcal{A}}$ is the left half of a hereditary cotorsion pair in ${}_R\mathsf{Mod}$ cogenerated by a set\footnote{This occurs for instance when $\A$ is \emph{semi-definable} (see \cite[Def. 2.7]{EstradaMarco}).}, then by \cite[Thm. 2.14]{EstradaMarco} we have that ${}_R\mathcal{GF}_{\mathcal{A}}$ is closed under extensions\footnote{In \cite[Thm. 2.14]{EstradaMarco} one needs that $\A$ is semi-definable, but the arguments given in the proof only require that ${}_R\mathcal{PGF}_{\A}$ is the left half of a hereditary and complete cotorsion pair in ${}_R\mathsf{Mod}$.}, and so $({}_R\mathcal{GF}_{\mathcal{A}},[{}_R\mathcal{GF}_{\mathcal{A}}]^\perp)$ is a hereditary cotorsion pair in ${}_R\mathsf{Mod}$ cogenerated by a set (and so complete), with ${}_R\mathcal{GF}_{\mathcal{A}} \cap [{}_R\mathcal{GF}_{\mathcal{A}}]^\perp = {}_R\mathcal{F} \cap {}_R\mathcal{C}$ (see \cite[Coroll. 2.20 \& Prop. 3.1]{EstradaMarco}). It then follows that $({}_R\mathcal{GF}_{\mathcal{A}},{}_R\mathcal{F} \cap {}_R\mathcal{C})$ is a Frobenius pair. Following the arguments in Example \ref{ExGF} (2), we can also note that this Frobenius pair is $({}_R\mathcal{F} \cap {}_R\mathcal{C})^\wedge_m$-complete for every $m \in \mathbb{Z}_{\geq 0}$. 
\end{example}

From the previous example and Proposition \ref{prop:Frobenius_stable} (1), we have that $\resdim_{{}_R\mathcal{GF}_{\A}}(-)$ is stable if ${}_R\mathcal{PGF}_{\mathcal{A}}$ is the left half of a hereditary cotorsion pair in ${}_R\mathsf{Mod}$ cogenerated by a set. The following result characterizes $\resdim_{{}_R\mathcal{GF}_{\A}}(-)$, and it is a simultaneous generalization of \cite[Thm. 2.14]{EstradaMarco} and \cite[Thm. 3.4]{Rachid}. Its proof follows using the arguments given in \cite[Thm. 3.4]{Rachid} but adapted to ${}_R\mathcal{GF}_{\A}$ and ${}_R\mathcal{PGF}_{\A}$\footnote{Keep in mind that ${}_R\mathcal{F}^\wedge_m \subseteq [{}_R\mathcal{PGF}]^\perp \subseteq [{}_R\mathcal{PGF}_{\A}]^\perp$ and that $[{}_R\mathcal{PGF}_{\A}]^\perp$ is thick by \cite[Proof of Thm. 3.2]{EstradaMarco}.}. \\

\begin{proposition} \label{General-Gillespie}
Let $\A \subseteq \mathsf{Mod}_R$ be a class containing the injective right $R$-modules, such that ${}_R\mathcal{PGF}_{\mathcal{A}}$ is the left half of a (hereditary) cotorsion pair in ${}_R\mathsf{Mod}$ cogenerated by a set. Then, for every $M \in {}_R\mathsf{Mod}$ and $m \in \mathbb{Z}_{\geq 0}$, the following statements are equivalent:
\begin{enumerate}[(a)]
\item $\resdim _{\GF _{\A}} (M) \leq m$.

\item There is an exact and $\Hom_R(-,{}_R\mathcal{F}^\wedge_m \cap [{}_R\mathcal{F}^\wedge_m]^\perp)$-exact sequence $ K \rightarrowtail L \twoheadrightarrow M $ of left $R$-modules with $K \in {}_R\mathcal{F}^\wedge_m$ and $L \in {}_R\mathcal{PGF}_{\mathcal{A}}$.

\item $\Ext_R^1(M, E) = 0$ for every $E \in [{}_R\mathcal{PGF}_{\mathcal{A}}]^\perp \cap [{}_R\mathcal{F}^\wedge_m]^\perp$. 

\item There is an exact sequence $M \rightarrowtail F \twoheadrightarrow L$ of left $R$-modules with $F \in {}_R\mathcal{F}^\wedge_m$ and $L \in {}_R\mathcal{PGF}_{\mathcal{A}}$.
\end{enumerate}
Consequently, the equality ${}_R\mathcal{F}^\wedge_m = [{}_R\mathcal{GF}_{\A}]^\wedge_m \cap [{}_R\mathcal{PGF}_{\mathcal{A}}]^\perp$ holds for every $m \in \mathbb{Z}_{\geq 0}$. \\
\end{proposition}

We can obtain the following cotorsion pair and hereditary abelian model structure involving ${}_R\mathcal{PGF}_{\mathcal{A}}$ and ${}_R\mathcal{GF}_{\mathcal{A}}$ from the previous proposition. This generalizes \cite[Coroll. 2.20 \& Thm. 3.2 (1)]{EstradaMarco} and \cite[Coroll. 3.5 \& Thm. 3.6]{Rachid} simultaneously.

\begin{theorem} \label{GF_A-triple}
Let $\A \subseteq \mathsf{Mod}_R$ be a class containing the injective right $R$-modules, such that ${}_R\mathcal{PGF}_{\mathcal{A}}$ is the left half of a (hereditary) complete cotorsion pair in ${}_R\mathsf{Mod}$. Then, the following assertions hold true for every $m \in \mathbb{Z}_{\geq 0}$:
\begin{enumerate}[(1)]
\item $([{}_R\GF_{\A}]^\wedge_m, [{}_R\mathcal{PGF}_{\A}]^\perp \cap [{}_R\mathcal{F}^\wedge_m] ^\perp)$ is a hereditary complete cotorsion pair in ${}_R\mathsf{Mod}$. 

\item There exists a hereditary Hovey triple in ${}_R\mathsf{Mod}$ given by 
\[
([{}_R\GF_{\A}]^\wedge_m, [{}_R\mathcal{PGF}_{\A}]^\perp, [{}_R\mathcal{F}^\wedge_m]^\perp).
\]
\end{enumerate}
\end{theorem}

\begin{proof}
Part (1) can be shown following the arguments in \cite[Coroll. 3.5]{Rachid}, and using Example \ref{ex:FP-Gflat-relative}, Theorem \ref{Family} (1) and Proposition \ref{General-Gillespie}. Part (2), on the other hand, is immediate from (1) and Proposition \ref{General-Gillespie} as well, along with the fact that $[{}_R\mathcal{PGF}_{\A}]^\perp$ is thick.
\end{proof}

%%%%%%%%%%%%%%%%%%%%%%%%%%%%%%%%%%%%%
%%%%%%%%%%%%%%%%%%%%%%%%%%%%%%%%%%%%%
%%%%%%%%%%%%%%%%%%%%%%%%%%%%%%%%%%%%%
%%%%%%%%%%%%%%%%%%%%%%%%%%%%%%%%%%%%%

\section{Constructing new complete Frobenius pairs from old ones} \label{sec:applications}

In this last section we show how to induce complete Frobenius pairs from a given complete Frobenius pair. For example, we obtain some complete Frobenius pairs in the category of chain complexes of modules or representations from a complete Frobenius pair of modules. We also look for sufficient conditions that allow us to transfer Frobenius pair and their completeness via functors that form an adjoint pair.

%%%%%%%%%%%%%%%%%%%%%%%%%%%%%%%%%%%%%
%%%%%%%%%%%%%%%%%%%%%%%%%%%%%%%%%%%%%

\subsection*{Completeness in the category of complexes}

Let us commence with induced complete Frobenius pairs in the category ${}_R\mathsf{Ch}$ of chain complexes of left $R$-modules. We shall need the following lemma, whose proof can be obtained by mimicking the arguments in \cite[Thm. 3.1]{AkinciAlizade}.

\begin{lemma}\label{Yang-Liu}
Let $(\X, \omega)$ be a Frobenius pair which is totally left $\omega ^{\wedge} _m$-complete in $\X ^{\perp _1}$, and $M_1 \rightarrowtail M_2 \twoheadrightarrow M_3$ be a short exact sequence with $M_1, M_3 \in \X^{\perp _1}$. If we are given short exact sequences $K_1 \rightarrowtail L_1 \twoheadrightarrow M_1$ and $ K_3 \rightarrowtail L_3 \twoheadrightarrow M_3$ with $L_1, L_3 \in \omega^\wedge_m$ and $K_1, K_3 \in [\omega^\wedge_m]^\perp$, then there exists a short exact sequence $K_2 \rightarrowtail L_2 \twoheadrightarrow M_2$ with $L_2 \in \omega^\wedge_m$ and $K_2 \in [\omega^\wedge_m]^\perp$, such that the following diagram commutes
\[
\begin{tikzpicture}[description/.style={fill=white,inner sep=2pt}] 
\matrix (m) [matrix of math nodes, row sep=2.3em, column sep=2.3em, text height=1.25ex, text depth=0.25ex] 
{ 
K_1 & K_2 & K_3 \\
L_1 & L_2 & L_3 \\
M_1 & M_2 & M_3 \\
}; 
\path[>->]
(m-1-1) edge node[above] {\footnotesize$\alpha_2$} (m-1-2) edge node[left] {\footnotesize$g_1$} (m-2-1)
(m-1-2) edge node[left] {\footnotesize$g_2$} (m-2-2)
(m-1-3) edge node[left] {\footnotesize$g_3$} (m-2-3)
(m-2-1) edge node[below] {\footnotesize$\alpha_1$} (m-2-2)
(m-3-1) edge node[below] {\footnotesize$\alpha$} (m-3-2)
;
\path[->>]
(m-1-2) edge node[above] {\footnotesize$\beta_2$} (m-1-3)
(m-2-2) edge node[below] {\footnotesize$\beta_1$} (m-2-3) edge node[right] {\footnotesize$f_2$} (m-3-2)
(m-2-3) edge node[right] {\footnotesize$f_3$} (m-3-3)
(m-3-2) edge node[below] {\footnotesize$\beta$} (m-3-3)
(m-2-1) edge node[right] {\footnotesize$f_1$} (m-3-1)
;
\end{tikzpicture}
\]
\end{lemma}

Regarding Frobenius pairs in the category ${}_R\mathsf{Ch}$ of chain complexes of left $R$-modules, the authors in \cite[Thm. 3.3]{LiangYangConstructions} show that the following assertions are equivalent for any two classes of left $R$-modules $\X, \omega \subseteq {}_R\mathsf{Mod}$:
\begin{enumerate}[(a)]
\item $(\X,\omega)$ is a Frobenius pair in ${}_R\mathsf{Mod}$. 

\item $({\rm dw}\,\widetilde{\X},\widetilde{\omega})$ is a Frobenius pair in ${}_R\mathsf{Ch}$. 

\item $({\rm ex}\,\widetilde{\X},\widetilde{\omega})$ is a Frobenius pair in ${}_R\mathsf{Ch}$.

\item $(\widetilde{\X},\widetilde{\omega})$ is a Frobenius pair in ${}_R\mathsf{Ch}$.
\end{enumerate}

Before checking whether completeness can be transferred from $(\X,\omega)$ to any of the pairs $({\rm dw}\,\widetilde{\X},\widetilde{\omega})$, $({\rm ex}\,\widetilde{\X},\widetilde{\omega})$ or $(\widetilde{\X},\widetilde{\omega})$, let us add another induced Frobenius pair of complexes to the previous list involving the class ${\rm dg}\,\widetilde{\mathcal{X}}$.

\begin{remark}
In this section, by an abuse of notation, ${\rm dg}\,\widetilde{\mathcal{X}}$ will denote the class of complexes $X_\bullet \in {\rm dw}\,\widetilde{\mathcal{X}}$ such that $\mathcal{H}om(X_\bullet, W_\bullet)$ is exact for every $W_\bullet \in \widetilde{\omega}$. Thus, this is not exactly the acyclicity condition defining ${\rm dg}\,\widetilde{\mathcal{X}}$ in the context of Example \ref{ex:pairs_in_Ch}.
\end{remark}

The following is a slight extension of \cite[Lem. 3.1]{GillespieFlat}. In what follows, given $k \in \mathbb{Z}$ and $C$ an object in an abelian category $\mathfrak{C}$, $S^k(C)$ denotes the complex given by $C$ at degree $k$, and zero elsewhere. Similarly, $D^k(C)$ denotes the (exact) complex given by $C$ at degrees $k$ and $k-1$, zero elsewhere, and whose only nonzero differential is the identity on $C$.

\begin{lemma}\label{lem:GilExt}
Let $\mathfrak{C}$ be an abelian category with either enough injective or projective objects. Then, for every object $C$ in $\mathfrak{C}$, and any complex $X_\bullet$ of objects in $\mathfrak{C}$, there are natural isomorphisms for every $i > 0$ and $k \in \mathbb{Z}$:
\begin{enumerate}[(1)]
\item $\Ext^i_{\C}(C,X_k) \cong \Ext^i(D^k(C),X_\bullet)$; 

\item $\Ext^i_{\C}(X_k,C) \cong \Ext^i(X_\bullet, D^{k+1}(C))$.
\end{enumerate}
In the case where $\mathfrak{C}$ has enough injective objects and $Y_\bullet$ is an exact complex of objects in $\mathfrak{C}$, then for every $i > 0$ and $k \in \mathbb{Z}$ there is a natural isomorphism:
\begin{enumerate}[(1)]
\setcounter{enumi}{2}
\item $\Ext^i_{\C}(C,Z_k(Y_\bullet)) \cong \Ext^i(S^k(C),Y_\bullet)$.
\end{enumerate} 
Dually, if $\mathfrak{C}$ has enough projective objects, with $C$ and $Y_\bullet$ as above, then for every $i > 0$ and $k \in \mathbb{Z}$ there is a natural isomorphism:
\begin{enumerate}[(1)]
\setcounter{enumi}{3}
\item $\Ext^i_{\C}(Z_k(Y_\bullet),C) \cong \Ext^i(Y_\bullet,S^k(C))$. \\
\end{enumerate} 
\end{lemma}

\begin{proof}
For (1) and (2), let us write the case where $\mathfrak{C}$ has enough injective objects (the other case is dual). 
\begin{enumerate}[(1)]
\item Any injective coresolution $X_\bullet \rightarrowtail E^0_\bullet \to E^1_\bullet \to \cdots$ of $X_\bullet$ gives rise to the injective coresolution $X_k \rightarrowtail E^0_k \to E^1_k \to \cdots$ of $X_k$, since a complex is injective if, and only if, it is exact with injective cycles. By \cite[Lem. 3.1 (1)]{GillespieFlat}, the following complexes of abelian groups are naturally isomorphic:
\[
\Hom(D^k(C),E^0_\bullet) \to \Hom(D^k(C),E^1_\bullet) \to \cdots
\]
and
\[
\Hom_{\C}(C,E^0_k) \to \Hom_{\C}(C,E^1_k) \to \cdots.
\]
It then follows that they have isomorphic cohomology groups, and hence (1) follows.

\item It suffices to note that any injective coresolution $C \rightarrowtail E^0 \to E^1 \to \cdots$ of $C$ gives rise to an injective coresolution $D^{k+1}(C) \rightarrowtail D^{k+1}(E^0) \to \cdots$ of $D^{k+1}(C)$. The rest is similar to (1) using \cite[Lem. 3.1 (2)]{GillespieFlat}.
\end{enumerate}
We now prove (3) (note that (4) is dual). 
\begin{enumerate}[(1)]
\setcounter{enumi}{2}
\item Let $Y_\bullet \rightarrowtail E^0_\bullet \to E^1_\bullet \to \cdots$ be an injective coresolution of $Y_\bullet$. Since $Y_\bullet$ and each $E^i_\bullet$ are exact, the induced sequence $Z_k(Y_\bullet) \rightarrowtail Z_k(E^0_\bullet) \to Z_k(E^1_\bullet) \to \cdots$ is exact, and so an injective coresolution of $Z_k(Y_\bullet)$. The rest is similar to (1) using \cite[Lem. 3.1 (3)]{GillespieFlat}.
\end{enumerate}
\end{proof}

\begin{proposition}
Let $\X, \omega \subseteq {}_R\mathsf{Mod}$ be classes of left $R$-modules such that $0 \in X \cap \omega$, $\omega$ is closed under extensions and $\id_{\X}(\omega) = 0$. Then, $(\X, \omega)$ is a Frobenius pair in ${}_R\mathsf{Mod}$ if, and only if, $({\rm dg}\,\widetilde{\X}, \widetilde{\omega})$ is a Frobenius pair in ${}_R\mathsf{Ch}$. \\
\end{proposition}

\begin{proof} Let us suppose first that $(\X, \omega)$ is a Frobenius pair in ${}_R\mathsf{Mod}$.
\begin{itemize}
\item {\sf (FP1)}: Let $X_{\bullet}' \rightarrowtail X_{\bullet} \twoheadrightarrow X_{\bullet}''$ be an exact sequence in ${}_R\mathsf{Ch}$ with $X_{\bullet}'' \in {\rm dg}\,\widetilde{\X}$. We show that $X_{\bullet} \in {\rm dg}\,\widetilde{\X}$ if, and only if, $X_{\bullet}' \in {\rm dg}\,\widetilde{\X}$. Let $W_\bullet \in \widetilde{\omega}$. Then, for each $k \in \mathbb{Z}$ there is a short exact sequence $Z_k(W_\bullet) \rightarrowtail W_k \twoheadrightarrow Z_{k-1}(W_\bullet)$ with $Z_k(W_\bullet), Z_{k-1}(W_\bullet) \in \omega$. Since $\omega$ is closed under extensions by \cite[Prop. 2.7]{BMSP}, we have that $W_k \in \omega$ for every $k \in \mathbb{Z}$. On the other hand, $X''_k \in \X$ for every $k \in \mathbb{Z}$. Then, we obtain the following induced short exact sequence of complexes of abelian groups 
\[
\mathcal{H}om(X''_{\bullet}, W_\bullet) \rightarrowtail \mathcal{H}om(X_{\bullet}, W_\bullet) \twoheadrightarrow \mathcal{H}om(X'_{\bullet}, W_\bullet).
\] 
Finally, since exact complexes form a thick class, and ${\rm dw}\,\widetilde{\X}$ is left thick by \cite[Thm. 3.3]{LiangYangConstructions}, we have that $X_{\bullet} \in {\rm dg}\,\widetilde{\X}$ if, and only if, $X_{\bullet}' \in {\rm dg}\,\widetilde{\X}$. It is also clear from the previous that ${\rm dg}\,\widetilde{\X}$ is closed under direct summands. Therefore, ${\rm dg}\,\widetilde{\X}$ is left thick. 

\item {\sf (FP2)}: $\widetilde{\omega}$ is closed under direct summands by \cite[Thm. 3.3]{LiangYangConstructions}. 

\item {\sf (FP3)}: $\Ext^i(X_\bullet,W_\bullet) = 0$ for every $X_\bullet \in {\rm dg}\,\widetilde{\X}$, $W_\bullet \in \widetilde{\omega}$ and $i > 0$ by \cite[Thm. 3.3]{LiangYangConstructions}, since ${\rm dg}\,\widetilde{\X} \subseteq {\rm dw}\,\widetilde{\X}$.

\item {\sf (FP4)}: Let $X_{\bullet} \in {\rm dg}\,\widetilde{\X} \subseteq {\rm dw}\,\widetilde{\X}$. Since $\widetilde{\omega}$ is a relative cogenerator in ${\rm dw}\,\widetilde{\X}$ by \cite[Thm. 3.3]{LiangYangConstructions},  there is an exact sequence $X_{\bullet} \rightarrowtail W_{\bullet} \twoheadrightarrow X_{\bullet}'$ with $W_{\bullet} \in \widetilde{\omega}$ and $X_{\bullet}' \in {\rm dw}\,\widetilde{\X}$. It remains to show that $\mathcal{H}om(X_{\bullet}', K_{\bullet})$ is exact for every $K_{\bullet} \in \widetilde{\omega}$, in order to conclude that $\widetilde{\omega}$ is a relative cogenerator in ${\rm dg}\,\widetilde{\X}$. This follows from the induced exact sequence 
\[
\mathcal{H}om (X_{\bullet}', K_{\bullet}) \rightarrowtail \mathcal{H}om(W_{\bullet}, K_{\bullet}) \twoheadrightarrow \mathcal{H}om (X_{\bullet}, K_{\bullet}),
\] 
since $\mathcal{H}om (X_{\bullet}, K_{\bullet})$ is exact by definition, and one can note that so is the complex $\mathcal{H}om(W_{\bullet}, K_{\bullet})$ by following the arguments in \cite[proof of Lem. 3.9]{GillespieFlat} (notice that $\id_\omega(\omega) = 0$).
\end{itemize} 

Now suppose that $({\rm dg}\,\widetilde{\X},\widetilde{\omega})$ is a Frobenius pair in ${}_R\mathsf{Ch}$. We verify the definition of Frobenius pair for $(\X,\omega)$:
\begin{itemize}
\item {\sf (FP1)}: Let $X' \rightarrowtail X \twoheadrightarrow X''$ be a short exact sequence in ${}_R\mathsf{Mod}$ with $X'' \in \X$. This induces a short exact sequence $S^0(X') \rightarrowtail S^0(X) \twoheadrightarrow S^0(X'')$ where $S^0(X'') \in {\rm dg}\,\widetilde{\X}$ by following the arguments in \cite[Lem. 3.4]{GillespieFlat} (notice that $\id_{\X}(\omega) = 0$). Note that $S^0(X)$ (resp., $S^0(X')$) belongs to ${\rm dg}\,\widetilde{\X}$ if, and only if, $X$ (resp., $X'$) belongs to $\X$. Then, the claimed equivalence follows. Similarly, one can note that $\X$ is closed under direct summands. 

\item {\sf (FP2)}: Note also that $\omega$ is closed under direct summands, from the same closure property assumed for $\widetilde{\omega}$. Indeed, note that $D^1(W) \in \widetilde{\omega}$ if, and only if, $W \in \omega$.

\item {\sf (FP3)}: Let $X \in \X$ and $W \in \omega$. Then, $S^0(X) \in {\rm dg}\,\widetilde{\X}$ and $D^1(W) \in \widetilde{\omega}$, and so by Lemma \ref{lem:GilExt} we have that $\Ext^i_R(X,W) \cong \Ext^i(S^0(X),D^1(W)) = 0$, since $({\rm dg}\,\widetilde{\X},\widetilde{\omega})$ is a Frobenius pair. Hence, $\id_{\X}(\omega) = 0$. 

\item {\sf (FP4)}: Let $X \in \X$ and consider $S^0(X) \in {\rm dg}\,(\X)$. Since $\widetilde{\omega}$ is an relative cogenerator in ${\rm dg}\,\widetilde{\X}$, there is an exact sequence $S^0(X) \rightarrowtail W_\bullet \twoheadrightarrow X'_\bullet$ where $W_\bullet \in \widetilde{\omega}$ and $X'_\bullet \in {\rm dg}\,\widetilde{\X}$. Then, there is an exact sequence (at degree 0) $X \rightarrowtail W_0 \twoheadrightarrow X'_0$ where $W_0 \in \omega$ and $X'_0 \in \X$. It then follows that $\omega$ is a relative cogenerator in $\X$. 
\end{itemize}
\end{proof}

We are now ready to study how completeness transfers from the Frobenius pair $(\X,\omega)$ of modules to the Frobenius pairs $({\rm dw}\,\widetilde{\X},\widetilde{\omega})$, $({\rm ex}\,\widetilde{\X},\widetilde{\omega})$, $({\rm dg}\,\widetilde{\X},\widetilde{\omega})$ and $(\widetilde{\X},\widetilde{\omega})$ of complexes. Recall that ${}_R\mathcal{I}$ denotes the class of injective left $R$-modules. In what follows, $\mathcal{E}$ denotes the class of exact complexes of left $R$-modules and ${\rm dg}\,\widetilde{{}_R\mathcal{I}}$ the class of dg-injective complexes of left $R$-modules.

\begin{proposition}
Let $(\X,\omega)$ be a Frobenius pair in ${}_R\mathsf{Mod}$, which is totally left $\omega^\wedge_m$-complete in $\X ^{\perp _1}$ and such that $R \in \X$. The following assertions hold:
\begin{enumerate}[(1)]
\item $({\rm dw}\,\widetilde{\X}, \widetilde{\omega})$ is left $\widetilde{\omega}^\wedge_m$-complete in $[{\rm dw}\,\widetilde{\X}]^{\perp_1}$.

\item $({\rm dg}\,\widetilde{\X}, \widetilde{\omega})$ is left $\widetilde{\omega}^\wedge_m$-complete in $[{\rm dg}\,\widetilde{\X}]^{\perp_1}$.
\end{enumerate}
In in addition $\X$ is closed under coproducts, then:
\begin{enumerate}[(1)]
\setcounter{enumi}{2}
\item $(\widetilde{\X}, \widetilde{\omega})$ is left $\widetilde{\omega}^\wedge_m$-complete in $[\widetilde{\X}]^{\perp_1}$. 

\item $({\rm ex}\,\widetilde{\X}, \widetilde{\omega})$ is left $\widetilde{\omega}^\wedge_m$-complete in $[{\rm ex}\,\widetilde{\X}]^{\perp_1}$. \\
\end{enumerate}
\end{proposition}

\begin{proof} ~\
\begin{enumerate}[(1)]
\item Let $Y_{\bullet} \in [{\rm dw}\,\widetilde{\X}]^{\perp _1}$. Then, $\Ext^1(S^k(R),Y_\bullet) = 0$, and thus $Y_\bullet$ is exact by \cite[Coroll. 2.1.7]{EnJen00-2}. This, along with Lemma \ref{lem:GilExt}, implies that $\Ext^1_R(X,Z_k(Y_\bullet)) \cong \Ext^1(S^k(X),Y_\bullet) = 0$ for every $X \in \X$, and so $Y_\bullet$ has cycles in $\X^{\perp_1}$. Hence, $[{\rm dw}\,\widetilde{\X}]^{\perp _1} \subseteq \widetilde{\X^{\perp_1}}$. 

Now for each $k \in \mathbb{Z}$ there is an exact sequence $K_k \rightarrowtail L_k \twoheadrightarrow  Z_k(Y_\bullet)$ with $K_k \in [\omega^\wedge_m]^\perp$ and $L_k \in \omega^\wedge_m$. Then by Lemma \ref{Yang-Liu} we can find an exact sequence $K'_k \rightarrowtail L'_k \twoheadrightarrow Y_k$ with $K'_k \in [\omega^\wedge_m]^\perp$ and $L'_k \in \omega^\wedge_m$ such that the following diagram commutes:
\[
\begin{tikzpicture}[description/.style={fill=white,inner sep=2pt}] 
\matrix (m) [matrix of math nodes, row sep=2.3em, column sep=2.3em, text height=1.25ex, text depth=0.25ex] 
{ 
K_k & K'_k & K_{k-1} \\
L_k & L'_k & L_{k-1} \\
Z_k(Y_\bullet) & Y_k & Z_{k-1}(Y_\bullet) \\
}; 
\path[>->]
(m-1-1) edge (m-1-2) edge (m-2-1)
(m-1-2) edge (m-2-2)
(m-1-3) edge (m-2-3)
(m-2-1) edge (m-2-2)
(m-3-1) edge (m-3-2)
;
\path[->>]
(m-1-2) edge (m-1-3)
(m-2-2) edge (m-2-3) edge (m-3-2)
(m-2-3) edge (m-3-3)
(m-3-2) edge (m-3-3)
(m-2-1) edge (m-3-1)
;
\end{tikzpicture}
\]
The upper and middle rows give rise to exact complexes $K'_\bullet$ and $L'_\bullet$, and an exact sequence $K'_\bullet \rightarrowtail L'_\bullet \twoheadrightarrow Y_\bullet$. At this point, the idea is to show that $L'_\bullet \in \widetilde{\omega}^\wedge_m$ and $K'_\bullet \in [\widetilde{\omega}^\wedge_m]^{\perp_1}$. To that end, we need to prove some previous assertions:
\begin{enumerate}[(i)]
\item $\widetilde\omega^\wedge_m = \widetilde{\omega^\wedge_m}$: First, let $L'_\bullet \in \widetilde{\omega}^\wedge_m$. Then, we have an exact sequence
\[
W^m_\bullet \rightarrowtail W^{m-1}_\bullet \to \cdots \to W_\bullet^1 \to W_\bullet^0  \twoheadrightarrow L'_\bullet,
\] 
where $W^j_\bullet \in \widetilde{\omega}$ for every $0 \leq j \leq m$. Since exact complexes form a thick class, we have that $L'_\bullet$ is exact. Then, the previous sequence gives rise to the following exact sequences on cycles
\[
Z_k(W^m_\bullet) \rightarrowtail Z_k(W^{m-1}_\bullet) \to \cdots \to Z_k(W_\bullet^1) \to Z_k(W_\bullet^0)  \twoheadrightarrow Z_k(L'_\bullet),
\]
where $Z_k(W^j_\bullet) \in \omega$ for every $0 \leq j \leq m$, and so $Z_k(L'_\bullet) \in \omega^\wedge_m$. Hence, $L'_\bullet \in \widetilde{\omega^\wedge_m}$. 
 
Now suppose that $L'_\bullet \in \widetilde{\omega^\wedge_m}$. Then $L'_\bullet$ is exact and $Z_k(L'_\bullet) \in \omega^\wedge_m$ for every $k \in \mathbb{Z}$. In particular, for each $k \in \mathbb{Z}$ we have an exact sequence
\[
W^m_k \rightarrowtail W^{m-1}_k \to \cdots \to W^1_k \to W^0_k \twoheadrightarrow Z_k(L'_\bullet),
\]
with $W^j_k \in \omega$ for every $0 \leq j \leq m$. On the other hand, for each exact sequence 
\[
\varepsilon_k \colon Z_k(L'_\bullet) \rightarrowtail L'_k \twoheadrightarrow Z_{k-1}(L'_\bullet)
\] 
note that $\Hom_R(W,\varepsilon_k)$ is exact for every $W \in \omega$, since $\id_{\X}(\omega^\wedge) = 0$. Recall also that $\omega$ is closed under extensions and direct summands, and so it is closed under finite direct sums. Then, by \cite[dual of Lem. 1.7]{Holm04} we can construct the following commutative diagram with exact rows and columns:
\[
\begin{tikzpicture}[description/.style={fill=white,inner sep=2pt}] 
\matrix (m) [matrix of math nodes, row sep=2.3em, column sep=2.3em, text height=1.25ex, text depth=0.25ex] 
{ 
W^m_k & Q^m_k & W^m_{k-1} \\
W^{m-1}_k & W^{m-1}_k \oplus W^{m-1}_{k-1} & W^{m-1}_{k-1} \\
\vdots & \vdots & \vdots \\
W^1_k & W^1_k \oplus W^1_{k-1} & W^1_{k-1} \\
W^0_k & W^0_k \oplus W^0_{k-1} & W^0_{k-1} \\
Z_k(L'_\bullet) & L'_k & Z_{k-1}(L'_\bullet) \\
}; 
\path[->]
(m-2-1) edge (m-3-1) (m-2-2) edge (m-3-2) (m-2-3) edge (m-3-3)
(m-3-1) edge (m-4-1) (m-3-2) edge (m-4-2) (m-3-3) edge (m-4-3)
(m-4-1) edge (m-5-1) (m-4-2) edge (m-5-2) (m-4-3) edge (m-5-3)
;
\path[>->]
(m-1-1) edge (m-1-2) edge (m-2-1)
(m-1-2) edge (m-2-2)
(m-1-3) edge (m-2-3)
(m-2-1) edge (m-2-2)
(m-4-1) edge (m-4-2)
(m-6-1) edge (m-6-2)
;
\path[->>]
(m-1-2) edge (m-1-3)
(m-2-2) edge (m-2-3)
(m-5-2) edge (m-5-3) edge (m-6-2)
(m-5-3) edge (m-6-3)
(m-4-2) edge (m-4-3)
(m-5-1) edge (m-6-1) edge (m-5-2)
(m-6-2) edge (m-6-3)
;
\end{tikzpicture}
\]
This implies that there is an exact sequence
\[
W^m_\bullet \rightarrowtail W^{m-1}_\bullet \to \cdots \to W^1_\bullet \to W^0_\bullet \twoheadrightarrow L'_\bullet,
\]
with $W^j_\bullet \in \widetilde{\omega}$ for every $0 \leq j \leq m$. Therefore, $L'_\bullet \in \widetilde{\omega}^\wedge_m$.

\item $\widetilde{[\omega^\wedge_m]^\perp} \subseteq [\widetilde{\omega^\wedge_m}]^{\perp_1}$: Let $K'_\bullet \in \widetilde{[\omega^\wedge_m]^\perp}$ and $L_\bullet \in \widetilde{\omega}^\wedge_m = \widetilde{\omega^\wedge_m}$. Recall that $\omega^\wedge_m$ is closed under extensions by Lemma \ref{LemaGorro} (2), and so $L_\bullet \in {\rm dw}\,\widetilde{\omega^\wedge_m}$. We show that $\mathcal{H}{\rm om}(L_\bullet,K'_\bullet)$ is exact. Let $Z(K'_\bullet)$ denote the complex formed by the cycles of $K'_\bullet$ with differentials given by the zero morphisms at each degree. Note that $Z(K'_\bullet)$ is a subcomplex of $K'_\bullet$ and that $Z(K'_\bullet) = \oplus_{k \in \mathbb{Z}} S^k(Z_k(K'_\bullet))$. Then, we have a short exact sequence 
\[
Z(K'_\bullet) \rightarrowtail K'_\bullet \twoheadrightarrow K'_\bullet / Z(K'_\bullet),
\] 
where $K'_\bullet / Z(K'_\bullet) = \oplus_{k \in \mathbb{Z}} S^k(Z_{k-1}(K'_\bullet))$. Since $Z_k(K'_\bullet) \in [\omega^\wedge_m]^\perp$ for every $k \in \mathbb{Z}$, we can note that $\mathcal{H}{\rm om}(L_\bullet,Z(K'_\bullet))$ and $\mathcal{H}{\rm om}(L_\bullet, K'_\bullet / Z(K'_\bullet))$ are exact. Moreover, since $\Ext^1_R(L_k,Z_{n+k}(K'_\bullet)) = 0$ for every $k, n \in \mathbb{Z}$, we have the induced exact sequence
\[
\mathcal{H}{\rm om}(L_\bullet,Z(K'_\bullet)) \rightarrowtail \mathcal{H}{\rm om}(L_\bullet,K'_\bullet) \twoheadrightarrow \mathcal{H}{\rm om}(L_\bullet, K'_\bullet / Z(K'_\bullet))
\]
where the left and right terms are exact complexes. Thus, $\mathcal{H}{\rm om}(L_\bullet,K'_\bullet)$ is exact. Then, by \cite[Lem. 5.38 (1)]{HMP-n-cot}, we have that $\Ext^1(L_\bullet,K'_\bullet) = 0$.
\end{enumerate}

Now let $K'_\bullet \rightarrowtail L'_\bullet \twoheadrightarrow Y_\bullet$ be the short exact sequence obtained right before Claim (i). Here, $L'_\bullet$ is an exact complex with cycles in $\omega^\wedge_m$, and so $L'_\bullet \in \widetilde{\omega}^\wedge_m$. Also, $K'_\bullet$ is an exact complex with cycles in $[\omega^\wedge_m]^\perp$, and so $K' \in [\widetilde{\omega^\wedge_m}]^{\perp_1}$. Therefore, the Frobenius pair $({\rm dw}\,\widetilde{\X}, \widetilde{\omega})$ is $\widetilde{\omega}^\wedge_m$-complete in $[{\rm dw}\,\widetilde{\X}]^{\perp_1}$. 

\item Follows from part (1) and the containment $({\rm dg}\,\widetilde{\X})^{\perp_1} \subseteq \widetilde{\X^{\perp_1}}$. The proof of the latter containment is similar to the proof of $[{\rm dw}\,\widetilde{\X}]^{\perp _1} \subseteq \widetilde{\X^{\perp_1}}$.

\item Let $Y_{\bullet} \in [\widetilde{\X}]^{\perp _1}$. We know from \cite[Prop. 2.3.25]{GRcomplexes} that $(\mathcal{E}, {\rm dg}\,\widetilde{{}_R\mathcal{I}})$ is a hereditary and perfect cotorsion pair in ${}_R\mathsf{Ch}$. In particular, there is a short exact sequence $\theta \colon E_\bullet \rightarrowtail W_\bullet \twoheadrightarrow Y_\bullet$ where $E_\bullet$ is dg-injective and $W_\bullet$ is exact. It follows from the definition of dg-injective complexes and from \cite[dual of Lem. 5.38 (3)]{HMP-n-cot} that $E_\bullet, Y_\bullet \in {\rm dw}\widetilde{\X^{\perp_1}}$ and that $\mathcal{H}{\rm om}(X_\bullet,E_\bullet)$ and $\mathcal{H}{\rm om}(X_\bullet,Y_\bullet)$ are exact for every $X_\bullet \in \widetilde{\X}$. Since $\Ext^1_R(X_k,E_{n+k}) = 0$ for every $k, n \in \mathbb{Z}$, the exact sequence $\theta$ induces the exact sequence of complexes
\[
\mathcal{H}{\rm om}(X_\bullet,E_\bullet) \rightarrowtail \mathcal{H}{\rm om}(X_\bullet, W_\bullet) \twoheadrightarrow \mathcal{H}{\rm om}(X_\bullet,Y_\bullet)
\]
where the exactness of the left and right terms implies that $\mathcal{H}{\rm om}(X_\bullet, W_\bullet)$ is exact. Note that $\X$ contains the projective left $R$-modules, since it is closed under direct summands and the free left $R$-modules are in $\X$ since we are assuming that $R \in \X$ and that $\X$ is closed under coproducts. Then, for every $X \in \X$, any projective resolution of $P_\bullet \twoheadrightarrow X$ is a complex in $\widetilde{\X}$, since $\X$ is left thick and projectives are in $\X$, and so $\mathcal{H}{\rm om}(P_\bullet \twoheadrightarrow X,W_\bullet)$ is exact. At this point, we can mimic \cite[proof of Thm. 3.12]{GillespieFlat} in order to show that $Z_k(W_\bullet) \in \X^{\perp_1}$ for every $k \in \mathbb{Z}$. Hence, $W_\bullet \in \widetilde{\X^{\perp_1}}$. 

Now from the proof of part (1), we know that we can obtain an exact sequence $K_\bullet \rightarrowtail L_\bullet \twoheadrightarrow W_\bullet$ where $L_\bullet \in \widetilde{\omega}^\wedge_m$ and $K_\bullet \in [\widetilde{\omega}^\wedge_m]^{\perp_1}$. Now form the following pullback diagram:
\[
\begin{tikzpicture}[description/.style={fill=white,inner sep=2pt}] 
\matrix (m) [matrix of math nodes, row sep=2.3em, column sep=2.3em, text height=1.25ex, text depth=0.25ex] 
{ 
K_\bullet & K_\bullet & {} \\
Q_\bullet & L_\bullet & Y_\bullet \\
E_\bullet & W_\bullet & Y_\bullet \\
}; 
\path[->] 
(m-2-1)-- node[pos=0.5] {\footnotesize$\mbox{\bf pb}$} (m-3-2) 
;
\path[>->]
(m-1-1) edge (m-2-1) (m-1-2) edge (m-2-2)
(m-2-1) edge (m-2-2)
(m-3-1) edge (m-3-2)
;
\path[->>]
(m-2-2) edge (m-2-3) edge (m-3-2)
(m-2-1) edge (m-3-1)
(m-3-2) edge (m-3-3)
;
\path[-,font=\scriptsize]
(m-1-1) edge [double, thick, double distance=2pt] (m-1-2)
(m-2-3) edge [double, thick, double distance=2pt] (m-3-3)
;
\end{tikzpicture}
\]
The proof will conclude after showing that $Q_\bullet \in [\widetilde{\omega}^\wedge_m]^{\perp_1}$, for which it suffices to note that $E_\bullet \in [\widetilde{\omega}^\wedge_m]^{\perp_1}$. For, let $L'_\bullet \in \widetilde{\omega}^\wedge_m = \widetilde{\omega^\wedge_m}$. In particular, $L'_\bullet$ is exact. On the other hand, since $E_\bullet$ has injective components, every short exact sequence in $\Ext^1(L'_\bullet, E_\bullet)$ splits at each degree. Then by \cite[Lem. 2.1]{GillespieFlat} we have that $\Ext^1(L'_\bullet, E_\bullet) \cong H_{-1}(\mathcal{H}{\rm om}(L'_\bullet,E_\bullet))$. Note also that $\mathcal{H}{\rm om}(L'_\bullet,E_\bullet)$ is exact since $L'_\bullet$ is exact and $E_\bullet$ is dg-injective. Hence, $\Ext^1(L'_\bullet, E_\bullet) = 0$. 
 
\item Follows from part (3) and the containment $[{\rm ex}\widetilde{\X}]^{\perp_1} \subseteq [\widetilde{\X}]^{\perp_1}$. 
\end{enumerate}
\end{proof}

%%%%%%%%%%%%%%%%%%%%%%%%%%%%%%%%%%%%%
%%%%%%%%%%%%%%%%%%%%%%%%%%%%%%%%%%%%%

\subsection*{Completeness of Frobenius pairs in categories of quiver representations}

Let $Q = (Q_0,Q_1)$ be a quiver, where $Q_0$ denotes its set of vertices, and $Q_1$ its set of arrows. Given an arrow $a \in Q_1$, we denote by $s(a)$ and $t(a)$ its source and target. Note that a quiver $Q$ can be regarded as a category in which the objects are given by $Q_0$ and the morphisms are the paths of $Q$. The identities are given by the paths of length zero at every vertex. We denote by $\mathrm{Rep}(Q,{}_R\mathsf{Mod})$ the (abelian) category of all ${}_R\mathsf{Mod}$-valued representations of $Q$. Here, a \emph{representation} $X$ of $Q$ is a covariant functor $X \colon Q \to {}_R\mathsf{Mod}$, and a morphism of representations $X \to Y$ is a natural transformation. 

For the rest of this section, $Q$ will be a \emph{left rooted} quiver, that is, it contains no paths of the form $\cdots \to \bullet \to \bullet \to \bullet$. 

For a vertex $i \in Q_0$, we set $Q_1^{\ast \to i} := \{ a \in Q_1 \, \text{:} \, t(a) = i \}$. Using the universal property of coproducts, we have for every $i \in Q_0$ a unique morphism 
\[
\phi_i^X \colon \coprod_{a \in Q_1^{\ast \to i }} X(s(a)) \to X(i)
\]
such that $\phi_i^X \circ \mu_a = X(a)$, where $\mu_a \colon X(s(a)) \rightarrowtail \coprod_{a \in Q_1^{\ast \to i }} X(s(a))$ are the canonical coproduct morphisms. Let $C_i(X) := {\rm CoKer}(\phi_i^X)$ denote the cokernel of this morphism. For any $\X \subseteq {}_R\mathsf{Mod}$, we have the following full subcategories of $\mathsf{Rep}(Q,{}_R\mathsf{Mod})$:
\begin{align*}
\mathsf{Rep}(Q,\X) & := \{ X \in \mathsf{Rep}(Q,{}_R\mathsf{Mod}) \, \text{:} X(i) \in \X, \, \text{for every} \, i \in Q_0 \}, \\
\Phi(\X) & := \{ X \in \mathsf{Rep}(Q,{}_R\mathsf{Mod}) \, \text{:} \, \phi^X_i \, \text{is monic and} \, C_i(X) \in \X, \, \text{for every} \, i \in Q_0 \}.
\end{align*}

Induced Frobenius pairs in $\mathsf{Rep}(Q,{}_R\mathsf{Mod})$ involving the previous classes have been previously studied by Liang and Yang in \cite[\S 4]{LiangYangConstructions}. The following is an immediate consequence of \cite[Thm. 4.5]{LiangYangConstructions} and Corollary \ref{coro:inducedXmFP}. \\

\begin{proposition}\label{coro:Rep_omega_complete}
Let $(\X,\omega)$ be a Frobenius pair in ${}_R\mathsf{Mod}$ such that $(\omega^{\wedge}_m,[\omega^{\wedge}_m]^{\perp_1})$ is a complete cotorsion pair in ${}_R\mathsf{Mod}$ (in particular, the Frobenius pair is $\omega^{\wedge}_m$-complete in $\X^{\perp_1}$). Then $(\Phi(\X^\wedge_m), \Phi(\omega^\wedge_m) \cap \mathsf{Rep}(Q,[\omega^\wedge_m]^{\perp_1}))$ is a Frobenius pair in $\mathsf{Rep}(Q,{}_R\mathsf{Mod})$.  
\end{proposition}

\begin{example}
From the previous, we can note that $$(\Phi({}_R\mathcal{GP}^\wedge_m), \Phi({}_R\mathcal{P}^\wedge_m) \cap \mathsf{Rep}(Q,[{}_R\mathcal{P}^\wedge_m]^{\perp_1}))$$ is a Frobenius pair in $\mathsf{Rep}(Q,{}_R\mathsf{Mod})$. Analogous Frobenius pairs can be obtained from the classes ${}_R\mathcal{PGF}_{\A}$ (if it is the left half of a hereditary cotorsion pair cogenerated by a set) and ${}_R\mathcal{GP}_{(\mathcal{L,A})}$ (if $(\mathcal{L,A})$ is a complete duality pair).
\end{example}

%%%%%%%%%%%%%%%%%%%%%%%%%%%%%%%%%%%%%
%%%%%%%%%%%%%%%%%%%%%%%%%%%%%%%%%%%%%

\subsection*{Complete Frobenius pairs in adjoint pairs}

In this last example we fix two additive functors functors $F \colon \C \to \mathfrak{D}$ and $G \colon \mathfrak{D} \to \C$ between abelian categories that form an adjoint pair $(G, F)$. The idea is to look for conditions so that $F$ reflects (complete) Frobenius pairs, that is, that $(\X,\omega)$ is a (complete) Frobenius pair in $\C$ whenever $(F(\X),F(\omega))$ is a (complete) Frobenius pair in $\mathfrak{D}$. Here, by $F(\X)$ and $F(\omega)$ we denote the \emph{essential images} of $F$ restricted to $\X$ and $\omega$, that is, $F(\X)$ (resp., $F(\omega)$) is the full subcategory of $\mathfrak{D}$ whose objects are given by those $D \in \mathfrak{D}$ for which there exists $X \in \X$ (resp., $W \in \omega$) such that $D$ is isomorphic to $FX$ (resp., to $FW$).

\begin{proposition} \label{Funtor-Frob}
Suppose $F$ is right exact (and so, exact) full and faithful. If $(\X, \omega)$ is a pair of classes of objects in $\C$ such that $(F(\X), F (\omega))$ is a Frobenius pair in $\mathfrak{D}$, then $\X$ is left thick. Moreover, if $G$ is left exact (and so, exact) and either $\mathfrak{D}$ has enough projective objects or $\C$ has enough injective objects,  then $(\X, \omega)$ is a Frobenius pair in $\C$. \\
\end{proposition}

\begin{proof} ~\
\begin{itemize}
\item {\sf (FP1)}: Suppose we are given a short exact sequence $X' \rightarrowtail X \twoheadrightarrow X''$ with $X'' \in \X$. Let us show that $X' \in \X$ if, and only if, $X \in \X$. Since $F$ is an exact functor, we have the exact sequence $FX' \rightarrowtail FX \twoheadrightarrow FX''$ with $FX'' \in F(\X)$. Suppose $X' \in \X$. Then, $FX' \in F(\X)$, and so $FX \in F(\X)$ since $F(\X)$ is left thick. It follows that $FX \simeq FY$ with $Y \in \X$. Now since $F$ is full and faithful, we have by \cite[Prop. 3.4.1]{BorceuxI} that the counit $\varepsilon \colon GF \to {\rm id}_{\C}$ is a natural isomorphism, and so $X \simeq GFX \simeq GFY \simeq Y$. Since $\X$ is closed under isomorphisms, we obtain $X \in \X$. In a similar way, we can show that $X \in \X$ implies that $X' \in \X$. 

Now if $X'$ is a direct summand of $X \in \X$, then $FX'$ is a direct summand of $FX \in F(\X)$. Since $F(\X)$ is closed under direct summands, we have that $FX' \in F(\X)$. Repeating the argument above, we can conclude that $X' \in \X$, and hence $\X$ is closed under direct summands. 

\item {\sf (FP2)}: That $\omega$ is closed under direct summands follows as in part {\sf (FP1)}. 
\end{itemize}
Now suppose that $G$ is right exact and that either $\mathfrak{D}$ has enough projective objects or $\C$ has enough injective objects. 
\begin{itemize}
\item {\sf (FP3)}: Let $X \in \X$ and $W \in \omega$, and write $X \simeq GFX$. By \cite[Coroll. 1]{AdamsRieffel}, we have that $\Ext^i_{\C}(X,W) \cong \Ext^i_{\C}(GFX,W) \cong \Ext^i_{\mathfrak{D}}(FX,FW) = 0$, where the last equality holds since $\id_{F(\X)}(F(\omega)) = 0$. Hence, $\id_{\X}(\omega) = 0$.

\item {\sf (FP4)}: We show that $\omega \subseteq \X$ and that $\omega$ is a relative cogenerator in $\X$. The first part is straightforward. Now let $X \in \X$. Then, $FX \in F(\X)$, and since $F(\omega)$ is a relative cogenerator in $F(\X)$, we can assert that there exists $FW$ and $FX'$ with $W \in \omega$ and $X' \in \X$ in a short exact sequence $FX \rightarrowtail FW \twoheadrightarrow FX'$. On the other hand, since $F$ is full and faithful, there are unique arrows $f \colon X \to W$ and $g \colon W \to X'$ such that $FX \rightarrowtail FW = Ff$ and $FW \twoheadrightarrow FX' = Fg$. We show that $f$ is monic, $g$ is epic, and that the sequence $X \stackrel{f}\rightarrowtail W \stackrel{g}\twoheadrightarrow X'$ is exact. Indeed, from the assumptions we have the commutative diagram
\[
\begin{tikzpicture}[description/.style={fill=white,inner sep=2pt}] 
\matrix (m) [matrix of math nodes, row sep=2.3em, column sep=2.3em, text height=1.25ex, text depth=0.25ex] 
{ 
GFX & GFW & GFX' \\
X & W & X' \\
}; 
\path[->] 
(m-1-1) edge node[left] {\footnotesize$\varepsilon_X$} (m-2-1) (m-1-2) edge node[left] {\footnotesize$\varepsilon_W$} (m-2-2) (m-1-3) edge node[left] {\footnotesize$\varepsilon_{X'}$} (m-2-3) 
;
\path[>->]
(m-1-1) edge node[above] {\footnotesize$GFf$} (m-1-2)
(m-2-1) edge node[below] {\footnotesize$f$} (m-2-2)
;
\path[->>]
(m-1-2) edge node[above] {\footnotesize$GFg$} (m-1-3)
(m-2-2) edge node[below] {\footnotesize$g$} (m-2-3)
;
\end{tikzpicture}
\]
where the top row is exact since $G$ is exact, and the vertical arrows are isomorphisms. Then, it is clear that $X \stackrel{f}\rightarrowtail W \stackrel{g}\twoheadrightarrow X'$ is exact.
\end{itemize}
\end{proof}

Now we study how completeness is reflected by $F$. We shall need the following lemma.

\begin{lemma}\label{lem:adjunction}
If $F$ is full, faithful and right exact, then the following assertions hold true for every class $\omega$ of objects in $\C$:
\begin{enumerate}[(1)]
\item $F(\omega^{\wedge}_m) \subseteq [F(\omega)]^\wedge_m$.

\item $F(\omega^{\wedge}_m) \supseteq [F(\omega)]^\wedge_m$, provided that $G$ is also left exact.
\end{enumerate}
\end{lemma}

\begin{proof}
Part (1) is straightforward, so we focus on showing (2). Let $D \in \mathfrak{D}$ be an object in $[F(\omega)]^\wedge_m$. Since $F$ is full and faithful, we have an exact sequence
\[
FW_m \stackrel{Ff_m}\rightarrowtail FW_{m-1} \to \cdots \to FW_1 \xrightarrow{Ff_1} FW_0 \twoheadrightarrow D
\]
where $W_k \in \omega$ for every $0 \leq k \leq m$. After applying $G$ we obtain the following commutative diagram with exact rows and where the vertical arrows are isomorphisms:
\[
\begin{tikzpicture}[description/.style={fill=white,inner sep=2pt}] 
\matrix (m) [matrix of math nodes, row sep=2.3em, column sep=3em, text height=1.25ex, text depth=0.25ex] 
{ 
GFW_m & GFW_{m-1} & \cdots & GFW_1 & GFW_0 & GD \\
W_m & W_{m-1} & \cdots & W_1 & W_0 & GD \\
}; 
\path[->] 
(m-1-1) edge node[left] {\footnotesize$\varepsilon_{W_m}$} (m-2-1) (m-1-2) edge node[left] {\footnotesize$\varepsilon_{W_{m-1}}$} (m-2-2) (m-1-4) edge node[left] {\footnotesize$\varepsilon_{W_1}$} (m-2-4) (m-1-5) edge node[left] {\footnotesize$\varepsilon_{W_0}$} (m-2-5) 
(m-1-2) edge (m-1-3) (m-1-3) edge (m-1-4) (m-1-4) edge node[above] {\footnotesize$GFf_1$} (m-1-5)
(m-2-2) edge (m-2-3) (m-2-3) edge (m-2-4) (m-2-4) edge node[below] {\footnotesize$f_1$} (m-2-5)
;
\path[>->]
(m-1-1) edge node[above] {\footnotesize$GFf_m$} (m-1-2)
(m-2-1) edge node[below] {\footnotesize$f_m$} (m-2-2)
;
\path[->>]
(m-1-5) edge (m-1-6)
(m-2-5) edge  (m-2-6)
;
\path[-,font=\scriptsize]
(m-1-6) edge [double, thick, double distance=2pt] (m-2-6)
;
\end{tikzpicture}
\]
Then, we have the following commutative diagram with exact rows:
\[
\begin{tikzpicture}[description/.style={fill=white,inner sep=2pt}] 
\matrix (m) [matrix of math nodes, row sep=2.3em, column sep=3em, text height=1.25ex, text depth=0.25ex] 
{ 
FW_m & FW_{m-1} & \cdots & FW_1 & FW_0 & FGD \\
FW_m & FW_{m-1} & \cdots & FW_1 & FW_0 & D \\
}; 
\path[->] 
(m-1-2) edge (m-1-3) (m-1-3) edge (m-1-4) (m-1-4) edge node[above] {\footnotesize$Ff_1$} (m-1-5)
(m-2-2) edge (m-2-3) (m-2-3) edge (m-2-4) (m-2-4) edge node[below] {\footnotesize$Ff_1$} (m-2-5)
(m-1-6) edge (m-2-6)
;
\path[>->]
(m-1-1) edge node[above] {\footnotesize$Ff_m$} (m-1-2)
(m-2-1) edge node[below] {\footnotesize$Ff_m$} (m-2-2)
;
\path[->>]
(m-1-5) edge (m-1-6)
(m-2-5) edge (m-2-6)
;
\path[-,font=\scriptsize]
(m-1-1) edge [double, thick, double distance=2pt] (m-2-1)
(m-1-2) edge [double, thick, double distance=2pt] (m-2-2)
(m-1-4) edge [double, thick, double distance=2pt] (m-2-4)
(m-1-5) edge [double, thick, double distance=2pt] (m-2-5)
;
\end{tikzpicture}
\]
It follows that $D \simeq FGD$, where $GD \in \omega^\wedge_m$, and hence $D \in F(\omega^\wedge_m)$.
\end{proof}

\begin{proposition}
Suppose either $\mathfrak{D}$ has enough projective objects or $\C$ has enough injective objects, that $G$ is left exact, and $F$ is right exact, full and faithful. If $(F(\X), F (\omega))$ is a Frobenius pair in $\mathfrak{D}$ which is left $[F(\omega)]^\wedge_m$-complete in $[F(\X)]^{\perp_1}$, then $(\X ,\omega)$ is a Frobenius pair in $\C$ which is left $\omega^\wedge_m$-complete in $\X^{\perp_1}$. 
\end{proposition}

\begin{proof}
The fact that $(\X,\omega)$ is a Frobenius pair follows from Proposition \ref{Funtor-Frob}. Now let $Y \in \X^{\perp_1}$. We show that $FY \in [F(\X)]^{\perp_1}$, so consider an object $FX$ with $X \in \X$. We have that $\Ext^1_{\mathfrak{D}}(FX,FY) \cong \Ext^1_{\C}(GFX,Y) \cong \Ext^1_{\C}(X,Y) = 0$. Then, $FY \in [F(\X)]^{\perp_1}$, and so there is an exact sequence $D' \rightarrowtail D \twoheadrightarrow FY$ with $D \in [F(\omega)]^\wedge_m$ and $D' \in ([F(\omega)]^\wedge_m)^{\perp_1}$. By Lemma \ref{lem:adjunction}, we have that $[F(\omega)]^\wedge_m = F(\omega^\wedge_m)$ and so $D \simeq FL$ for some $L \in \omega^\wedge_m$. Thus, we have the exact sequence 
\[
D' \rightarrowtail FL \twoheadrightarrow FY
\] 
where $FL \twoheadrightarrow FY = Fg$ for some morphism $g \colon L \to Y$. Now consider the following commutative diagram with exact rows and where the vertical arrows are isomorphisms:
\[
\begin{tikzpicture}[description/.style={fill=white,inner sep=2pt}] 
\matrix (m) [matrix of math nodes, row sep=2.3em, column sep=2.3em, text height=1.25ex, text depth=0.25ex] 
{ 
GD' & GFL & GFY \\
GD' & L & Y \\
}; 
\path[->] 
(m-1-2) edge node[left] {\footnotesize$\varepsilon_L$} (m-2-2) (m-1-3) edge node[left] {\footnotesize$\varepsilon_Y$} (m-2-3) 
;
\path[>->]
(m-1-1) edge (m-1-2)
(m-2-1) edge (m-2-2)
;
\path[->>]
(m-1-2) edge node[above] {\footnotesize$GFg$} (m-1-3)
(m-2-2) edge node[below] {\footnotesize$g$} (m-2-3)
;
\path[-,font=\scriptsize]
(m-1-1) edge [double, thick, double distance=2pt] (m-2-1)
;
\end{tikzpicture}
\]
It is only left to show that $GD' \in [\omega^\wedge_m]^{\perp_1}$. As in the last part of the previous proof, we can note that $FGD' \simeq D' \in ([F(\omega)]^\wedge_m)^{\perp_1} = [F(\omega^\wedge_m)]^{\perp_1}$. Now every $L' \in \omega^\wedge_m$ we have that
\[
\Ext^1_{\C}(L',GD') \cong \Ext^1_{\C}(GFL',GD') \cong \Ext^1_{\mathfrak{D}}(FL',FGD') \cong \Ext^1_{\mathfrak{D}}(FL',D') = 0,
\]
where the last equality holds since $FL' \in F(\omega^\wedge_m)$ and $D' \in [F(\omega^\wedge_m)]^{\perp_1}$.
\end{proof}

%%%%%%%%%%%%%%%%%%%%%%%%%%%%%%%%%%%%%
%%%%%%%%%%%%%%%%%%%%%%%%%%%%%%%%%%%%%
%%%%%%%%%%%%%%%%%%%%%%%%%%%%%%%%%%%%%
%%%%%%%%%%%%%%%%%%%%%%%%%%%%%%%%%%%%%

\section*{Declaration of AI usage}

No form of AI was used during the preparation of this paper, nor in the research that it presents.

%%%%%%%%%%%%%%%%%%%%%%%%%%%%%%%%%%%%%
%%%%%%%%%%%%%%%%%%%%%%%%%%%%%%%%%%%%%
%%%%%%%%%%%%%%%%%%%%%%%%%%%%%%%%%%%%%
%%%%%%%%%%%%%%%%%%%%%%%%%%%%%%%%%%%%%

\section*{Funding}

The first named author was fully supported by a SECIHTI posdoctoral fellowship at Centro de Ciencias Matem\'aticas, UNAM. The second named was supported by the following institutions:  ANII - Agencia Nacional de Investigación e Innovación and PEDECIBA - Programa de Desarrollo de las Ciencias Básica. 

Both authors thank the financial support from the  projects “IN100124 PAPIIT-UNAM” and “Despegue Cient\'ifico 2023” from PEDECIBA. \\

%\textbf{Acknowledgements} The author thanks to his childs without this work could be finished more quickly.

\bibliographystyle{alpha}
\bibliography{biblioFP}

@incollection {AB89,
    AUTHOR = {Auslander, M. and Buchweitz, R.-).},
     TITLE = {The homological theory of maximal {C}ohen-{M}acaulay
              approximations},
      NOTE = {Colloque en l'honneur de Pierre Samuel (Orsay, 1987)},
   JOURNAL = {M\'em. Soc. Math. France (N.S.)},
  FJOURNAL = {M\'emoires de la Soci\'et\'e{} Math\'ematique de France.
              Nouvelle S\'erie},
    NUMBER = {38},
      YEAR = {1989},
     PAGES = {5--37},
      ISSN = {0037-9484},
   MRCLASS = {13C14 (13H10)},
  MRNUMBER = {1044344},
MRREVIEWER = {Dieter\ Happel},
}

@article {AkinciAlizade,
    AUTHOR = {Akinci, K. D. and Alizade, R.},
     TITLE = {Special precovers in cotorsion theories},
   JOURNAL = {Proc. Edinb. Math. Soc. (2)},
  FJOURNAL = {Proceedings of the Edinburgh Mathematical Society. Series II},
    VOLUME = {45},
      YEAR = {2002},
    NUMBER = {2},
     PAGES = {411--420},
      ISSN = {0013-0915,1464-3839},
   MRCLASS = {16S90},
  MRNUMBER = {1912649},
MRREVIEWER = {Stephen\ Tempest\ Aldrich},
       DOI = {10.1017/S001309150000122X},
       URL = {https://doi.org/10.1017/S001309150000122X},
}

@article {ChristensenEstradaThompson-stable,
    AUTHOR = {Christensen, L. W. and Estrada, S. and Thompson,
              P.},
     TITLE = {The stable category of {G}orenstein flat sheaves on a
              noetherian scheme},
   JOURNAL = {Proc. Amer. Math. Soc.},
  FJOURNAL = {Proceedings of the American Mathematical Society},
    VOLUME = {149},
      YEAR = {2021},
    NUMBER = {2},
     PAGES = {525--538},
      ISSN = {0002-9939,1088-6826},
   MRCLASS = {14F08 (18G35)},
  MRNUMBER = {4198062},
MRREVIEWER = {Inder\ Kaur},
       DOI = {10.1090/proc/15258},
       URL = {https://doi.org/10.1090/proc/15258},
}

@article{WeiWangLiu-DingModels,
 author = {Wei, C. Q. and Wang, L. M. and Liu, Z. K.},
 title = {Abelian model structures and {Ding} homological dimensions},
 fjournal = {Hacettepe Journal of Mathematics and Statistics},
 journal = {Hacet. J. Math. Stat.},
 issn = {1303-5010},
 volume = {45},
 number = {5},
 pages = {1461--1474},
 year = {2016},
 language = {English},
 doi = {10.15672/HJMS.2015449425},
 zbMATH = {6729928},
 Zbl = {1366.16005}
}

@article{Fieldhouse,
 author = {Fieldhouse, D. J.},
 title = {Character modules, dimension and purity},
 fjournal = {Glasgow Mathematical Journal},
 journal = {Glasg. Math. J.},
 issn = {0017-0895},
 volume = {13},
 pages = {144--146},
 year = {1972},
 language = {English},
 doi = {10.1017/S0017089500001567},
 zbMATH = {3407871},
 Zbl = {0258.16028}
}

@article{Gillespie-DP,
 author = {Gillespie, J.},
 title = {Duality pairs and stable module categories},
 fjournal = {Journal of Pure and Applied Algebra},
 journal = {J. Pure Appl. Algebra},
 issn = {0022-4049},
 volume = {223},
 number = {8},
 pages = {3425--3435},
 year = {2019},
 language = {English},
 doi = {10.1016/j.jpaa.2018.11.010},
 zbMATH = {7047903},
 Zbl = {1409.13034}
}

@article {GillespieAdvances,
    AUTHOR = {Gillespie, J.},
     TITLE = {Gorenstein complexes and recollements from cotorsion pairs},
   JOURNAL = {Adv. Math.},
  FJOURNAL = {Advances in Mathematics},
    VOLUME = {291},
      YEAR = {2016},
     PAGES = {859--911},
      ISSN = {0001-8708,1090-2082},
   MRCLASS = {18E30 (16E05)},
  MRNUMBER = {3459032},
MRREVIEWER = {Yuehui\ Zhang},
       DOI = {10.1016/j.aim.2016.01.004},
       URL = {https://doi.org/10.1016/j.aim.2016.01.004},
}

@article{GillespieHereditary,
 author = {Gillespie, J.},
 title = {Hereditary abelian model categories},
 fjournal = {Bulletin of the London Mathematical Society},
 journal = {Bull. Lond. Math. Soc.},
 issn = {0024-6093},
 volume = {48},
 number = {6},
 pages = {895--922},
 year = {2016},
 language = {English},
 doi = {10.1112/blms/bdw051},
 zbMATH = {6708763},
 Zbl = {1372.18001}
}

@article{ElMaaouy26,
 author = {El Maaouy, R.},
 title = {Homological algebra in categories of representations of quivers via duality pairs},
 fjournal = {Journal of Algebra},
 journal = {J. Algebra},
 issn = {0021-8693},
 volume = {708},
 pages = {91--137},
 year = {2026},
 language = {English},
 doi = {10.1016/j.jalgebra.2026.05.029},
 zbMATH = {8276229}
}

@misc{ElMaaouyOuberka,
 author = {El Maaouy, R. and Ouberka, H.},
 title = {Approximations and {Hovey} triples by objects of finite homological dimensions: {Applications} to sheaves},
 year = {2026},
 howpublished = {Preprint, {arXiv}:2604.25523 [math.{AG}] (2026)},
 url = {https://arxiv.org/abs/2604.25523},
 arXiv = {arXiv:2604.25523}
}

@article{LiangMaYang-contraejemplo,
 author = {Liang, L. and Ma, Y. J. and Yang, G.},
 title = {Frobenius pairs arising from {Morita} rings},
 fjournal = {Journal of Algebra and its Applications},
 journal = {J. Algebra Appl.},
 issn = {0219-4988},
 volume = {25},
 number = {3},
 pages = {22},
 note = {Id/No 2550337},
 year = {2026},
 language = {English},
 doi = {10.1142/S0219498825503372},
 zbMATH = {8142552}
}

@book{BeligiannisReiten,
 author = {Beligiannis, A. and Reiten, I.},
 title = {Homological and homotopical aspects of torsion theories},
 fseries = {Memoirs of the American Mathematical Society},
 series = {Mem. Am. Math. Soc.},
 issn = {0065-9266},
 volume = {883},
 isbn = {978-0-8218-3996-6; 978-1-4704-0487-1},
 year = {2007},
 publisher = {Providence, RI: American Mathematical Society (AMS)},
 language = {English},
 doi = {10.1090/memo/0883},
 zbMATH = {5179140},
 Zbl = {1124.18005}
}

@book{GRcomplexes,
 author = {Garc{\'{\i}}a Rozas, J. R.},
 title = {Covers and envelopes in the category of complexes of modules},
 fseries = {Chapman \& Hall/CRC Research Notes in Mathematics},
 series = {Chapman Hall/CRC Res. Notes Math.},
 volume = {407},
 isbn = {1-58488-004-X},
 year = {1999},
 publisher = {Boca Raton, FL: Chapman {and} Hall/CRC},
 language = {English},
 zbMATH = {1313601},
 Zbl = {0922.16001}
}

@article{LiangYangConstructions,
 author = {Liang, L. and Yang, G.},
 title = {Constructions of {Frobenius} pairs in abelian categories},
 fjournal = {Mediterranean Journal of Mathematics},
 journal = {Mediterr. J. Math.},
 issn = {1660-5446},
 volume = {19},
 number = {2},
 pages = {17},
 note = {Id/No 76},
 year = {2022},
 language = {English},
 doi = {10.1007/s00009-022-01999-3},
 zbMATH = {7488632},
 Zbl = {1482.18006}
}

@article {BMSP,
    AUTHOR = {Becerril, V. and Mendoza, O. and P\'erez, M. A.
              and Santiago, Valente},
     TITLE = {Frobenius pairs in abelian categories. {C}orrespondences with
              cotorsion pairs, exact model categories, and
              {A}uslander-{B}uchweitz contexts},
   JOURNAL = {J. Homotopy Relat. Struct.},
  FJOURNAL = {Journal of Homotopy and Related Structures},
    VOLUME = {14},
      YEAR = {2019},
    NUMBER = {1},
     PAGES = {1--50},
      ISSN = {2193-8407,1512-2891},
   MRCLASS = {18G10 (16E10 18G20 18G25 18N40)},
  MRNUMBER = {3913970},
MRREVIEWER = {Panyue\ Zhou},
       DOI = {10.1007/s40062-018-0208-4},
       URL = {https://doi.org/10.1007/s40062-018-0208-4},
}

@article {BMS,
    AUTHOR = {Becerril, V. and Mendoza, O. and Santiago, V.},
     TITLE = {Relative {G}orenstein objects in abelian categories},
   JOURNAL = {Comm. Algebra},
  FJOURNAL = {Communications in Algebra},
    VOLUME = {49},
      YEAR = {2021},
    NUMBER = {1},
     PAGES = {352--402},
      ISSN = {0092-7872,1532-4125},
   MRCLASS = {18G10 (16E10 18G20 18G25)},
  MRNUMBER = {4193637},
MRREVIEWER = {Xi\ Tang},
       DOI = {10.1080/00927872.2020.1800023},
       URL = {https://doi.org/10.1080/00927872.2020.1800023},
}

@misc{Cortes,
 author = {Cort{\'e}s-Izurdiaga, M. and {\v{S}}aroch, J.},
 title = {The cotorsion pair generated by the {Gorenstein} projective modules and ${{\lambda}}$-pure-injective modules},
   JOURNAL = {Isr. J. Math.},
  FJOURNAL = {Israel Journal of Mathematics},
    VOLUME = {TBD},
      YEAR = {2025},
    NUMBER = {TBD},
     PAGES = {1--30},
      ISSN = {1565-8511,0021-2172},
       DOI = {https://doi.org/10.1007/s11856-025-2826-y},
       URL = {https://link.springer.com/article/10.1007/s11856-025-2826-y},
}

@book {EnJen00,
    AUTHOR = {Enochs, E. E. and Jenda, O. M. G.},
     TITLE = {{Relative Homological Algebra}},
    SERIES = {De Gruyter Expositions in Mathematics},
    VOLUME = {30},
 PUBLISHER = {Walter de Gruyter \& Co., Berlin},
      YEAR = {2000},
     PAGES = {xii+339},
      ISBN = {3-11-016633-X},
   MRCLASS = {16E65 (13H10 16-02 16E10 18G25)},
  MRNUMBER = {1753146},
MRREVIEWER = {J.\ Kuzmanovich},
       DOI = {10.1515/9783110803662},
       URL = {https://doi.org/10.1515/9783110803662},
}

@book{EnJen00-2,
 author = {Enochs, E. E. and Jenda, O. M. G.},
 title = {{Relative Homological Algebra. {Vol}. 2}},
 edition = {2nd revised ed.},
 fseries = {De Gruyter Expositions in Mathematics},
 series = {De Gruyter Expo. Math.},
 issn = {0938-6572},
 volume = {54},
 isbn = {978-3-11-021522-9; 978-3-11-021523-6},
 year = {2011},
 publisher = {Berlin: Walter de Gruyter},
 language = {English},
 zbMATH = {5818169},
 Zbl = {1238.13002}
}

@article {Estrada24,
    AUTHOR = {Estrada, S. and Iacob, A.},
     TITLE = {Gorenstein projective precovers and finitely presented
              modules},
   JOURNAL = {Rocky Mountain J. Math.},
  FJOURNAL = {The Rocky Mountain Journal of Mathematics},
    VOLUME = {54},
      YEAR = {2024},
    NUMBER = {3},
     PAGES = {715--721},
      ISSN = {0035-7596,1945-3795},
   MRCLASS = {16E05 (16E10)},
  MRNUMBER = {4778282},
       DOI = {10.1216/rmj.2024.54.715},
       URL = {https://doi.org/10.1216/rmj.2024.54.715},
}

@article{DiEstradaLiangOdabasi,
 author = {Di, Z. and Estrada, S. and Liang, L. and Odaba{\c{s}}{\i}, S.},
 title = {Gorenstein flat representations of left rooted quivers},
 fjournal = {Journal of Algebra},
 journal = {J. Algebra},
 issn = {0021-8693},
 volume = {584},
 pages = {180--214},
 year = {2021},
 language = {English},
 doi = {10.1016/j.jalgebra.2021.05.008},
 zbMATH = {7363530},
 Zbl = {1467.18024}
}

@incollection {EstradaMarco,
    AUTHOR = {Estrada, S. and Iacob, A. and P\'erez, M. A.},
     TITLE = {Model structures and relative {G}orenstein flat modules and
              chain complexes},
 BOOKTITLE = {Categorical, homological and combinatorial methods in algebra},
    SERIES = {Contemp. Math.},
    VOLUME = {751},
     PAGES = {135--175},
 PUBLISHER = {Amer. Math. Soc., [Providence], RI},
      YEAR = {[2020] \copyright 2020},
      ISBN = {978-1-4704-4368-9},
   MRCLASS = {18G35 (13D30 16E65 18N40 55U15)},
  MRNUMBER = {4132088},
MRREVIEWER = {Liang\ Zhao},
       DOI = {10.1090/conm/751/15084},
       URL = {https://doi.org/10.1090/conm/751/15084},
}

@article {EIY17,
    AUTHOR = {Estrada, S. and Iacob, A. and Yeomans, K.},
     TITLE = {Gorenstein projective precovers},
   JOURNAL = {Mediterr. J. Math.},
  FJOURNAL = {Mediterranean Journal of Mathematics},
    VOLUME = {14},
      YEAR = {2017},
    NUMBER = {1},
     PAGES = {Paper No. 33, 10},
      ISSN = {1660-5446,1660-5454},
   MRCLASS = {18G25 (18G35)},
  MRNUMBER = {3598789},
MRREVIEWER = {Nan\ Gao},
       DOI = {10.1007/s00009-016-0822-5},
       URL = {https://doi.org/10.1007/s00009-016-0822-5},
}

@misc{Chains,
 author = {Gao, N. and Lu, X. S. and Zhang, P.},
 title = {Chains of model structures arising from modules of finite {Gorenstein} dimension},
 year = {2024},
 howpublished = {Preprint, {arXiv}:2403.05232 [math.{RT}] (2024)},
 url = {https://arxiv.org/abs/2403.05232},
 arXiv = {arXiv:2403.05232}
}

@article{GillespieFlat,
 author = {Gillespie, J.},
 title = {The flat model structure on {{\(\mathbf{Ch}(R)\)}}},
 fjournal = {Transactions of the American Mathematical Society},
 journal = {Trans. Am. Math. Soc.},
 issn = {0002-9947},
 volume = {356},
 number = {8},
 pages = {3369--3390},
 year = {2004},
 language = {English},
 doi = {10.1090/S0002-9947-04-03416-6},
 zbMATH = {2091204},
 Zbl = {1056.55011}
}

@article {GillespieExact,
    AUTHOR = {Gillespie, J.},
     TITLE = {Model structures on exact categories},
   JOURNAL = {J. Pure Appl. Algebra},
  FJOURNAL = {Journal of Pure and Applied Algebra},
    VOLUME = {215},
      YEAR = {2011},
    NUMBER = {12},
     PAGES = {2892--2902},
      ISSN = {0022-4049,1873-1376},
   MRCLASS = {18E10 (18G35 55U15 55U35)},
  MRNUMBER = {2811572},
MRREVIEWER = {Timothy\ Porter},
       DOI = {10.1016/j.jpaa.2011.04.010},
       URL = {https://doi.org/10.1016/j.jpaa.2011.04.010},
}

@book{HoveyBook,
 author = {Hovey, M.},
 title = {{Model Categories}},
 fseries = {{Mathematical Surveys and Monographs}},
 series = {Math. Surv. Monogr.},
 issn = {0076-5376},
 volume = {63},
 isbn = {0-8218-1359-5},
 year = {1999},
 publisher = {Providence, RI: American Mathematical Society},
 language = {English},
 zbMATH = {1226952},
 Zbl = {0909.55001}
}

@book {GT06,
    AUTHOR = {G\"obel, R. and Trlifaj, J.},
     TITLE = {Approximations and endomorphism algebras of modules},
    SERIES = {De Gruyter Expositions in Mathematics},
    VOLUME = {41},
 PUBLISHER = {Walter de Gruyter GmbH \& Co. KG, Berlin},
      YEAR = {2006},
     PAGES = {xxiv+640},
      ISBN = {978-3-11-011079-1; 3-11-011079-2},
   MRCLASS = {16D90 (03E75 16D10 16E30 16S50 16S90)},
  MRNUMBER = {2251271},
MRREVIEWER = {Sverre\ O.\ Smal\o},
       DOI = {10.1515/9783110199727},
       URL = {https://doi.org/10.1515/9783110199727},
}

@article {Holm04,
    AUTHOR = {Holm, H.},
     TITLE = {Gorenstein homological dimensions},
   JOURNAL = {J. Pure Appl. Algebra},
  FJOURNAL = {Journal of Pure and Applied Algebra},
    VOLUME = {189},
      YEAR = {2004},
    NUMBER = {1-3},
     PAGES = {167--193},
      ISSN = {0022-4049,1873-1376},
   MRCLASS = {16E10 (16E05 16E30)},
  MRNUMBER = {2038564},
MRREVIEWER = {Zhaoyong\ Huang},
       DOI = {10.1016/j.jpaa.2003.11.007},
       URL = {https://doi.org/10.1016/j.jpaa.2003.11.007},
}

@article {Hov02,
    AUTHOR = {Hovey, M.},
     TITLE = {Cotorsion pairs, model category structures, and representation
              theory},
   JOURNAL = {Math. Z.},
  FJOURNAL = {Mathematische Zeitschrift},
    VOLUME = {241},
      YEAR = {2002},
    NUMBER = {3},
     PAGES = {553--592},
      ISSN = {0025-5874,1432-1823},
   MRCLASS = {55U35 (18E30 18G55)},
  MRNUMBER = {1938704},
       DOI = {10.1007/s00209-002-0431-9},
       URL = {https://doi.org/10.1007/s00209-002-0431-9},
}

@article{AdamsRieffel,
 author = {Adams, W. W. and Rieffel, M. A.},
 title = {Adjoint functors and derived functors with an application to the cohomology of semigroups},
 fjournal = {Journal of Algebra},
 journal = {J. Algebra},
 issn = {0021-8693},
 volume = {7},
 pages = {25--34},
 year = {1967},
 language = {English},
 doi = {10.1016/0021-8693(67)90065-8},
 zbMATH = {3245493},
 Zbl = {0153.02901}
}

@book{BorceuxI,
 author = {Borceux, F.},
 title = {Handbook of categorical algebra. {Volume} 1: {Basic} category theory},
 edition = {Paperback reprint of the hardback edition 1994},
 fseries = {Encyclopedia of Mathematics and Its Applications},
 series = {Encycl. Math. Appl.},
 issn = {0953-4806},
 volume = {50},
 isbn = {978-0-521-06119-3},
 year = {2008},
 publisher = {Cambridge: Cambridge University Press},
 language = {English},
 zbMATH = {5273316},
 Zbl = {1143.18001}
}

@article {HMP-n-cot,
    AUTHOR = {Huerta, M. and Mendoza, O. and P\'erez, M. A.},
     TITLE = {{$n$}-cotorsion pairs},
   JOURNAL = {J. Pure Appl. Algebra},
  FJOURNAL = {Journal of Pure and Applied Algebra},
    VOLUME = {225},
      YEAR = {2021},
    NUMBER = {5},
     PAGES = {Paper No. 106556, 34},
      ISSN = {0022-4049,1873-1376},
   MRCLASS = {18G25 (18G10 18G20 18G35)},
  MRNUMBER = {4161404},
MRREVIEWER = {Jiangsheng\ Hu},
       DOI = {10.1016/j.jpaa.2020.106556},
       URL = {https://doi.org/10.1016/j.jpaa.2020.106556},
}

@article {HMP-cut,
    AUTHOR = {Huerta, M. and Mendoza, O. and P\'erez, M. A.},
     TITLE = {Cut cotorsion pairs},
   JOURNAL = {Glasg. Math. J.},
  FJOURNAL = {Glasgow Mathematical Journal},
    VOLUME = {64},
      YEAR = {2022},
    NUMBER = {3},
     PAGES = {548--585},
      ISSN = {0017-0895,1469-509X},
   MRCLASS = {18G25 (16E65)},
  MRNUMBER = {4462378},
MRREVIEWER = {Yunxia\ Li},
       DOI = {10.1017/S0017089521000367},
       URL = {https://doi.org/10.1017/S0017089521000367},
}

@article {MaoDing,
    AUTHOR = {Mao, L. X. and Ding, N. Q.},
     TITLE = {Envelopes and covers by modules of finite {FP}-injective and
              flat dimensions},
   JOURNAL = {Comm. Algebra},
  FJOURNAL = {Communications in Algebra},
    VOLUME = {35},
      YEAR = {2007},
    NUMBER = {3},
     PAGES = {833--849},
      ISSN = {0092-7872,1532-4125},
   MRCLASS = {16E10 (16D40 16D50)},
  MRNUMBER = {2305235},
MRREVIEWER = {Javad\ Asadollahi},
       DOI = {10.1080/00927870601115757},
       URL = {https://doi.org/10.1080/00927870601115757},
}

@article {Rachid,
    AUTHOR = {El Maaouy, R.},
     TITLE = {Model structures, {$n$}-{G}orenstein flat modules and {PGF}
              dimensions},
   JOURNAL = {Proc. Edinb. Math. Soc. (2)},
  FJOURNAL = {Proceedings of the Edinburgh Mathematical Society. Series II},
    VOLUME = {67},
      YEAR = {2024},
    NUMBER = {4},
     PAGES = {1241--1264},
      ISSN = {0013-0915,1464-3839},
   MRCLASS = {18N40 (16E10 16E65 18G65 18G80)},
  MRNUMBER = {4832996},
MRREVIEWER = {Li\ Liang},
       DOI = {10.1017/S0013091524000518},
       URL = {https://doi.org/10.1017/S0013091524000518},
}

@article{Alina2020,
 author = {Iacob, A.},
 title = {Projectively coresolved {Gorenstein} flat and {Ding} projective modules},
 fjournal = {Communications in Algebra},
 journal = {Commun. Algebra},
 issn = {0092-7872},
 volume = {48},
 number = {7},
 pages = {2883--2893},
 year = {2020},
 language = {English},
 doi = {10.1080/00927872.2020.1723612},
 zbMATH = {7243434},
 Zbl = {1479.16004}
}

@misc{BravoGillespieHovey-arxiv,
      title={The stable module category of a general ring}, 
      author={D. Bravo and J. Gillespie and M. Hovey},
      year={2014},
      eprint={1405.5768},
      archivePrefix={arXiv},
      primaryClass={math.RA},
      url={https://arxiv.org/abs/1405.5768}, 
}

@article {Saroch,
    AUTHOR = {{\v{S}}aroch, J. and {\v{S}}\v{t}ov\'i\v{c}ek, J.},
     TITLE = {Singular compactness and definability for {$\Sigma$}-cotorsion
              and {G}orenstein modules},
   JOURNAL = {Selecta Math. (N.S.)},
  FJOURNAL = {Selecta Mathematica. New Series},
    VOLUME = {26},
      YEAR = {2020},
    NUMBER = {2},
     PAGES = {Paper No. 23, 40},
      ISSN = {1022-1824,1420-9020},
   MRCLASS = {16E30 (03E75 16B70)},
  MRNUMBER = {4076700},
MRREVIEWER = {Ramalingam\ Udhayakumar},
       DOI = {10.1007/s00029-020-0543-2},
       URL = {https://doi.org/10.1007/s00029-020-0543-2},
}

@incollection {Stovi14,
    AUTHOR = {\v{S}\v{t}ov\'i\v{c}ek, J.},
     TITLE = {Exact model categories, approximation theory, and cohomology
              of quasi-coherent sheaves},
 BOOKTITLE = {Advances in representation theory of algebras},
    SERIES = {EMS Ser. Congr. Rep.},
     PAGES = {297--367},
 PUBLISHER = {Eur. Math. Soc., Z\"urich},
      YEAR = {2013},
      ISBN = {978-3-03719-125-5},
   MRCLASS = {18E10 (18E30 18F20)},
  MRNUMBER = {3220541},
MRREVIEWER = {R.\ H.\ Street},
}

\end{document}